\documentclass[10pt]{article} %{amsart}

\usepackage[usenames,dvipsnames]{xcolor}

\RequirePackage[OT1]{fontenc}
\usepackage{underscore}
\usepackage[portrait,margin=3.5cm]{geometry}
\usepackage{mathrsfs}
\usepackage[colorlinks,citecolor=blue,urlcolor=blue,linkcolor=blue,linktocpage=true]{hyperref}
\usepackage{amssymb,amsthm,amsfonts,amsbsy,latexsym}%,dsfont}
\usepackage[textsize=small]{todonotes}       % includes TO DO LIST
\usepackage{graphicx}
\usepackage[english]{babel}
\usepackage{subfigure}
\usepackage{appendix}

\usepackage{enumerate}
\newenvironment{enumeratei}{\begin{enumerate}[\upshape (i)]}{\end{enumerate}}

\usepackage{color}              % Need the color package
\usepackage{graphicx}
\usepackage[]{amsmath}
\usepackage{amsthm}
\usepackage{amssymb}
\usepackage{hyperref}
\usepackage{comment}
\usepackage{microtype}

\usepackage[numbers]{natbib}

\definecolor{MyDarkblue}{rgb}{0,0.08,0.50}
\definecolor{Brickred}{rgb}{0.65,0.08,0}

\newtheorem{theorem}{Theorem}[section]
\newtheorem{lemma}[theorem]{Lemma}
\newtheorem{proposition}[theorem]{Proposition}
\newtheorem{corollary}[theorem]{Corollary}

\newtheorem{definition}[theorem]{Definition}
\newtheorem{notation}[theorem]{Notation}

\newtheorem{remark}[theorem]{Remark}

\newtheorem{example}[theorem]{Example}

\newcommand{\E}[1]{\mathbb{E}\left[#1\right]}
\newcommand{\Ev}{\mathbb{E}}

\newcommand{\set}[1]{\left\{#1\right\}}

\newcommand{\s}[2]{\sum_{#1}^{#2}}

\newcommand{\eqn}[1]{\begin{equation} #1 \end{equation}}

\newcommand{\e}{{\mathrm e}}

\numberwithin{equation}{section}
\newcommand{\N}{\mathbb{N}}

\newcommand{\indicator}[1]{1_{\set{#1}}}
\newcommand{\equalsd}{\overset{d}{=}}

\newcommand*{\be}{\begin{equation}}
\newcommand*{\ee}{\end{equation}}
\newcommand*{\ba}{\begin{aligned}}
\newcommand*{\ea}{\end{aligned}}
\newcommand*{\barr}{\begin{array}{c}}
\newcommand*{\earr}{\end{array}}

\def\namedlabel#1#2{\begingroup
    #2%
    \def\@currentlabel{#2}%
    \phantomsection\label{#1}\endgroup
}

\newcommand{\argmin}[1]{\underset{#1}{\operatorname{arg}\,\operatorname{min}}\;}
\newcommand{\minun}[1]{\underset{#1}{\operatorname{min}}\;}
\newcommand{\Dn}{D_n}
\newcommand{\Dne}{D_n^e}
\renewcommand{\P}[1]{\mathbb{P}\left(#1\right)}
\newcommand{\leqd}{\overset{d}{\leq}}
\newcommand{\geqd}{\overset{d}{\geq}}
\newcommand{\HG}{(h,G)}
\newcommand{\HGa}{(h_{\alpha},G)}
\newcommand{\HaG}{(h_{\alpha},G)}
\newcommand{\HaLG}{(h^L_{\alpha},G)}
\newcommand{\Ha}{h_{\alpha}}
\newcommand{\HaL}{h_{\alpha}^L}
\newcommand{\HGCf}{(h,G,C)^{f}}
\newcommand{\HaGCf}{(h_{\alpha},G,C)^{f}}
\newcommand{\HaLGCf}{(h^L_{\alpha},G,C)^{f}}
\newcommand{\lr}[1]{\left(#1\right)}

\newcommand{\HGIf}{(h,G,I)_{f}}

\newcommand{\HaLGIf}{(h^L_{\alpha},G,I)_{f}}
\newcommand{\HaLGIb}{(h^L_{\alpha},G,I)_{b}}

\newcommand{\HGCb}{(h,G,C)^{b}}

\newcommand{\TIf}{T_f^{(h,G,I)}}
\newcommand{\TIb}{T_b^{(h,G,I)}}
\newcommand{\TG}{T_{(h,G)}}
\newcommand{\TCf}{ T^f_{(h,G,C)}}
\newcommand{\TCb}{ T^b_{(h,G,C)}}

\newcommand{\HGIb}{(h,G,I)_{b}}
\newcommand{\HaGIb}{(h_{\alpha},G,I)_b}

\newcommand{\ADBPs}{age-dependent branching processes}

\newcommand{\BPs}{branching processes}

\begin{document}

\begin{titlepage}
\newcommand{\HRule}{\rule{\linewidth}{0.75mm}}
\center

{TECHNISCHE UNIVERSITEIT EINDHOVEN \\
Department of Mathematics and Computer Science}\\ [0.4cm]

\HRule \\[0.5cm]
{\huge \bfseries \rmfamily Explosiveness of Age-Dependent Branching Processes with Contagious and Incubation Periods}\\[0.4cm]
\HRule \\[2.5cm]
{\Large \bfseries \sffamily Master's Thesis}\\[2.5cm]

{\LARGE \bfseries \rmfamily Lennart Gulikers}\\
{\large November 3, 2014}\\[8 cm]

\begin{flushleft}
\emph{Supervisors:} \\
Dr. J\'{u}lia Komj\'{a}thy \\
Prof. Dr. Remco van der Hofstad
\end{flushleft}

\vfill
\end{titlepage}
\newpage\null\thispagestyle{empty}\newpage

\section*{Abstract}
We study explosiveness of various age-dependent branching processes capable of describing the early stages of an epidemic-spread. The spread is considered from a twofold perspective: propagating from an individual through the population, the so-called forward process, and its infection trail towards an individual, the backward process. Our analysis leans strongly on two techniques: an analytical study of fixed point equations and a probabilistic method based on stochastic domination arguments.
\\
For the classical age-dependent branching process $\HG$, where the offspring has probability generating function $h$ and all individuals have life-lengths independently picked from a distribution $G$, we focus on the setting $h = \HaL$, with $L$ a function varying slowly at infinity and $\alpha \in (0,1)$. Here, $h^L_{\alpha}(s) = 1 - (1-s)^{\alpha} L(\frac{1}{1-s}),$ as $s \to 1$. The first main result in this thesis is that for a fixed $G$, the process $\HaLG$ explodes either for all $\alpha \in (0,1)$ or for no $\alpha \in (0,1)$, regardless of $L$. 
\\
We then add contagious periods to all individuals and let their offspring survive only if their life-length is smaller than the contagious period of their mother, hereby constituting a forward process. It turns out that an explosive process $\HaLG$, as above, stays explosive when adding a non-zero contagious period. This is the second main result in the underlying thesis. 
\\
We extend this setting to backward processes with contagious periods.
\\
Further, we consider processes with incubation periods during which an individual has already contracted the disease but is not able \emph{yet} to infect her acquaintances.  We let these incubation periods follow a distribution $I$. In the forward process $\HaLGIf$, every individual possesses an incubation period and only her offspring with life-time larger than this period survives. In the backward process $\HaLGIb$, individuals survive only if their life-time exceeds their \emph{own} incubation period. These two processes are the content of the third main result that we establish: under a mild condition on $G$ and $I$, explosiveness of both $\HG$ and $(h,I)$ is necessary and \emph{sufficient} for processes $\HaLGIf$ and $\HaLGIb$ to explode.
\\ 
We obtain the fourth main result by comparing forward to backward processes: the explosion time in the former stochastically dominates the explosion time in the latter. 
 
\newpage\null\thispagestyle{empty}\newpage

\vspace*{5 cm}
\noindent { \large \sffamily ``Whatever an education is, it should make you a unique individual, not a conformist; it should furnish you with an original spirit with which to tackle the big challenges; it should allow you to find values which will be your road map through life; it should make you spiritually rich, a person who loves whatever you are doing, wherever you are, whomever you are with; it should teach you what is important: how to live and how to die.''\\ }
\begin{flushright}
John Taylor Gatto, \textit{Dumbing Us Down}\footnote{Quotation taken from \cite{Gat92}.}
\end{flushright}

\newpage\null\thispagestyle{empty}\newpage

 \tableofcontents 

\newpage\null\thispagestyle{empty}\newpage

\section*{Acknowledgements}
The model in this thesis and part of the problem formulation are due to my supervisors Julia Komj\'{a}thy and Remco van der Hofdstad. In particular, ideas presented in \citep{Bh14} and references therein. Further, van der Hofstad's book \citep{Ho14} has been an indispensable source.\vspace{0.3cm}
\\Thanks are due to Onno Boxma, who introduced me to research in probability theory in the context of a honors project.
\\
I wish to thank Sem Borst for supervising me during another honors project. 
\\
I offer my sincere gratitude to my supervisor Remco van der Hofdstad for our fruitful discussions. I am also indebted to him for funding my visit to Vancouver.
\\
I would like to thank J\'{u}lia Komj\'{a}thy for being my supervisor.  I thank her and her husband Tim Hulshof for their warm support during my stay in Vancouver.
\\
Finally, I want to thank my father.

\newpage\null\thispagestyle{empty}\newpage

\section*{Introduction}
\addcontentsline{toc}{section}{Introduction}

In this thesis we model and analyse the  \emph{early} stages of an epidemic spread by studying age-dependent branching processes. 
\\
We shall first present an overview on age-dependent processes when the number of offspring of any given individual has infinite expectation. Having defined these processes, we introduce the model of epidemic spreads that is of our interest and point out a natural connection to  age-dependent branching processes. We then motivate the need for studying an infinite mean in those branching processes by considering random graph models as an appropriate tool for the study of epidemics on a \emph{finite} population.
\\
Because age-dependent branching processes are extensively studied throughout literature, we shall not try to give a complete overview of the field. Rather, we mention the papers that are key to the investigation in the underlying thesis. 
\\
Grey \cite{Grey74} wrote in 1974 one of the earlier papers on age-dependent branching processes. He considered the \emph{explosiveness} problem of the  following Bellman-Harris process. An initial ancestor has a random life-time and upon death she produces a random number of offspring with  probability generating function $h$. Their life story (life length and number of children they give birth to) is independent and has the same distribution as that of the initial ancestor, and so on for their children. All life-lengths are picked from some (cumulative) distribution function $G$. All individuals in this process, that we shall denote by $\HG$, reproduce independently of one another. Explosiveness is now the phenomenon that an infinite amount of individuals are born in a finite amount of time. Grey \cite{Grey74} uses an analytical approach to characterize explosiveness: he derives a fixed point equation for the generating function of the total size of the alive population at any time $t$, and studies the number as well as behaviour of its solutions. 
He concludes his paper by presenting some examples; most notable, a process where the offspring  probability generating function $h$ corresponds to a random variable $D$ that satisfies a power law, that is there are $\alpha \in (0,1)$, $c, C > 0$ and a function $\widehat{L}$ varying slowly at infinity such that
\be \frac{c \widehat{L}(x)}{x^{\alpha }} \leq \P{D \geq x} \leq \frac{C \widehat{L}(x) }{x^{\alpha }}, \label{eq::power} \ee
for large $x$.
Roughly speaking, whether the process $\HG$ is explosive or not depends on the behaviour of $G$ (i.e., it's flatness) around the origin. Keeping this in mind, Grey makes a successful attempt to localize the borderline in terms of life-length distribution functions that together with $h$ are just not flat enough to prevent the process from exploding. He puts forward two distribution functions that are extremely flat around the origin: one such that the process is explosive and another which ensures that the alive population is finite at all time (i.e., a conservative or non-explosive process). We shall improve on this boundary and demonstrate that, in fact, explosiveness of those processes is independent of the particular choice of $\alpha \in (0,1)$: for a fixed $G$ the process explodes either for all $\alpha$ or is conservative for all choices of $\alpha$. 
\\
It is worth noting that Markovian age-dependent branching processes, that is processes with exponential life-times, are well-understood, see for instance \cite{Har63}. Such a process is explosive if and only if $h'(1)=\infty$ (i.e., infinite expectation for the offspring) and
\[ \int_{1-\epsilon}^1 \frac{\mathrm d s}{s-h(s)} < \infty, \]
for suitably small $\epsilon > 0$.
\\
Since then, various attempts have been made to characterize explosive branching processes. For instance, Sevastyanov  \cite{Sev70} proved a necessary condition for explosiveness of age-dependent branching processes. He demonstrated that if there exist $\theta, \epsilon > 0$ such that
\[ \int_{0}^{\epsilon} G^{\leftarrow}\lr{\frac{1+\theta}{w(y)}} \frac{\mathrm d y}{y} < \infty, \]
where $w$ is defined for $y \geq 0$ by $y \cdot w(y) = 1 - h(1-y)$ and $G^{\leftarrow}$ is the inverse function of $G$, then the process $\HG$ is explosive. 
Inspired by this result, he suggests that the condition
\be \int_{0}^{\epsilon} G^{\leftarrow}\lr{\frac{1}{w(y)}} \frac{\mathrm d y}{y} = \infty, \label{eq::cond_sevastyanov} \ee
for all suitably small $\epsilon > 0$, might very well be necessary and sufficient for a process to be conservative. 
\\
In 1987, Vatutin \cite{Vat87} gives a counterexample that disproves necessity and sufficiency of \eqref{eq::cond_sevastyanov} for explosiveness. Remarkably, he presents a distribution function $G$ defined for small $t \geq 0$ by $G(t) = \e^{-1/t}$ as an example for which \eqref{eq::cond_sevastyanov} is satisfied for all probability generating functions $h$. However, he notes that there exists a $h$ such that the process $\HG$ is conservative and thereby violating the conjecture. 
\\
In 2013, Amini et al. presented in \cite{AmDeGrNe13} a necessary and sufficient condition for explosiveness covering a wide range of processes, namely those for which the number of offspring is an independent copy of some random variable $D$ that is plump, i.e., for some $\epsilon > 0$,
\begin{equation}
\mathbb{P}(D \geq m^{1+\epsilon}) \geq \frac{1}{m},
\end{equation}
for all sufficiently large $m$.
Their paper falls within the broader context of branching random walks, a quite different approach that we shall describe later on. 
They reasoned that in an explosive process one finds in almost every realization an exploding path: an infinite path with the property that the sum of the edge-lengths along it is finite. Obviously, the sum of life-lengths along every single path to infinity is always less than or equal to the sum over the minimal life-length in every generation. We call a process min-summable if the latter sum is almost surely finite, which is a necessary condition for explosion. Remarkably, although seemingly much weaker, they prove that this event is also sufficient for a plump process to explode. They provide a constructive proof: given that a process is min-summable, they construct an algorithm that almost surely finds an exploding path. This results in an easily verified if and only if condition that we recite in Theorem \ref{th::AGE_PLUMP}.
\\
To fully employ the last result, it is desirable to have estimates on the growth of generation sizes. Davies shows in \cite{Dav78} that for an ordinary branching process, its growth settles to an almost deterministic course after some initial random development. Loosely speaking, if the distribution $D$ behaves as in \eqref{eq::power} for some $\alpha \in (0,1)$, then upon denoting the size of generation $n$ by $Z_n$, we have for large $n$: $Z_n + 1 \simeq \e^{W/ \alpha^n}$, where $W$ is a random variable with exponential tail behaviour and $\P{W=0}$ equals the extinction probability of the process. We shall combine the aforementioned results in \cite{AmDeGrNe13} and \cite{Dav78} to extend the theory presented by Grey in \cite{Grey74}.
\\
We now briefly discuss non-negative branching random walks on $\mathbb{R}^+$ as they are described in \cite{AmDeGrNe13}. The process starts with one particle at the origin, that jumps to the right according to some displacement distribution function $G$, at which instant it gives birth to an offspring with probability generating function $h$. The process is then repeated: all particles in a particular generation independently  take a random-length jump to the right according to the distribution function $G$ and then give birth to the individuals in the next generation. We let $M_n$ be the distance from the origin of the leftmost particle in generation  $n \in \mathbb{N}$ and define $M_n = \infty$ if there are no particles in generation $n$. Explosion is the event that $\displaystyle \lim_{n \to \infty} M_n = V < \infty$ with positive probability, i.e., an infinite amount of particles stays within a distance $V < \infty$ from the origin. 
\\
Let us write $D$ for the number of offspring of a single individual with probability generating function $h$ in a non-negative branching random walk. It turns out that $\E{D} \in (1,\infty)$ ensures the existence of a constant $\gamma$ such that, conditional on survival,
\[ \frac{M_n}{n} \overset{\text{a.s.}}\to \gamma, \quad \quad \mbox{as } n \to \infty, \]
see Hammersley \cite{Ham74}, Kingman \cite{Kin75} and Biggins \cite{Big90}. Thus, $M_n = \gamma n + o(n)$ for large $n$, and, hence, explosiveness necessitates $\gamma$ to be $0$. 
\\
Consider now $H:=  \E{D} G(0)$, with $\E{D} \geq 1$ and note that there are three cases of interest: $H<1, H=1$  and $H>1$. Amini et al. infer that explosion does not occur for $H<1$. They additionally point out that, due to a theorem of Dekking and Horst \cite{DeHo91}, the random variables $(M_n)_n$ converge almost surely to a finite random variable $M$, whenever $H>1$. They ultimately reason that the case $\E{D} = \infty$  and $G(0)=0$ is the most interesting, and it is precisely this setting that we shall be principally concerned with. 
\\
We shall  explain shortly how \ADBPs $  $ may be employed to study epidemic spreads. 
They are interesting tools in their own right to describe the initial phase of an epidemic. Moreover, we will show that they naturally arise when studying epidemic spreads on a random graph. 
\\
But, first, we define the rules of infection that we have in mind (based on ideas in \citep{Bh14}). We consider a very large population of individuals subject to a disease and make the following assumptions regarding the transmission of the disease:
\begin{itemize}
\item an individual $i$ \emph{may} infect another individual $j$ at one and only one specific instant \emph{after} she herself has contracted the disease. We call this period the \emph{infection-time} $X(i,j)$. We shall assume that the variables $\lr{X(i,j)}_{i,j}$ are i.i.d. with a general continuous distribution function $G$;
\item an infectious contact results only in the infection of an individual if she has not been infected earlier; 
\item an individual $i$ may possibly infect $D_i$ others that we call her children or offspring;
\item infection might be prevented from happening because a person is only able to transmit the disease during a \emph{contagious period}. Thus, we give each node $i$ an independent weight $\tau^C(i) \sim C$;
\item each individual $i$ additionally has an \emph{incubation period} $\tau_I(i)$ that is independently chosen from some distribution function $I$. This time captures the period during which an individual is not \emph{yet} able to spread the disease;
\item all random variables are non-negative. 
\end{itemize}
From the preceding considerations it is obvious that individual $i $ infects $j$ if and only if
\be \tau_I(i) \leq X(i,j) \leq \tau^C(i). \label{eq::infection_rule} \ee
The connection to the above \ADBPs $  $ is made in the following way (again, see \citep{Bh14}):
We select one individual at random that we shall call the root, and then  artificially start the epidemic by designating the root to be infected. 
Every individual $i$ in the branching process gives birth to $D_i$ susceptible individuals. Each of those has a life-time that we identify with the infection-time introduced above. Each individual furthermore possesses both an incubation and contagious period and an infectious contact with each of her children does \emph{or} does not result in an infection according to the rule \eqref{eq::infection_rule}. During the initial spread of the epidemic, the number of susceptible individuals constitutes almost the total population and any alteration of it remains initially small in comparison with the total number of individuals. Hence, if the population does not exhibit an underlying structure, such as the presence of cliques of people, then the early stages of the epidemic are well approximated with an age-dependent branching process. 
\\
We now discuss some literature on mathematical models of epidemics. Already in the early fifties, scientists in the field used branching processes to model such spreads \cite{Whi55}. Those models often assumed the population to have very large or infinite size. It was in the late fifties and early sixties that random graphs were introduced \cite{Gil59}. These are powerful tools to model realistic networks where the population has a large, but finite size $n$.  The \ADBPs $  $ with \emph{infinite} mean will naturally occur in the particular random graph model that we have in mind. That model allows for edge-weights and has the property that the local neighbourhood of a vertex has a tree-like structure. 
\\
We shall now briefly introduce random graphs.
One could argue that the field was initiated by Erd\H{o}s and R\'enyi when they studied the random graph model that is named after them, \cite{ErRe59,ErRe60}. In their model, the graph has $n$ vertices and edges between any two vertices are independently present with probability $p$. 
\\
We discuss next the configuration model (\citep{Bh14}, \citep{Bo01} or \citep{Ho14}), in which the vertices follow a power-law degree sequence. Measurements from real-life networks show that in many of them, the number of vertices with degree $k$ falls of as an inverse power of $k$, \cite{BaAl99,EbHoMi02}, i.e., a power-law. In a follow-up paper we study the time it takes for the infection to spread from one uniformly picked person $U_n \in [n]$ to another $V_n \in [n]$ on the configuration model, $CM_n(\widehat{{\underline{d}}}_n)$. This is a random-graph on $n$ nodes with degrees given by $\widehat{{\underline{d}}}_n = (\widehat{{d}}_1,\ldots,\widehat{{d}}_n)$.
We use $[n] = \{1,\ldots,n\}$ to denote the collectino of nodes. To encode the degree of each node $i$, we  associate a total of $\widehat{{d}}_i$ half edges to it, where all $\widehat{{d}}_i$ are independent copies of some random variable $\widehat{D}$. The distribution of $\widehat{D}$ follows a power law, that is, there are constants $\widehat c_1,\widehat C_2  > 0$ and $\alpha \in (0,1)$ such that
\be \frac{\widehat c_1 }{x^{\alpha + 1}} \leq \P{\widehat{D} \geq x} \leq \frac{\widehat C_2 }{x^{\alpha + 1}}, \label{eq::intro_power} \ee
for large $x$. Let us denote the total amount of half-edges in the graph by $\mathcal{L}_n$ and increase one of the degrees by one in case $\mathcal{L}_n$ is odd. Henceforth, we may assume that $\mathcal{L}_n $ is an even number. To construct the graph, we join the half-edges uniformly at random to each other.  We give each half-edge (attached to vertex $i \in [n]$) a weight $X(i,j)$ indicating the time needed to infect the neighbour $j$ connecting to this half-edge. In addition we give each node $i$ a contagious period $\tau^C(i)$ and an incubation period $\tau_I(i)$. The infection spreads then according to the rule set out above \eqref{eq::infection_rule}. The main result of the follow-up paper of this thesis is that 
\be 
d(U_n, V_n) \overset{d} \to V_f + V_b \mbox{ as } n \to \infty,
\label{eq::INF_TIME}
\ee
where $d(u,v)$ indicates the time needed for $u \in [n]$ to infect $v \in [n]$ and
 $V_f$ and $V_b$ are \emph{explosion times} intrinsically coupled to \emph{forward}, (respectively) \emph{backward} branching process that we shall introduce shortly. 
These two types of \BPs $  $ are described in \cite{BaRe13} and \citep{Ho14}. The forward process describes how the epidemic spreads in its early stages from an individual \emph{forward} through the population. The backward process in its turn approximates the process leading to a potential infection of a randomly chosen individual, see below for further details. 
\\
During the initial phase, individuals with a high degree are more likely to be infected first. Thus, the infection, in its early stages, is biased to spread towards high-degreed vertices, resulting in a \emph{different} offspring distribution $D$. It can be shown (\citep{Ho14}) that for both processes the probability that an individual gives birth to an offspring of size $j$ is \emph{approximated} by
\be \P{D=j} = \frac{ (j+1) \P{\widehat{D} = j+1} }{\E{\widehat{D}}}.   \ee
It may be calculated from \eqref{eq::intro_power} that the offspring distribution satisfies again a power-law:
\be \frac{c_1}{x^{\alpha }} \leq \P{D \geq x} \leq \frac{C_2}{x^{\alpha }},\ee
for some $c_1, C_2 > 0$ and all sufficiently large $x$. Note that $\alpha$ is the \emph{same} as above, it is an intrinsic property of the model.
We emphasize that $\alpha \in (0,1)$ throughout this thesis and therefore $\widehat{D}$ has finite expectation, but $D$ has \emph{infinite} expectation. 
\\
Equation \eqref{eq::INF_TIME} clearly demonstrates the strong dependence of the epidemic spread on the underlying branching processes. The rest of this thesis therefore intends to be a more detailed study of various age-dependent processes. 
\\
Most, if not all, literature on age-dependent branching processes deals with a setting where individuals enjoy complete independence from one another.  In this sense, our model is different and we study it not in the least for its mathematical interest. We shall, however, restrict our attention to the case where there is \emph{either} a contagious period \emph{or} an incubation period, thereby excluding the case of both being mutually present. The \emph{main} question that this thesis aims to answer is formulated as follows: given an \emph{explosive} process $\HG$, is it possible to stop it from being explosive by adding either an incubation period or a contagious period? And, in case of an affirmative answer, what conditions must be imposed on those periods to ensure conservativeness of the resulting process?

\subsection*{Structure of this thesis}
\addcontentsline{toc}{subsection}{Structure of this thesis}
The first chapter starts with introducing necessary terminology and notation. It furthermore gives a more formal definition of the models followed by a presentation of the main results of this thesis. 
\\
The second chapter is fully devoted to classical age-dependent branching processes:
the only parameters are the offspring distribution and the infection-time. Our attention is thus restricted to a setting with infinite contagious periods and zero incubation times.
\\In the third chapter we study \emph{forward} age-dependent branching processes with contagious periods. These model the early stage of the forward spread, without the presence of an incubation period.
\\Chapter four deals with \emph{backward} age-dependent branching processes with contagious periods. Their purpose is to model the process leading to the infection of an individual when the spread of a disease is subject to contagious periods.
\\Chapter five compares the forward and backward processes with contagious periods. Interestingly, in terms of explosiveness, both processes exhibit similar behaviour.
\\In chapter six we deal with \emph{forward} age-dependent branching processes with incubation periods. We shall study how the epidemic spreads forward in time under the influence of an incubation period.
\\Chapter seven treats \emph{backward} age-dependent branching processes with incubation periods.
\\Finally, in chapter eight we shall compare backward to forward processes with incubation  periods. A detailed analysis will show that, in a wide class of processes, explosiveness of the forward process may be deduced from explosiveness in the backward process.  

\newpage

\section{Models and Main Results}

\subsection{Preliminary Definitions}
Here we present necessary terminology and definitions that we shall make frequent use of.

\begin{definition}
By a random variable $D$ that follows a \emph{heavy-tailed power law} (\citep{Ho14}) we mean that $D$ satisfies, for some  $c,C > 0, \alpha \in (0,1)$,
\begin{equation}
\frac{c}{x^{\alpha}} \leq \mathbb{P}(D > x) \leq \frac{C}{x^{\alpha}}, 
\label{eq::HEAVY_TAIL}
\end{equation}
for all sufficiently large $x$.
\end{definition} 
\begin{definition}
For any $\alpha \in (0,1)$, the function $\Ha: [0,1] \to [0,1]$ is defined for $s \in [0,1]$ by
\be
h_{\alpha}(s) = 1 - (1-s)^{\alpha}.
\label{eq::Ha}
\ee 
\end{definition}
\noindent
We show later that $\Ha$ is the probability generating function of a random variable $D$ that satisfies \eqref{eq::HEAVY_TAIL}. 
\begin{definition}
A random variable $D$ is said to have a \emph{plump} distribution (\citep{AmDeGrNe13}) if, for some $\epsilon > 0$,
\begin{equation}
\mathbb{P}(D \geq m^{1+\epsilon}) \geq \frac{1}{m},
\label{eq::AGE_PLUMP}
\end{equation}
for all sufficiently large $m$.
\end{definition}

\begin{definition}
By an \emph{i.i.d.} sequence of random variables we shall mean a sequence of random variables that are mutually independent and identically distributed. 
\end{definition}

\begin{notation}
For any function $f$, we denote its \emph{inverse} function by $f^{\leftarrow}$.
\end{notation}

\begin{notation}
For any event $A$, we denote its \emph{complement}  by $\bar{A}$.
\end{notation}

\begin{notation}
For any event $A$, we denote its \emph{indicator} function  by $\indicator{A}$.
\end{notation}

\begin{notation}
For any random variable $Y$ and distribution function $F$, the notation $Y \sim F$ means that $Y$ has \emph{distribution} function $F$.
\end{notation}

\begin{notation}
For any two random variables $Y$ and $Z$, the notation $Y \overset{d}= Z$ means that $Y$ is \emph{equal in distribution} to $Z$.
\end{notation}

\begin{notation}
For any two random variables $Y$ and $Z$, the notation $Y \overset{d} \leq Z$ means that $Y$ is \emph{stochastically dominated} by $Z$.
\end{notation}

\subsection{Age-Dependent Branching Processes}
Let $h$ be a probability generating function and $G$ the distribution function of a non-negative random variable. The first model considers an age-dependent branching process that is defined as follows:
every individual $i$ has a random life-length $X_i \sim G$. An individual produces an offspring of random size $D_i$ with probability generating function $h: s \mapsto \E{s^{D_i}}$ exactly when she dies.
We assume that all life-lengths and family sizes are independent of each other. 
The process starts with a single individual that is called the \emph{root} and has life-length $X_0$.
We use $t$ as a running parameter for the time and distinguish between two cases:
\begin{itemize}
\item the delayed case: the time $t=0$ corresponds to the birth of the root; it is precisely the setting studied by Grey in \cite{Grey74}. We denote this process by $(h,G)^{\text{del}}$;
\item the usual case: the time $t=0$ corresponds to the death of the first individual. We denote this process by $(h,G)$.
\end{itemize}
Note that the number $N(t)$ of  \emph{alive} individuals at time $t \geq 0$ in the usual case  equals the number $N^{\text{del}}(t)$  of individuals in the delayed case at time $t + X_0$. Having defined the process, we shall want to know when it is explosive.

\subsubsection{Explosiveness}
Denote by $\tau_n$ the birth-time of the $n$-th individual. It may happen that there exists some random variable $V < \infty$ such that with positive probability $\tau_n \leq V$ for    all $n \geq 0$, in which case we speak of an explosive process (\cite{Grey74}):
\begin{definition}
We say that a process $(h,G)$ is explosive if there is a $t > 0$ such that $\P{N(t) = \infty} > 0$.  
Similarly, by saying that a process $(h,G)^{\text{del}}$ is explosive we mean that there is a $t > 0$ such that $\P{N^{\text{del}}(t) = \infty} > 0$. 
We call a process conservative if it is not explosive.
\label{def::AGE_explosiveness}
\end{definition}
\noindent Note that the process $(h,G)$ is explosive if and only if $(h,G)^{\text{del}}$ is. Hence, as our primary interest is on answering the explosiveness question, we restrict our attention to  the usual case. 

\subsubsection{Results}
%The theorems presented here enable us to decide on the explosiveness of a process $\HG$ when $h$ is the probability generating function of a distribution $D$ that follows a heavy-tailed power law, conform \eqref{eq::HEAVY_TAIL}.
Later we demonstrate that $\Ha$ as defined in \eqref{eq::Ha} is the probability generating function of  a distribution $D$ that follows a heavy-tailed power law, conform \eqref{eq::HEAVY_TAIL}. It is precisely this probability generating function that Grey studies in \citep{Grey74}. Moreover, he proved that to every offspring distribution $D$ with infinite expectation there exist two special life-length distributions. On the one hand a  distribution that together with $D$ constitutes an explosive process, although it puts zero mass at the origin. On the other hand a life-length distribution $G$, such that $G(t)>0$ for all $t>0$, that forms together with $D$ a conservative process. Grey validates these theorems for $\Ha$ by presenting two functions that are extremely flat around the origin: $G_{\ell,\beta}$ defined for fixed $\ell,\beta > 0$ and small $t \geq 0$ by $G_{\ell,\beta}(t) = \exp \lr{-\frac{ \mathcal{\ell} }{t^{\beta}}}$ and $\widehat{G}_{k,\gamma}$ given for fixed $\gamma,k>0$ and small $t \geq 0$ by $\widehat{G}_{k,\gamma}(t) = \exp\lr{ -\exp\lr{\frac{k}{t^{\gamma}}}}$.
The process $(\Ha,G_{\ell,\beta})$ is demonstrated to be explosive for all $\alpha \in (0,1)$ and $\ell,\beta > 0$, though $(\Ha,\widehat{G}_{k,1})$ is conservative whenever $\alpha \e^k > 1$. The latter examples raises the presumption that explosiveness of a process $\HaG$ strongly depends on the particular value of $\alpha$.
However, we shall demonstrate that, regardless of the distribution $G$, explosiveness of $\HaG$, does \emph{not} depend on the value of $\alpha \in (0,1)$:
\begin{theorem}
Let $G$ be the distribution function of a non-negative random variable. 
Then, the process $\HaG$ is either explosive for all $\alpha \in (0,1)$ or conservative for all $\alpha \in (0,1)$.
\label{th::AGE_allAlpha}
\end{theorem}
\noindent
We extend this theorem to include all processes $(h^L_{\alpha},G)$, where 
\be h^L_{\alpha}(s) = 1 - (1-s)^{\alpha} L(\frac{1}{1-s}), \quad \mbox{as } s \to 1,  \label{eq::HaL}\ee
with $L$  a function varying slowly at infinity and $0<\alpha<1$:
\begin{theorem}
Let $G$ be the distribution function of a non-negative random variable. Let $L$ be a function varying slowly at infinity. 
Then, the process $(h^L_{\alpha},G)$ is either explosive for all $\alpha \in (0,1)$ or conservative for all $\alpha \in (0,1)$.
\label{th::AGE_allAlphaL}
\end{theorem}

\noindent
These theorems are based on the following proposition that is interesting in its own right because it provides a much easier tool to check if a process is explosive or not:
\begin{proposition}
Let $h$ be the probability generating function of a distribution $D$ that satisfies (\ref{eq::HEAVY_TAIL}) for some  $c,C > 0, \alpha \in (0,1)$  and all sufficiently large $x$.
Let $G$ be the distribution function of a non-negative random variable. Let $(X_i^j)_{i,j \geq 0}$ be an i.i.d. sequence of random variables with distribution $G$ and furthermore independent of everything else. 
Under these assumptions, the process $\HG$ is explosive if and only if
\begin{equation}
\mathbb{P} \left( \sum_{n=1}^{\infty} \min (X_1^n , \ldots , X_{\e^{1 / \alpha^n}}^n ) < \infty  \right) = 1.
\label{eq::MIN_SUMMABLE_EQUIVALENT}
\end{equation}
\label{thm::ALPHA_GROWTH}
\end{proposition}
\noindent
In the models that are to come, by an explosive process we shall mean that $N(t)$, $N^f(t), N^b(t), N_f(t)$ or $N_b(t)$,  as it is defined there, becomes infinite for some $t < \infty$ with positive probability. Those quantities are to be understood as a measure for the size of the corresponding branching process in one sense or another. 

\subsection{Forward Age-Dependent Branching Processes with Contagious Periods}
As the classical age-dependent branching process is a fairly crude way to model an epidemic spread, we try to improve on it by introducing contagious periods. 
\\
Let $h$ be a probability generating function and both $G$ and $C$ be distribution functions of non-negative random variables.
We assume that every individual $i$ \emph{may} possibly infect a random number $D_i$ (with probability generating function $h$) of others that we shall call her children. Though, infection is only possible as long as she herself is contagious: we denote the length of this period by $\tau_i^C \sim C$. Hence, a person $j$ is infected by her parent $m(j)$ if her infection time $X_j \sim G$ is less than or equal to the contagious period $\tau^C(m(j))$ of her mother. All random variables are independent of each other. However, the children of a particular parent are not completely independent in the sense that the mother has the same contagious period for all of them. We shall say that an individual is alive or born if she is infected and emphasize  that not everybody will eventually be infected. 
\\
The variable $t$ serves as a running parameter for the time and we let the process start at $t=0$ with only one infected person, the \emph{root}. The root will infect (or give birth) to  an \emph{effective} offspring in distribution equal to
\be D^f = \s{i=1}{D} \indicator{X_i \leq \tau^C}, \label{eq::CP_EFF} \ee 
where $D$ has probability generating function $h$ and $\tau^C$ has distribution $C$. All her infected offspring shall, at their respective times of birth, infect again a random number of individuals in distribution equal to $D^f$, and so on. Thus, we have a recurrent relation to describe the process: at the birth (or infection) of an individual, the process develops from it as an \emph{independent} copy of the original process, in which the individual serves as a root. Hence, we have an age-dependent branching process 
that we denote by $\HGCf$. 
\\
By  $N^f(t)$ we shall denote the number of individuals in the \emph{coming generation} at time $t$. 
Here, we say that an individual in the branching process is belonging to the coming generation if and only if her mother has been already born but she is not yet, but she will eventually.
\begin{remark}
Note that we re-obtain the classical branching process $\HG$ by putting $\P{C = \infty} = 1$, that is, by a slight abuse of notation: $\HG = (h,G,C = \infty)_f$.
\label{rem::HG_HGinfty}
\end{remark}

\subsubsection{Results}
Consider an explosive process $(h,G)$. A natural question to ask is under what conditions will adding a contagious period $C$ stop this process from being explosive? 
\\
One possible way to answer this question is observing that a process is explosive if and only if with positive probability the branching process has a finite ray. That means the following: an infinite sequence of vertices (that are consecutive descendants of each other) together with their life-times, lying on a path starting at the root. The length of a ray is then the (possibly infinite) sum of life-times on the edges in this path.
Now, if $\HGCf$ is conservative, but $\HG$ not, then $C$ must be such that it kills with probability 1 at least one edge on every finite ray. We will employ shortly this observation to show that $\HGCf$ is explosive if there are $\beta,\theta > 0$ such that $C(t) \leq \beta t$ for all $t \in [0,\theta]$, which is the content of Theorem \ref{thm::CP_ray}:
\begin{theorem}
Let $h$ be a probability generating function. Let $G$ be the distribution function of a non-negative random variable. Let $C$ be the distribution function of a non-negative random variable such that for some  $\beta,\theta > 0$:
\begin{itemize}
\item $C(t) \leq \beta t$ for all $0 \leq t \leq \theta$;
\item $C(t) < 1$ for all $t > 0$.
\end{itemize} 
If the process $\HG$ is explosive, then so is $\HGCf$.
\label{thm::CP_ray} 
\end{theorem}
\noindent
Another way of addressing the explosiveness question is by employing analytical techniques to study a fixed point equation describing $\HGCf$, see for more details Section 3.1. By that approach we will establish that for any $\alpha \in (0,1)$ and $C$ such that $C(0) < 1$, the process $\HaGCf$ is explosive if and only if $\HaG$ is explosive: %We extend this result to Theorem \ref{thm::HaLGCf}.

\noindent
%The second theorem shows that only a contagious period that is identically zero can stop the process $\HaG$, $\alpha \in (0,1)$, from being explosive:
\begin{theorem}
Let $\alpha \in (0,1)$. Let $G$ be the distribution function of a non-negative random variable. Let $C$ be the distribution function of a non-negative random variable such that  $C(\theta) < 1$ for some $\theta >0$. If the process $\HaG$  is explosive, then so is  $(\Ha,G,C)^f$.   
\label{thm::HaGCf}
\end{theorem}
\noindent
We extend this result to the setting where we allow for probability generating functions $\Ha^L,$ conform \eqref{eq::HaL}, yielding the statement of Theorem \ref{thm::HaLGCf}:
\begin{theorem}
Let $\alpha \in (0,1)$. Let $G$ be the distribution function of a non-negative random variable. Let $C$ be the distribution function of a non-negative random variable such that  $C(\theta) < 1$ for some $\theta >0$. Let $L$ be a function varying slowly at infinity.   If the process $(\Ha^L,G)$  is explosive, then so is  $(\Ha^L,G,C)^f$.   
\label{thm::HaLGCf}
\end{theorem}
\noindent
We emphasize the implication of Theorem \ref{thm::HaLGCf}: a process $(\Ha^L,G)$ can never be stopped from explosion by adding a non-zero contagious period. Hence, from the explosion perspective, the processes $(\Ha^L,G)$ and $(\Ha^L,G,C)^f$ are equivalent if $C(0) \neq 1$. But, the explosion times (see \eqref{eq::CP_FW_EXPL_TIME} and \eqref{eq::CP_BW_EXPL_TIME}) change and so does the survival probability. Thus, for the proportion of infected individuals in the random graph, adding a contagious period will be relevant. 

\subsection{Backward Age-Dependent Branching Processes with Contagious Periods}
The next step in the analysis of the epidemic spread studies how an individual may be infected \emph{herself}, rather than contaminating others. It aims at identifying a possible infection trail towards an individual and is in essence the \emph{backward} in time equivalent of the forward process defined in the previous section.
\\
We denote this process by $\HGCb$, where as usual $h$ is the probability generating function of an offspring distribution and both $G$ and $C$ are distribution functions of non-negative random variables. We now describe this process in more detail. At time $t=0$ the process starts with only the \emph{root}, whose possible infection by $D$ (with probability generating function $h$) others, say her offspring, is under consideration. A child $j$ of the root infects her (provided that $j$ herself will be infected at some point of time) only if the time $X_j \sim G$ needed for $j$ to infect the root falls within the contagious period $\tau^C(j) \sim C$ of $j$. 
\\
Repeated application of this procedure to subsequent descendant gives the following recurrent relation for the branching process. At birth, an individual is contacted (and possibly infected) by an \emph{effective} number of individuals in distribution equal to 
\be D^b = \s{j=1}{D} \indicator{X_j \leq \tau^C(j)}. \ee
I.e., she is contacted by $D$ individuals, but only $D^b$ might possibly infect her. 
Each of her effective offspring, is then identified as a \emph{possible infector} after a random time with distribution $G_{\text{eff}}$, and the process starts again from it as an independent copy of the original process, with the particular child in the role of a root. $G_{\text{eff}}$ is the distribution function of $(X_1|X_1 \leq \tau^C(1))$. Due to the \emph{independence} between all individuals, we identify this process with a classical branching process, namely $(h_{\text{eff}},G_{\text{eff}})$, where $h_{\text{eff}}$ is the probability generating function of $D^b$.
\\
Note that there is another, equivalent, classical branching processes giving an appropriate description: $(h,G_C)$, with $G_C$  the law of a random variable $Y$ such that:
\begin{equation}
 Y = \left\{ 
  \begin{array}{l l}
    (X_1|X_1 \leq \tau^C) & \quad \text{with probability $\P{X_1 \leq \tau^C}$},\\
    \infty & \quad \text{with probability $\P{X_1 > \tau^C}$}.
  \end{array} \right.
\end{equation}
Indeed, the root will never be infected by $j$ if $\tau^C(j) < X_j$ and in that case we may as well say that it takes an infinite amount of time for the disease to pass.

\subsubsection{Results}
Also with this process we wonder how adding a contagious period might stop $\HG$ from being explosive. As this process is somehow easier than the forward process, we are able to prove a result similar to Theorem \ref{thm::HaGCf}, capturing even a broader class of probability generating functions: 
\begin{theorem}
Let $h$ be the probability generating function of an offspring distribution $D$ that is plump, conform \eqref{eq::AGE_PLUMP}. Let $G$ be the distribution function of a non-negative random variable. Let $C$ be the distribution function of a non-negative random variable such that $C(T) < 1$ for some $T > 0$. If $\HG$ is explosive, then so is $\HGCb$.
\label{thm::BCP_explosiveness}
\end{theorem}
\noindent
Again, a \emph{plump} process $\HG$ cannot be stopped from explosion by adding any non-zero contagious period, but the explosion time (see \eqref{eq::IP_FW_EXPL_TIME} and \eqref{eq::IP_BW_EXPL_TIME}) and survival probability do change. The backward process is in fact a binomial thinning of the original process: $\HGCb = (h_{\text{eff}},G_{\text{eff}})$, where only the \emph{right}, i.e., short, edges survive. Evidently, we conclude that the growth of the branching process with offspring probability generating function $h_{\text{eff}}$ is thus still fast enough to maintain explosiveness. 
\\
It remains an open question whether it holds for all probability generating functions $h$ (and $C$ with $C(0)<1$) that $\HGCb$ is explosive if and only if $(h,G)$ is. The theorems here hint in favour of this conjecture. However, we remark that in this particular case the effective offspring of $\HaG$ (with $\alpha \in (0,1)$) satisfies again a power-law with the same degree $\alpha$ ensuring that the growth of generation sizes is of the same order as in the process $\HG$. Not all probability generating functions have this property, for example $h_{1}: s \mapsto s + (1-s)\log(1-s)$ \emph{might} be a candidate for which the conjecture does not hold. Note that $h_{1}$ corresponds to the probability distribution that we obtain by letting $\alpha \uparrow 1$ in \eqref{eq::HEAVY_TAIL} and it thus gives stochastically smaller offspring than  $h_{\alpha}$ for every $\alpha \in (0,1)$.

\subsection{Comparison of Backward and Forward Processes with Contagious Periods}
Due to the strong similarities between the forward and backward process, it is natural to ask whether explosiveness of the one implies explosiveness of the other. In fact, we conjecture that for all probability generating functions $h$ and distribution functions $G$ and $C$ of non-negative random variables the following holds:
\[ \HGCf \mbox{ is explosive } \Leftrightarrow \HGCb \mbox{ is explosive. } \]
Due to the identification $\HGCb = (h,G_C)$, validity of this conjecture would give us a perfect tool to check explosiveness of $\HGCf$ through studying the much better understood process $(h,G_C)$.
\\
The right-to-left implication has been partly shown in the two previous sections. Here we prove the left-to-right implication in its full generality. For that, we need the concept of explosion times:
\begin{definition}
The explosion times $V^f$ and $V^b$ in the forward, respectively, backward contagious process are defined as
\be V^f = \inf \{ t \geq 0 \  | \  N^f(t) = \infty \}, \label{eq::CP_FW_EXPL_TIME} \ee
and, 
\be V^b = \inf \{ t \geq 0 \  | \  N^b(t) = \infty \}. \label{eq::CP_BW_EXPL_TIME} \ee
\end{definition}
\noindent
The implication is included in the following theorem, which states that the forward explosion time is stochastically smaller than that time in the backward process:

\begin{theorem}
Let $h$ be a probability generating function. Let $G$ and $C$ be distribution functions of non-negative random variables. The explosion times $V^f$ and $V^b$ in the forward, respectively, backward process satisfy the relation
\be V^b \overset{d}\leq V^f.  \ee
Hence, if the process $\HGCf$ is explosive, then so is $\HGCb$.
\label{th::CP_COMPARING_BW_FW}
\end{theorem}
\noindent
Note that this theorem provides a way to check conservativeness of a forward process: if the process $\HGCb$ is conservative, then so is $\HGCf$.

\subsection{Forward Age-Dependent Branching Processes with Incubation Periods}
Evidently, incorporating the influence of an incubation period on age-dependent branching processes leads to a more realistic model for the early stages of an epidemic. The better part of the model that we are about to introduce resembles the forward process with contagious periods, though the transmission criteria could not be more different. 
\\
We assume that an individual $i$ may possibly infect a random number $D_i$ (with probability generating function $h$) which we shall call her children. However, in this model, contamination may only take place 
after some incubation period $\tau_I(i) \sim I$ has elapsed. That is, person $j$ will only be infected by her mother $i$ if the infection-time $X_j \sim G$ is larger than or equal to the incubation period $\tau_I(i)$:
\be X_j \geq \tau_I(i). \label{eq::incubation_criteria} \ee
All random variables are independent of each other, though we emphasize that dependence comes in due to the fact that a mother has the same incubation period for all of her children. 
\\
The rest of this process follows exactly the description of the forward process with contagious periods. That is, it is an age-dependent branching process where every individual infects an \emph{effective} number of offspring equal in distribution to
\be D_f = \s{j=1}{D} \indicator{X_j \geq \tau_I}, \ee
where $D$ has probability generating function $h$ and $\tau_I$ has distribution $I$. Each  child $j$ in the \emph{effective} offspring of an individual $i$ in the branching process will be infected a time $(X_j | X_j \geq \tau_I(i))$ after the infection of individual $i$, and so on.
We denote this process by $\HGIf$ and shall study its explosiveness. 

\subsubsection{Results}
From the criterion \eqref{eq::incubation_criteria}, we observe that this time only the \emph{bad}, or lengthy, edges are kept. Hence, we expect that to every explosive process $\HG$ there exist various distributions $I$ such that $\HGIf$ is conservative. An obvious example is $I$ such that $\tau_I = \epsilon$ a.s. for some fixed $\epsilon > 0$. But more is true, we prove that explosiveness of $\HG$ and $(h,I)$ is necessary for $\HGIf$ to be explosive: 
\begin{theorem}
Let $h$ be a probability generating function. Let both $G$ and $I$ be distribution functions of non-negative random variables. If the process $(h,I)$ is conservative, then so is $\HGIf$, regardless of $G$. 
\label{thm::IP_for_necessary}
\end{theorem}
\noindent
Note that this is a very important observation: the incubation period has a \emph{crucial} influence on the explosiveness of a process. Its implications are far reaching: for a conservative process there is \emph{no} finite time $t$ such that the number of individuals in the branching process before time $t$ exceeds $n^{\rho}$ for some $\rho \in (0,1]$ and all $n$. Hence, the number of infected individuals in the random graph mentioned in the introduction grows as a function of $n$, rather than that it is independent of the population size.  
\\
Eventually, we would like to find a criterion in terms of $h,G$ and $I$ that tells us exactly when $\HGIf$ is explosive. We conjecture that:
\[ \HGIf \mbox{ is explosive } \Leftrightarrow  \HG \mbox{ and } (h,I) \mbox{ both explode.} \]
Its validity would provide us with a very transparent way of verifying explosiveness of $\HGIf$. At this point in the analysis we are not ready yet to prove this conjecture. It will turn out useful to first study  the backward process $(\Ha^L,G,I)_b$ and prove a similar conjecture for it under a very mild condition on $G$ and $I$. After establishing that result, we shall see that for most life-time distributions $G$, $(\Ha^L,G,I)_f$ is explosive if and only if $(\Ha^L,G,I)_b$ is. We expect this equivalence to hold between all forward and backward processes, though we prove a slightly less general result to avoid tedious technical difficulties. 

\subsection{Backward Age-Dependent Processes with Incubation Periods}
Here we study how an infection might reach a person backward in time, in the same fashion as the contagious period backward process. We denote this process by $\HGIb$ where $h$ is a probability generating function and both $G$ and $I$ are the distribution functions of non-negative random variables. 
\\
For an individual $i$, we let $D_i$ be the number of persons that might possibly infect her and call those persons her offspring. A child $j$ of $i$ has an infection time $X_j \sim G$  that tells how long it takes before she might pass the disease to $i$. Obviously, $j$ can only infect $i$ if a certain incubation period $\tau_I(j) \sim I$ has passed:
\[ \tau_I(j) \leq X_j. \]
Thus, $i$ will never be infected by $j$ if $X_j < \tau_I(j)$ and in that case it thus takes an infinite amount of time before $j$ passes the disease to $i$. We conclude that the process $\HGIb$ is essentially equal to $(h,G_I)$, where $G_I$ is the distribution of a random variable $Y$ that behaves as

\begin{equation}
 Y = \left\{ 
  \begin{array}{l l}
    (X|X \geq \tau_I) & \quad \text{with probability $\P{X \geq \tau_I}$},\\
    \infty & \quad \text{with probability $\P{X < \tau_I}$}.
  \end{array} \right.
\end{equation}

\subsubsection{Results}
We conjecture that
\[ \HGIb \mbox{ is explosive } \Leftrightarrow \HG \mbox{ and } (h,I) \mbox{ are both explosive.} \]
We will prove one side of the assertion, that is $\HG$ and $(h,I)$ must necessarily be explosive if $\HGIb$ is to explode, i.e., Theorem \ref{thm::IP_back_necessary}.
The other direction is shown for almost all processes where $h = \Ha$ for some $\alpha \in (0,1)$:

\begin{theorem}
Let $h$ be a probability generating function. Let $G$ be the distribution function of a non-negative random variable. Let $I$ be the distribution function of a non-negative random variable. If the process $\HGIb$ is explosive, then $\HG$ and $(h,I)$ both explode.
\label{thm::IP_back_necessary}
\end{theorem}
\noindent
The implications of this theorem become clear once we rephrase it: given a process $\HG$, we may already stop it from being explosive by adding an incubation period $I$ that is just flat enough for $(h,I)$ to be conservative. Recalling the definition of the backward process and in particular the statement \eqref{eq::INF_TIME}, we see that an uniformly picked individual 
on the random graph \emph{may} be exempted from illness if the incubation period is on average long enough.
 
\begin{theorem}
Let $\alpha \in (0,1)$. Let $G$ be the distribution function of a non-negative random variable, such that $G(0) = 0$. Let $I$ be the distribution function of a non-negative random variable, such that \emph{either}
\begin{itemize}
\item there exists $T>0$ such that $I \geq G$ on $[0,T]$, \emph{or},
\item there exists $\delta > 0$ such that $G$ has density $g$ on $[0,\delta)$, $I$ has density $i$ on the \emph{open} interval $(0,\delta)$ and $g \geq i$ on $(0,\delta)$.
\end{itemize}
If both processes $\HaG$ and $(\Ha,I)$ are explosive, then so is $(\Ha,G,I)_b$. 
\label{thm::IP_back_sufficient}
\end{theorem}
\begin{remark}
We excluded on purpose $G(0)>0$. Since then, adding $I$ (with $I(0) = 0$) might stop it from explosion. For instance if $G(\infty) = 1 - G(0)$, then $\HaGIb$ is always conservative.
Indeed, all individuals that have infection-time larger than the incubation period have in fact an infinite life-time.
\end{remark}
\noindent
Note that the assumptions in the last theorem are not that stringent: in fact, the only case that would violate it is when $I-G$ switches sign infinitely often around the origin. We therefore tend to believe that the theorem holds in the most general setting. This would be rather striking, since, as far as explosiveness is concerned, there seems to be no mutual dependence between the processes $\HaG$ and $(\Ha,I)$: it is merely their individual explosiveness that matters.  

\subsection{Comparison of Backward and Forward Processes with Incubation Periods}
Inspired by the results for contagious processes, we again conjecture that:
\[ \HGIf \mbox{ is explosive } \Leftrightarrow \HGIb \mbox{ is explosive.}\]
We will prove the first assertion in its full generality, i.e., the forward process only explodes if the backward process does so.
We prove the other direction under a very mild condition on $G$ for all processes where $h=h_{\alpha}^L$ for some $\alpha \in (0,1)$ and any function $L$ that varies slowly at infinity. For the left-to-right implication, we need the concept of explosion times for incubation processes:
\begin{definition}
The explosion times $V_f$ and $V_b$ in the forward, respectively, backward incubation process are defined as
\be V_f = \inf \{ t \geq 0 \  | \  N_f(t) = \infty \}, \label{eq::IP_FW_EXPL_TIME} \ee
and, 
\be V_b = \inf \{ t \geq 0 \  | \  N_b(t) = \infty \}. \label{eq::IP_BW_EXPL_TIME} \ee
\end{definition}
\noindent

\begin{theorem}
Let $h$ be a probability generating function. Let both $G$ and $I$ (with $I(0)=0$) be distribution functions of non-negative random variables. Then, the explosion times $V_f$ and $V_b$ in the forward, respectively, backward process satisfy the relation
\be V_b \overset{d}\leq V_f.  \ee
Hence, if the process $\HGIf$ is explosive, then so is $\HGIb$.
\label{thm::IP_COMPARING_BW_FW}
\end{theorem}
\noindent
Theorem \ref{thm::IP_COMPARING_BW_FW} provides an easy tool to verify conservativeness of a process: if the backward process is conservative, then so is the forward process. The former is in general better understood; explosiveness of $\HGIb$ requires $\HG$ and $(h,I)$ to explode (Theorem \ref{thm::IP_back_necessary}). Hence, explosiveness of both $\HG$ and $(h,I)$ is a necessary condition for $\HGIf$ to explode. Note that we hereby established an alternative proof of Theorem \ref{thm::IP_for_necessary}.
\\
Deducing explosiveness of $(\Ha^L,G,I)_f$ if $(\Ha^L,G,I)_b$ is postulated to explode requires a more elaborate proof. We shall see that this is feasible under a very mild condition on $G$ by showing that, with positive probability, a realization of the forward branching process contains an exploding path (i.e., an infinite path with finite length), if the backward process explodes. 
\\
The proof is inspired by an algorithm, in \cite{AmDeGrNe13}, that is used to prove equivalence of explosiveness and min-summability for \emph{plump} processes. A plump branching process $\HG$ satisfies a certain deterministic growth condition (i.e., at least double exponential) after some initial random development of the branching process. More precisely, there is a function $f$ such that with positive probability $Z_n \geq f(n)$ for all $n$, where $Z_n$ represents the size of generation $n$. Min-summability of the process implies
\be \s{n=1}{\infty} \minun{i \in \{ 1, \cdots, f(n) \}} X_i^n < \infty, \label{eq::IP_COMPARE_MINSUM} \ee  
where $(X_i^n)_{i,n}$ is an i.i.d. sequence with distribution $G$. 
\\
The algorithm in \citep{AmDeGrNe13} works according to the following idea: in step $n$ of the algorithm, let $x_n$ denote the lowest node on the candidate exploding path. Consider only those children of $x_n$ who, on their turn, have an offspring of size at least $f(n+1)$. Call them the \emph{good} children, denote their number by $W_n$ and their i.i.d. life-lengths with $\widehat{X}_1^n , \ldots, \widehat{X}_{W_n}^n \sim G$. Let $x_{n+1}$ be the one among them that has shortest life-length. Amini et al. show that with positive probability the path length is finite, i.e., 
\be \s{n=1}{\infty} \minun{i \in \{ 1, \cdots, W_n \}} \widehat{X}_i^n < \infty, \ee 
under the condition \eqref{eq::IP_COMPARE_MINSUM}. Their proof uses that $W_n$ has approximately the same order of magnitude as $f(n)$ for all $n$.
\\
We give another algorithm that comes up with an exploding path for the \emph{forward} processes. Unfortunately, there is no independence among the children of a particular parent to employ, therefore we cannot hope that a condition as simple as min-summability would be sufficient. We may, however, make use of the fact that the growth in the process follows again almost a deterministic course. Taking this into account, the algorithm picks the \emph{good} children in a similar way and their number grows approximately as fast as the generation sizes in the backward process. However, this time, it is not sufficient to pick the child with smallest infection-time among the \emph{good} children. The complication here is that if we would pick the candidate with smallest infection-time, she might still have an arbitrary large incubation-time: all her children have a life-time larger than this incubation period. Thus, instead, we must pick the child for which the sum of the infection-time and incubation-time is smallest. It turns out feasible to prove convergence of the infection-times along the constructed path by taking into account that the backward process is min-summable, see the proof of Theorem \ref{thm::IP_COMPARISON_AGING} for details:

\begin{theorem}
Let $\alpha \in (0,1)$. Let $G$ be the distribution function of a non-negative random variable. Assume that there exists a $\delta > 0$ such that
\be (X-x|X \geq x) \overset{d}{\leq} X, \label{eq::aging} \ee
for all $x \leq \delta$, where $X$ has distribution $G$. Let $I$ be the distribution function of a non-negative random variable. Let $L$ be a function varying slowly at infinity. If the process $\HaLGIb$ is explosive, then so is $\HaLGIf$. 
\label{thm::IP_COMPARISON_AGING}
\end{theorem}
\noindent
We remark that the condition $(X-x|X \geq x) \overset{d}{\leq} X$ is met in most cases of interest, namely those for which the density of $G$, $g$ satisfies \emph{either} $g(0) =0$ and $g$ is locally increasing at zero, \emph{or}, for some $\epsilon > 0$, $g \geq \epsilon$ on some interval around zero. Only distribution functions $G$ for which $g$ has infinitely many zeros around the origin fail to meet this condition. Hence, we have a very general theorem:

\begin{theorem}
Let $\alpha \in (0,1)$. Let $G$ be the distribution function of a non-negative random variable with density $g$ around zero, such that, \emph{either},
\begin{itemize}
\item $g(0)=0$ and there exists a $\delta > 0$ such that $g$ is increasing on $[0,\delta]$, \emph{or,}
\item there exist $\epsilon, \delta > 0$ such that $g \geq \epsilon$ on $[0,\delta]$.
\end{itemize}
Further, let $I$ be the distribution function of a non-negative random variable. Let $L$ be a function varying slowly at infinity. If the process $\HaLGIb$ is explosive, then so is $\HaLGIf$. 
\label{thm::IP_COMPARISON_MAIN}
\end{theorem}
\noindent
Considering the proofs of the last two theorems, the condition on $G$ is only there to avoid tedious technical calculations. We expect therefore that for all $\alpha \in (0,1)$, $L$, $G$ and $I$ there is equivalence in terms of explosiveness between the forward and backward processes. To handle processes with a more general probability generating function $h$, we should again try to find appropriate lower bounds on the generation sizes. Ideally, it might be feasible to use the \emph{same} algorithm and show that its number of \emph{good} children grows with the same order of magnitude as the generations in the backward process. 
\\
Establishment of equivalence between forward and backward processes would imply a substantial simplification of the explosiveness question for forward processes.

\newpage\null\thispagestyle{empty}\newpage

\section{Age-Dependent Branching Processes}
In this section we are concerned with an age-dependent branching process $\HG$, where $h$ is a probability generating function and $G$ the distribution of a non-negative random variable. 
Later on, in Section 3, we will study age-dependent processes $\HGCf$, where an individual $i$ is born if and only if her life-length (or infection-time) $X_i$ is less than or equal to the contagious period $\tau^C(m(i)) \sim C$ of her parent $m(i)$. As pointed out in Remark \ref{rem::HG_HGinfty}, the process $\HG$ is  exactly the same as $(h,G,C=\infty)_{f}$. Hence, we shall make generous use of the results that are derived in Section 3. 
\\
Further, some of our results here are the same as in \citep{Grey74} (our prove technique is based on the ideas presented there).
We define the generating function $F: [0,1] \times [0,\infty] \mapsto [0,1]$ of the total alive population at time t
by $F(s,t) = \E{s^{N(t)}}$. Putting $C = \infty$ in equation \eqref{eq::CP_FPE_MAIN} shows that
\eqn{F(s,t) =h\left(s (1-G(t)) + \int_0^t F(s,t-u)\mathrm d G(u) \right),}
for $(s,t) \in [0,1] \times [0,\infty]$. 
Hence, the function $\phi: [0,\infty] \mapsto [0,1]$, defined for $t \geq 0$ by $\phi(t) = \displaystyle \lim_{s \uparrow 1} F(s,t)$, satisfies
\be 
\phi(t) =h\left( 1-G(t) + \int_0^t \phi(t-u)\mathrm d G(u) \right), \quad t \geq 0.
\label{eq::AGE_gen_func}
\ee

\subsection{Explosiveness}
Recall the concept of an explosive process from Definition \ref{def::AGE_explosiveness}. Note that 
\eqn{\P{N(t) = \infty} = 1 - \s{n=0}{\infty} \P{N(t) = n} = 1 - \phi(t),}
and hence, the process $(h,G)$ is explosive if and only if $\phi(t) < 1$ for some $t > 0$.

\subsection{Fixed point equation}
Let 
\be\mathcal{V} = \{ \Phi \mbox{ }  | \mbox{ }  \Phi: [0,\infty) \mapsto [0,1] \} \label{eq::function_space} \ee
be the space of all functions that map $[0,\infty)$ into $[0,1]$. Define the operator $\TG: \mathcal{V} \mapsto \mathcal{V}$ for $\Phi \in \mathcal{V}$ by
\be \lr{\TG \Phi}(t) =   h\left( 1-G(t) + \int_0^t \Phi(t-u)\mathrm d G(u) \right), \quad t \geq 0.  \ee
We know that $\phi = \TG \phi$, but there may be more functions that satisfy such a relation, therefore we study the fixed point equation
\be \Phi = \TG \Phi, \hspace*{2 cm} \Phi \in \mathcal{V} \label{eq::AGE_FP_usual}, \ee
and in particular the number and behaviour of its solutions. We provide an overview of theorems that describe those solutions, their proofs are found in Section 3, by putting $C=\infty$ there.
\\
Obviously the function $t \mapsto 1$ solves \eqref{eq::AGE_FP_usual} and if it is the only solution to \eqref{eq::AGE_FP_usual} then $\phi=1$ and for all $t \geq 0$, $\P{N(t) = \infty} =0$, so that the process is conservative. 
The following theorem captures that $\phi$ is the \emph{smallest} solution to \eqref{eq::AGE_FP_usual} and we may thus conclude that a process is conservative if and only if $t \mapsto 1$ is the \emph{only} solution to \eqref{eq::AGE_FP_usual}:
\begin{theorem}
The function $\phi: t \mapsto 1- \P{N(t) = \infty}$ is the smallest solution to (\ref{eq::AGE_FP_usual}), in the sense that $\phi(t) \leq \Phi(t)$, for all $t \geq 0$, if $\Phi$ is another solution to (\ref{eq::AGE_FP_usual}).
\label{thm::AGE_right}
\end{theorem}
\noindent
Clearly $\P{N_t = \infty}$ does not decrease with $t$ and hence:
\begin{theorem}
The function $\phi$ is non-increasing. 
\end{theorem}
\noindent
The following theorem shows that the explosiveness of a process is essentially determined by the behaviour of $G$ and $C$ around the origin: 
\begin{theorem}
Let $h$ be a probability generating function. Let $G$ be the distribution function of a non-negative random variable. The process $(h,G)$ is explosive if and only if there exists $T > 0$ and a function $\Psi: [0,T] \to [0,1]$ such that $\Psi \neq 1$ and 
\be \Psi(t) \geq \lr{\TG \Psi}(t), \quad \forall t \in [0,T], \ee 
or, equivalently,
\be \Psi(t) \geq  h \left(1 - G(t) + \int_0^t \Psi(t-u)\mathrm d G(u)   \right), \quad \forall t \in [0,T].  \ee 
\label{th::AGE_simpl_cond}
\end{theorem}

\subsection{Comparison theorems}
The observation in Theorem \ref{th::AGE_simpl_cond} raises the question whether comparing life-length distributions around zero may allow us to deduce explosiveness of some process, to the explosiveness of another process. An affirmative answer is captured in the following theorem:
\begin{theorem}
Let $h$ be a probability generating function. Let both $G$ and $G^*$ be distribution functions of  non-negative random variables, such that $G(0)=G^*(0)=0$. Assume there is a $T$ such that $G^* \geq G$ on $[0,T]$. If the process $(h,G)$ is explosive, then so is $(h,G^*)$.
\label{thm::AGE_COMP_G}
\end{theorem}
\noindent
Finally, comparing offspring distributions leads to the following theorem:
\begin{theorem}
Let $h$ and $h^*$ be probability generating functions  such that for some $\theta < 1$ we have $h^*(s) \leq h(s)$ for all $s \in [\theta,1]$. Let $G$ be the distribution function of a non-negative random variable such that $G(0) = 0$. If the process $(h,G)$ is explosive, then so is	 $(h^*,G)$.
\label{thm::AGE_COMP_h}
\end{theorem}

\subsection{Examples}
In Grey's paper it is shown that the theorems above are rich enough to decide on the explosiveness question for some particular examples of processes. Object of study there is a process for which $h_{\alpha}$ ($\alpha \in (0,1)$) is defined by $h_{\alpha}(s) = 1 - (1-s)^{\alpha}$, for $s \in [0,1]$, together with a variety of life-length distributions that are extremely flat around the origin:
\begin{example}
Define the function $G_{\ell,\beta}$ for small $t \geq 0$ by $G_{\ell,\beta}(t) = \exp \lr{-\frac{ \mathcal{\ell} }{t^{\beta}}}$. By picking a suitable test-function in Theorem \ref{th::AGE_simpl_cond}, it follows that the process $(h,G_{\ell,\beta})$ is explosive for all $\ell,\beta > 0$.
\end{example}
\begin{example}
\noindent
Define the function $\widehat{G}_{k,\gamma}$ for small $t \geq 0$ by $\widehat{G}_{k,\gamma}(t) = \exp\lr{ -\exp\lr{\frac{k}{t^{\gamma}}}}$. This function is much flatter around zero than the one studied in the first example. In fact, Grey showed that for given $\alpha \in [0,1]$, $\gamma = 1$ and $k$ such that $\alpha e^k > 1$, the process $(h,\widehat{G}_{k,1})$ is conservative. 
\label{ex::GREY}
\end{example}
\noindent These two examples indicate that the borderline, in terms of life-length distributions that together with $h$ are just not flat enough to constitute a conservative process, must lie somewhere between $G_{l,\beta}$  and $\widehat{G}_{k,1}$ (where $l,\beta$ and $k$ take the values as above). We will shortly improve this boundary and shall also show that the treshold is independent of $\alpha$. But, before we proceed, we need to study the notion of min-summability.

\subsection{Min-summability}
If a process is explosive, then for some $t>0$  a realization of the branching process contains with positive probability a convergent ray smaller than $t$ (an infinite path starting from the root where the $i+1$-th vertex is the child of the $i$-th vertex for all $i \in \N$ ). To understand this, we introduce the notation BP($t$) for the collection of \emph{all dead and alive} individuals (vertices) in the branching process at time $t$ and let $|v|$ denote the generation of an individual $v \in$ BP($t$). Indeed, if there exists $k \in \mathbb{N}$ such that for all $v \in $ BP($t$), $|v| \leq k$, then, as $D < \infty$ almost surely, we must have $|$BP($t$)$|<\infty$. Therefore,
\eqn{\bigcap_{k=1}^{\infty} \{\exists v \in \mbox{BP}(t) \mbox{ s.t. } |v| = k \} = \{ |\mbox{BP}(t)| = \infty \}.}
Obviously, $\{ |\mbox{BP}(t)| = \infty \} = \{N(t) = \infty \}$, since $D \geq 1$. It follows that the life-length of individual number $n$ on a convergent ray is larger than or equal to the minimum over all life-lengths in generation $n$. In other words, a necessary condition for explosiveness is that the process is \emph{min-summable}: 
\begin{definition}
An age-dependent branching process with generation sizes $(Z_n)_{n \geq 0}$ and life-lengths $(X_{i_n}^n)_{i_n \in \{1, \cdots, Z_n \}, n \geq 0}$ is said to be \emph{min-summable}, if conditioned on survival,
\begin{equation}
\sum_{n=1}^{\infty} \min(X_1^n , \cdots , X_{Z_n}^n) < \infty
\label{eq::MIN_SUMMABLE}
\end{equation}
holds almost surely.
\end{definition}
\begin{remark}
Since, conditioned on survival, min-summability is a tail-event, by Kolmogorov's 0-1 law, (\ref{eq::MIN_SUMMABLE}) happens with probability either 1 or 0. It thus suffices to demonstrate that (\ref{eq::MIN_SUMMABLE}) holds with positive probability if one wants to show that a given process is \emph{min-summable}.
\end{remark}
\noindent
Amini et al. show, in \citep{AmDeGrNe13}, that a process, having a so-called \emph{plump} offspring distribution as in \eqref{eq::AGE_PLUMP}, is explosive if and only if it is min-summable. Recall that we gave an intuitive sketch of their proof in Section 1.8. 
Here we cite a simplified version of their theorem: 
\begin{theorem}
(Amini et al.) Let $h$ be the probability generating function of a distribution $D$ that is plump. Let $G$ be the distribution function of a non-negative random variable. Then, the process $\HG$ is explosive if and only if it is min-summable. 
 A necessary and sufficient condition for the latter is that for the inverse of $G$, $G^{\leftarrow}$, it holds that
\begin{equation}
\sum_{n=1}^{\infty} G^{\leftarrow} \left(\frac{1}{f(n)}\right) < \infty,
\end{equation}
where $f: \mathbb{N} \to (0,\infty)$ is defined by
\begin{equation*}
f(0) = m_0 \mbox{ and } f(n+1) = F^{\leftarrow}_D(1 - \frac{1}{f(n)}), n \geq 0,
\end{equation*}
with $m_0 > 1$ large enough such that (\ref{eq::AGE_PLUMP}) holds for all $m \geq m_0$.
\label{th::AGE_PLUMP}
\end{theorem}
\noindent
\begin{remark}
We emphasize that for a plump distribution explosiveness and min-summability of a process are equivalent. This is the essence of the theorem, which furthermore provides us with a way of checking min-summability, though proving min-summability by other means establishes explosiveness as well. 
\end{remark}
\noindent
By virtue of Theorem \ref{th::AGE_PLUMP}, it may be feasible to prove explosiveness of a process by studying its generation sizes. A very detailed description of the growth in the corresponding Galton-Watson process is given in \citep{Dav78} for a $\HG$ process, where $h$ is the probability generating function of a distribution $D$ that satisfies \eqref{eq::HEAVY_TAIL}. Note that $\Ha$ is of this form for all $\alpha \in (0,1)$, see below. 

\begin{theorem}
(Davies) Let $D$ be a random variable that satisfies (\ref{eq::HEAVY_TAIL}) for some  $c,C > 0, \alpha \in (0,1)$  and all sufficiently large $x$. Consider the ordinary Galton-Watson process where each individual produces an offspring equal to an independent copy of $D$. Denote the generation-sizes by $Z_n$ for each $n \in \mathbb{N}$.  Then, there exists a random variable $W$ such that 
\begin{equation}
\alpha^n \log (Z_n + 1) \overset{a.s.}{\rightarrow} W, \mbox{ as } n \to \infty.
\label{eq::DAVIES_GROWTH}
\end{equation}
Further, if $J$ denotes the distribution function of $W$, then
\begin{equation}
\lim_{x \to \infty} \left\{ \frac{-\log ( 1 - J(x) )}{x} \right\} = 1,
\label{eq::DAVIES_W}
\end{equation}
and,
\begin{equation*}
J(0) = \mathbb{P}( W = 0) = q,
\end{equation*}
where $q$ is the extinction-probability of the process.
\label{th::DAVIES} 
\end{theorem}

\noindent
Thus, loosely speaking, $Z_n + 1 \simeq e^{\frac{W}{\alpha^n}}$ for large $n$ and $W$ has exponential-tail behaviour (i.e., $\P{W = \infty} = 0$).

\subsection{Results}
In this section we shall first extend the examples studied by Grey. We then prove Theorems  \ref{th::AGE_allAlpha}, \ref{th::AGE_allAlphaL} and \ref{thm::ALPHA_GROWTH}. After that, we present some more results dealing with the explosiveness of age-dependent branching processes. 

\subsubsection{Re-analysis on the examples}
With the tools from the previous section at hand we find ourselves in a position to re-analyse the examples put forward by Grey. Let us first demonstrate that $h_{\alpha}$ ($\alpha \in (0,1)$) is the probability generating function of a plump distribution $D$. To this end, write, for $s \in [0,1]$,
\begin{eqnarray*}
h_{\alpha}(s) &=& 1 - (1-s)^{\alpha} \\
&=& 1- \sum_{n=0}^{\infty} {\alpha \choose n} (-s)^{\alpha} \\
&=& \alpha \sum_{n=1}^{\infty} \frac{(n-1-\alpha) \cdots (1-\alpha)}{n!} s^n,
\end{eqnarray*}
the last quantity equals
\[ \frac{\alpha}{\Gamma(1-\alpha)} \sum_{n=1}^{\infty} \frac{\Gamma(n-\alpha)}{\Gamma(n+1)} s^n,\]
where $\Gamma$ is the gamma-function defined for $t \geq 0$ by $\Gamma(t):= \int_0^\infty \e^{-x} x^{t-1} \mathrm d x$. For large $n$, $\frac{\Gamma(n-\alpha)}{\Gamma(n+1)} = (n - \alpha)^{-\alpha-1}(1 + \mathcal{O}(\frac{1}{n - \alpha}))$, see e.g. \cite{GrRy07}, so that
\begin{eqnarray*}
\mathbb{P}(D \geq k) &\geq&  \frac{\alpha}{\Gamma(1-\alpha)} \int_k^{\infty} \left( (x-\alpha)^{-\alpha - 1} + \mathcal{O}(x-\alpha)^{-\alpha - 2} \right) \mathrm d x \\
&=& \frac{(k-\alpha)^{-\alpha}}{\Gamma(1-\alpha)} + \mathcal{O}(k-\alpha)^{-\alpha - 1}, 
\end{eqnarray*}
and similarly,
\begin{equation*}
\mathbb{P}(D \geq k) \leq \frac{(k-\alpha)^{-\alpha}}{\Gamma(1-\alpha)} + \mathcal{O}(k-\alpha)^{-\alpha - 1}.
\end{equation*}
Hence, for large $k$,
\begin{equation*}
\mathbb{P}(D \geq k) = L(k) \frac{(k-\alpha)^{-\alpha}}{\Gamma(1-\alpha)},
\end{equation*}
with $\displaystyle \lim_{k \to \infty} L(k) = 1$, from which we conclude that for those $k$,
\eqn{\frac{c_{\alpha}}{k^{\alpha}} \leq \mathbb{P}(D \geq k) \leq \frac{C_{\alpha}}{k^{\alpha}}, }
for some $0 < c_{\alpha} \leq C_{\alpha}$ and we may assume $c_{\alpha} < 1$. For $F_D^{\leftarrow}$ we thus have the estimates
\eqn{ c_{\alpha} y^{1/ \alpha} \leq F_D^{\leftarrow}(1-\frac{1}{y}) \leq  C_{\alpha} y^{1/ \alpha},
\label{eq::AGE_ALPHA_D_INV}}
for large $y$. Thus, $D$ is plump. 
\\
We next investigate min-summability of $(h,\widehat{G}_{k,\gamma})$. Take $m_0 > c_{\alpha}^{1/(\alpha-1)}$ so large that (\ref{eq::AGE_ALPHA_D_INV}) holds for all $y \geq m_0$ and define $f$ as in Theorem \ref{th::AGE_PLUMP}. More precisely,  $f(0) = m_0$ and 
\[ f(n) \geq c_{\alpha}^{(1-\alpha^{-n})/(\alpha-1)} m_0^{1/\alpha^n} = c_{\alpha}^{1/(\alpha-1)} \left( \frac{m_0}{c_{\alpha}^{1/(\alpha-1)}} \right)^{\frac{1}{\alpha^n}}.\]
Put $m = \frac{m_0}{c_{\alpha}^{1 / \lr{\alpha-1}}} > 1$, then 
\be f(n) \geq m^{1/ \alpha^n}  . \label{eq::AGE_EXAMPLE_lower} \ee
Similarly, for some $ \widehat{m} < \infty$,  
\be f(n) \leq \widehat{m}^{1/{\alpha^n}}, \label{eq::AGE_EXAMPLE_upper} \ee
 for all $n$. We are now ready to continue with Example \ref{ex::GREY}. The inverse function $\widehat{G}_{k,\gamma}^{\leftarrow}$ is for small $t$ given by  \[\widehat{G}_{k,\gamma}^{\leftarrow}(t) = \left(\frac{k}{\log(\log(1/t))}\right)^{1/{\gamma}}.\] 
Now,
\[ \sum_{n=1}^{\infty} \widehat{G}_{k,\gamma}^{-1} \left(\widehat{m}^{{-1}/{\alpha^n}}\right) \leq \sum_{n=1}^{\infty} \widehat{G}_{k,\gamma}^{-1} \left( \frac{1}{f(n)} \right)  \leq  \sum_{n=1}^{\infty} \widehat{G}_{k,\gamma}^{-1} \left(m^{{-1}/{\alpha^n}}\right).  \]
Both estimators are of the form
\[\s{n=1}{\infty} \left(\frac{k}{\log(1/{\alpha}) + \frac{\log(\log(m))}{n}}\right)^{1/{\gamma}} \frac{1}{n^{1/ \gamma}}, \]
and
\[\s{n=1}{\infty} \left(\frac{k}{\log(1/{\alpha}) + \frac{\log(\log(\widehat{m}))}{n}}\right)^{1/{\gamma}} \frac{1}{n^{1/ \gamma}}, \]
both are finite if and only if $\gamma < 1$. Hence (Theorem \ref{th::AGE_PLUMP}), the process is explosive if and only if $\gamma < 1$, \emph{independent} of $\alpha$ or $k$.

\subsubsection{Theorems}
 We use the growth characterization put forward by Davies (Theorem \ref{th::DAVIES}) to prove Proposition \ref{thm::ALPHA_GROWTH} that
characterizes all processes $\HG$, where $h$ is the probability generating function of a distribution $D$ that has the form (\ref{eq::HEAVY_TAIL}):

\begin{proof}[Proof of Proposition \ref{thm::ALPHA_GROWTH}]
Let the process be explosive. Fix $\epsilon \in (0,1)$. An appeal to Theorem \ref{th::DAVIES} gives us $W$ such that 
$\alpha^n $log$ (Z_n + 1) \overset{a.s.}{\rightarrow} W$,  as $n \to \infty$, where $Z_n$ represents the size of generation $n$. Let $N$ be a random variable such that $Z_n \leq e^{\frac{W + \epsilon}{\alpha^n}}$ for all $n \geq N$. Then since $N < \infty$ by the a.s. convergence in Theorem \ref{th::DAVIES} (denoting probability conditional on survival by $\mathbb{P}_s$),
\be\ba
\mathbb{P}_s &\Big( \sum_{k=1}^{\infty} \min(X_1^k , \cdots , X^k_{Z_k}) < \infty \Big)\! =\! \sum_{n=0}^{\infty} \!\mathbb{P}_s(N \!=\! n)\mathbb{P}_s \Big( \sum_{k=1}^{\infty} \min(X_1^k , \cdots , X^k_{Z_k}) < \infty \big| N \!=\!n \Big) \\
&\qquad\qquad\qquad\qquad\leq \sum_{n=0}^{\infty} \mathbb{P}_s(N \!=\! n) \mathbb{P}_s \Big( \sum_{k=n}^{\infty} \min(X_1^k , \cdots , X_{\e^{\frac{W+\epsilon}{\alpha^k}}}^k) < \infty \big| N \!=\!n \Big), \\
\ea\ee
since the latter minimum is taken over a larger group of i.i.d. random variables.  From here we proceed by further conditioning on $W$ (which is finite almost surely, as can be seen from \eqref{eq::DAVIES_W}), to rewrite the last term as 
\[ \sum_{n=0}^{\infty} \mathbb{P}_s(N \!=\! n)  \sum_{l=0}^{\infty} \mathbb{P}_s(W \!\in\! [l,l+1)|N \!=\! n)  \mathbb{P}_s \Big( \sum_{k=n}^{\infty} \min(X_1^k \!, \cdots ,\! X_{e^{\frac{W+\epsilon}{\alpha^k}}}^k)\! < \!\infty | W\! \in\! [l,l+1), N\! =n\! \Big), \]
which we can bound by 
\begin{equation}
\sum_{n=0}^{\infty} \mathbb{P}_s(N \!=\! n)  \sum_{l=0}^{\infty} \mathbb{P}_s(W \!\in\! [l,l+1)|N \!=\! n)  \mathbb{P}_s \Big( \sum_{k=n}^{\infty} \min(X_1^k \!, \cdots ,\! X_{e^{\frac{l+1+\epsilon}{\alpha^k}}}^k)\! < \!\infty | W\! \in\! [l,l+1), N\! =n\! \Big),
\label{eq::MS_1}
\end{equation} 
since in each individual term $W \leq l + 1$. To proceed further, we note that for all $l$ there exists $K_l$ such that $e^{1/\alpha^{K_l}} \geq e^{l + 1+\epsilon}$ and thus for all $k \geq 0$,  
\[ e^{1/\alpha^{K_l + k}} \geq e^{(l + 1+\epsilon) / \alpha^k},\] which tells us that 
\[ \sum_{k=0}^{\infty} \mbox{min}(X_1^k , \cdots , X_{e^{\frac{l+1+\epsilon}{\alpha^k}}}^k) \overset{d}\geq \sum_{k=K_l}^{\infty} \mbox{min}(X_1^k , \cdots , X_{e^{\frac{1}{\alpha^k}}}^k). \]  
But, because $K_l < \infty$, 
\[ \sum_{k=0}^{K_l} \mbox{min}(X_1^k , \cdots , X_{e^{\frac{1}{\alpha^k}}}^k) <\infty \mbox{ a.s., } \] 
and thus
\[ \P{ \sum_{k=n}^{\infty} \mbox{min}(X_1^k , \cdots , X_{e^{\frac{l+1+\epsilon}{\alpha^k}}}^k) < \infty | W\! \in\! [l,l+1), N\! =n\! } \leq \P{ \sum_{k=0}^{\infty} \mbox{min}(X_1^k , \cdots , X_{e^{\frac{1}{\alpha^k}}}^k) < \infty }. \]
Note that the sum in the left probability is independent of both $W$ and $N$, since the life-lengths are independent of the generation sizes.
Hence, (\ref{eq::MS_1}) is bounded by
\begin{multline*}
\sum_{n=0}^{\infty} \mathbb{P}_s(N\! =\! n)  \sum_{l=0}^{\infty} \mathbb{P}_s(W \!\in\! [l,l+1)|N \!=\! n)  \mathbb{P} \Big( \sum_{k=0}^{\infty} \mbox{min}(X_1^k \!, \cdots ,\! X_{\e^{\frac{1}{\alpha^k}}}^k)\! <\! \infty  \Big) \\
 = \mathbb{P} \Big( \sum_{k=0}^{\infty} \mbox{min}(X_1^k , \cdots , X_{e^{\frac{1}{\alpha^k}}}^k) < \infty \Big),
\end{multline*}
 We conclude that 
\[ \mathbb{P} \Big( \sum_{k=0}^{\infty} \mbox{min}(X_1^k , \cdots , X_{e^{\frac{1}{\alpha^k}}}^k) < \infty  \Big) \geq \mathbb{P} \Big( \Big. \sum_{k=1}^{\infty} \mbox{min}(X_1^k , \cdots , X^k_{Z_k}) < \infty  \Big| \mbox{survival} \Big) > 0, \]
which is all we need since Kolmogorov's $0-1$ law says that the event under consideration happens with probability either $0$ or $1$.
\\
To prove the other direction, we note that (\ref{eq::DAVIES_W}) entails that $\mathbb{P} ( W > 1) > 0$ which we will use shortly. Fix $\epsilon < 1$ and let this time $N$ be such that $Z_k \geq e^{\frac{W - \epsilon}{\alpha^k}}$ for all $k \geq N$.
\begin{multline*}
\mathbb{P} \left( \left. \sum_{k=0}^{\infty} \mbox{min}(X_1^k , \cdots , X^k_{Z_k}) < \infty \right| \mbox{survival}  \right) \P{\mbox{survival}} \\ \geq \sum_{n=0}^{\infty} \mathbb{P}(N = n) \mathbb{P}(W > 1 | N = n)\mathbb{P} \left( \left. \sum_{k=n}^{\infty} \mbox{min}(X_1^k , \cdots , X^k_{Z_k}) < \infty \right| W > 1, N =n \right), \\
\end{multline*}
because $W > 1$ already implies survival.
Since $Z_k$ dominates $e^{\frac{W - \epsilon}{\alpha^k}}$ whenever $k \geq N$, the last term may be bounded from below by
\begin{multline*}
\sum_{n=0}^{\infty} \mathbb{P}(N = n) \mathbb{P}(W > 1 | N = n)\mathbb{P} \left( \left. \sum_{k=n}^{\infty} \mbox{min}(X_1^k , \cdots , X^k_{e^{\frac{W - \epsilon}{\alpha^k}}}) < \infty \right| W > 1, N =n \right) \\
= \sum_{n=0}^{\infty} \mathbb{P}(N = n) \mathbb{P}(W > 1 | N = n)\mathbb{P} \left( \sum_{k=0}^{\infty} \mbox{min}(X_1^k , \cdots , X^k_{e^{\frac{1}{\alpha^k}}}) < \infty \right). 
\end{multline*} 
Hence, 
\[ \mathbb{P} \left( \left. \sum_{k=0}^{\infty} \mbox{min}(X_1^k , \cdots , X^k_{Z_k}) < \infty \right| \mbox{survival} \right) \geq \frac{\mathbb{P}(W > 1)}{\P{\mbox{survival}}} \mathbb{P} \left( \sum_{k=0}^{\infty} \mbox{min}(X_1^k , \cdots , X^k_{e^{\frac{1}{\alpha^k}}}) < \infty \right), \]
which concludes the proof as the right-side is strictly positive. 
\end{proof}

\noindent
We will now prepare for the proof of Theorem \ref{th::AGE_allAlpha} that heavily relies on the result of Proposition \ref{thm::ALPHA_GROWTH}.
\begin{lemma}
Let $\alpha \in (0,1)$. Let $G$ be the distribution function of a non-negative random variable. The process $\HGa$ is explosive if and only if $(h_{\alpha^{1/p}},G)$ is explosive for all $p \in \mathbb{N}$.
\label{lm::AGE_EXPL_IND_ALPHA}
\end{lemma}
\begin{proof}
If $p \in \mathbb{N}$, then $\alpha^{1/p} \geq \alpha$. Hence, the comparison theorems confirm that $(h_{\alpha^{1/p}},G)$ being explosive implies that $\HGa$ is explosive. Indeed, since for $s$ close to $1$, $h_{\alpha^{1/p}}(s) \geq h_{\alpha}(s)$, and Theorem \ref{thm::AGE_COMP_h} applies. 
\\
To prove the other direction, we assume that $\HGa$ is explosive. We introduce the notation 
\[f(n) = e^{{1}/{\alpha^n}},\]
and 
\[f_p(n) = e^{1/{\alpha^{{n}/{p}}}},\] for $n \in \mathbb{N}.$
Then, almost surely,
\[ \sum_{k=0}^{\infty} \mbox{min}(X_1^k , \cdots , X^k_{f(k)}) < \infty.\] 
 We shall show that this implies that 
\[\sum_{k=0}^{\infty} \mbox{min}(X_1^k , \cdots , X^k_{f_p(k)}) < \infty,\]
 almost surely. 
To this end, we note that $f_p(pn) = f(n)$ for all  $n \in \mathbb{N}$, which tells us that we may split the last sum into individuals parts,
\[ \sum_{k=0}^{\infty} \mbox{min}(X_1^k , \ldots , X^k_{f_p(k)}) = \sum_{l=0}^{\infty} \sum_{k=lp}^{(l+1)p - 1} \mbox{min}(X_1^k , \ldots , X^k_{f_p(k)}), \]
and bound each sub term from above. Indeed, whenever $k \in \{ lp, \ldots, (l+1)p -1 \}$,
\begin{eqnarray*}
\mbox{min}(X_1^k , \cdots , X^k_{f_p(k)}) &\leq& \mbox{min}(X_1^k , \ldots , X^k_{f_p(lp)}) \\
&=& \mbox{min}(X_1^k , \cdots , X^k_{f(l)}),
\end{eqnarray*}
by the monotonicity of $f_p$.  
Hence,
\[\ba \sum_{k=0}^{\infty} \mbox{min}(X_1^k , \cdots , X^k_{f_p(k)}) &\leq \sum_{l=0}^{\infty} \sum_{k=lp}^{(l+1)p - 1} \mbox{min}(X_1^k , \cdots , X^k_{f(l)})
\\ & \overset{d} = \sum_{i=0}^{p} M_i , \ea \]
which is almost surely finite , where $(M_i)_{i=1}^p$ are i.i.d. copies of $\sum_{k=0}^{\infty} \mbox{min}(X_1^k , \cdots , X^k_{f(k)})$, a sequence that is almost surely finite.
\end{proof} 

\noindent
Now, Lemma \ref{lm::AGE_EXPL_IND_ALPHA} together with the comparison theorems is all we need to prove Theorem \ref{th::AGE_allAlpha}:
\begin{proof}[Proof of Theorem \ref{th::AGE_allAlpha}]
Let $1 > \beta > \alpha > 0$ and assume first that $\HaG$ is explosive. Fix now $p$ so large that $\alpha^{{1}/{p}} \geq \beta$. Note that $(h_{\alpha^{{1}/{p}}},G)$ is explosive by Lemma \ref{lm::AGE_EXPL_IND_ALPHA} and hence, by the comparison theorems, $(h_{\beta},G)$ is explosive. Indeed, since for $s$ close to $1$, $h_{\alpha^{1/p}}(s) \geq h_{\beta}(s)$, and Theorem \ref{thm::AGE_COMP_h} applies.
\\
Next, assume that $(h_{\beta},G)$ is explosive. We have, for $s$ close to $1$, $h_{\beta}(s)\geq h_{\alpha}(s)$ and comparison Theorem  \ref{thm::AGE_COMP_h} establishes explosiveness of $\HaG$.   %All other assertions follow upon another appeal to the comparison theorems.
\end{proof}
\noindent Theorem \ref{th::AGE_allAlphaL} follows now easily:
\begin{proof}[Proof of Theorem \ref{th::AGE_allAlphaL}.]
By Potter's theorem (see, e.g. \cite{BiGoTe87}), for all $\epsilon > 0$ and $s$ near 1,
\[ (1-s)^{\epsilon} \leq L(\frac{1}{1-s}) \leq (1-s)^{-\epsilon}, \]
hence,
\[ 1 - (1-s)^{\alpha - \epsilon} \leq h^L_{\alpha}(s) \leq 1-(1-s)^{\alpha + \epsilon}, \label{eq::AGE_POTTER} \]
and the comparison theorems finish the proof.
\end{proof}

\noindent
We finish this section by listing two more properties of age-dependent branching process.  We shall refer to them at multiple occasions throughout this thesis.

\begin{theorem}
Let $h$ be a probability generating function for an offspring distribution $D$ that satisfies \eqref{eq::HEAVY_TAIL}
for some  $c,C > 0, \alpha \in (0,1)$  and all sufficiently large $x$. Let $G$ be the distribution of a non-negative random variable. If the process $\HG$ is explosive, then $(h,G^n)$ is explosive for all $n > 0$, where the function $G^n$ is defined for $x \geq 0$ by $G^n (x) = (G(x))^n$. 
\label{th::AGE_POWER}
\end{theorem}
\begin{proof}
Put $H=G^n$ and note that, for small $t$, $H^{\leftarrow}(t) = G^{\leftarrow}(t^{1/n})$, where we recall the notation for inverse functions. We will employ Theorem \ref{th::AGE_PLUMP} to show that $(h,G^n)$ is explosive. That same theorem, together with the re-analysis on Grey's example, entails estimates on $f$ (equations \eqref{eq::AGE_EXAMPLE_lower} and \eqref{eq::AGE_EXAMPLE_upper}), defined by \[f(0) = m_0 \mbox{ and } f(n+1) = F^{-1}_D(1 - \frac{1}{f(n)}), n \geq 0, \] 
combined with $G^{\leftarrow}$. Most importantly, there exists $\widehat{m} \geq m > 0$ such that 
\eqn{m^{{1}/{\alpha^k}} \leq f(k) \leq \widehat{m}^{{1}/{\alpha^k}},
\label{eq::AGE_GN_fn} }
for all $k \geq 0$, and
\eqn{\sum_{k=1}^{\infty} G^{\leftarrow} \left(\frac{1}{f(k)}\right) < \infty. 
\label{eq::AGE_GN_sum}}
We are done if we demonstrate that $\sum_{k=1}^{\infty} H^{\leftarrow} \left(\frac{1}{f(k)}\right) < \infty$. Because there is $K$ such that 
\[(m^{\frac{1}{n}})^{\frac{1}{\alpha^K}} \geq \widehat{m},\] 
we note that, for all $l \geq 0$, 
\[f(K+l)^{\frac{1}{n}} \geq (m^{\frac{1}{n}})^{1/\alpha^{K+l}} \geq \widehat{m}^{\frac{1}{\alpha^l}}.\]
Hence,
for all $l \geq 0$,
\[  H^{\leftarrow} \left(\frac{1}{f(K+l)}\right) =  G^{\leftarrow} \left(\frac{1}{f(K+l)^{\frac{1}{n}}}\right) \leq G^{\leftarrow} \left(\frac{1}{\widehat{m}^{\frac{1}{\alpha^l}}}\right) \leq G^{\leftarrow} \left(\frac{1}{f(l)}\right), \]
due to the upper bound (\ref{eq::AGE_GN_fn}). Evidently,
\begin{eqnarray*}
\sum_{k=0}^{\infty} H^{\leftarrow} \left(\frac{1}{f(k)}\right) &=& \s{k=0}{K-1} H^{\leftarrow} \left(\frac{1}{f(k)}\right) + \s{l=0}{\infty} H^{\leftarrow} \left(\frac{1}{f(K+l)}\right) \\
&\leq& \s{k=0}{K-1} H^{\leftarrow} \left(\frac{1}{f(k)}\right) + \s{l=0}{\infty} G^{\leftarrow} \left(\frac{1}{f(l)}\right) < \infty,
\end{eqnarray*}
as follows from (\ref{eq::AGE_GN_sum}). This establishes the result.
\end{proof}
\begin{remark}
Note that $G^n$ is the distribution function of $\max\{X_1, \ldots, X_n\}$.
\end{remark}
\begin{definition}
Since the explosiveness of a process is completely determined by its behaviour around the origin, we shall mean by  $cG$ a \emph{distribution} that is around the origin equal to $cG$ and is extended in an arbitrary way so to make it a proper distribution. 
\end{definition}
\begin{theorem}
Let $\alpha \in (0,1)$. Let $G$ be the distribution of a non-negative random variable. If the process $(\Ha,G)$ is explosive, then so is $(\Ha,cG)$ for all $c>0$.
\label{th::AGE_CONSTANT}
\end{theorem}
\begin{proof}
We consider only $0 < c < 1$, because the case $c \geq 1$ follows from the comparison theorems. There exists $\phi: [0,\infty] \mapsto [0,1]$ such that, for $t \geq 0$, 
\[\phi(t) = h\left( 1-G(t) + \int_0^t \phi(t-u) \mathrm d G(u)\right), \quad t \geq 0,\]
and for $t > 0$, $\phi(t) < 1$.
Put $\eta = 1 - \phi \neq 0$, then for $t \geq 0$,
\[ \eta(t) = \lr{\int_0^t \eta(t-u) \mathrm d G(u)}^{\alpha}. \] 
We need to find a function $\widehat{\eta} \neq 0$ such that, for $t \geq 0$
\[ \widehat{\eta}(t) = \lr{ \int_0^t \widehat{\eta}(t-u) \  c \ \mathrm d G(u) }^{\alpha}. \]
If $\widehat{\eta} = A \eta$, where $A = c^{\alpha/(1-\alpha)}$, then
\[ \lr{ \int_0^t \widehat{\eta}(t-u) c \mathrm d G(u) }^{\alpha}= A^{\alpha} c^{\alpha} \lr{\int_0^t \eta^{\alpha}(t-u) \mathrm d G(u)}^{\alpha}  =  A^{\alpha} c^{\alpha} \eta(t). \]
But $\eta = A^{-1} \widehat{\eta}$ and thus $A^{\alpha} c^{\alpha} \eta = A^{\alpha - 1} c^{\alpha} \widehat{\eta}= \widehat{\eta},$ hence
\[ \lr{\int_0^t \widehat{\eta}(t-u) c \mathrm d G(u)}^{\alpha} = \widehat{\eta}(t). \]
Explosiveness follows because, for $t > 0$, $\eta(t) \geq \widehat{\eta}(t) = A \eta(t) >0$.
\end{proof}

\newpage

\section{Forward Age-Dependent Branching Processes with Contagious Periods}
It is with the more elaborate model $\HGCf$ that this section will be principally concerned. Here, $h$ is a probability generating function and both $G$ and $C$ are distribution functions of non-negative random variables. 
\\
Define $F: [0,1] \times [0,\infty] \mapsto [0,1]$ by $F(s,t) = \E{s^{N^f(t)}}$. Again, we will derive a fixed point equation for $F$, which captures essential information regarding the explosiveness of $\HGCf$.
\\
By the branching property, conditioned on vertex $i$ being born, each sub tree starting at time $X_i$ has the same distribution as the original branching process, evaluated at time $t-X_i$ and these sub trees are i.i.d..
We have
\eqn{ N^f(t) = \sum_{i=1}^D \lr{ \indicator{ t<X_i< C} + \indicator{ X_i < \min(C, t)} N^{(i)}(t-X_i) },}
where $C$ is the contagious period of the root and $(N^{(i)})_i$ are i.i.d. copies of $N^f(t)$.
Then,
\[ F(s,t) =\sum_{k=1}^{\infty} \P{D=k} \int_{0}^\infty \Ev\left[ s^{\indicator{ t<X_i<x} + \indicator{ X_i < \min(t,x) } N^{(i)}(t-X_i)}\right]^k \mathrm d C(x).\]
Let us split the integral at $t$ to have
\[ \ba F(s,t) &=\sum_{k=1}^{\infty} \P{D=k} \int_0^t \Ev\left[ s^{\indicator{ X_i < x } N^{(i)}(t-X_i)}\right]^k \mathrm d C(x)\\
&\quad +\sum_{k=1}^{\infty} \P{D=k} \int_t^\infty \Ev\left[ s^{\indicator{t<X_i<x}+\indicator{ X_i < t} N^{(i)}(t-X_i)}\right]^k \mathrm d C(x)
.\ea\]
Now we integrate also over $X_i$, paying particular attention to the fact that if the indicators are not satisfied then we have a factor $1$, giving rise to the last terms in each line:
 \[ \ba F(s,t) &=\sum_{k=1}^{\infty} \P{D\!=\!k} \int_0^t \left(\int_0^x \Ev\left[ s^{N(t-u)}\right]\mathrm d G(u) + 1-G(x)\right)^k \mathrm d C(x)\\
&\quad +\sum_{k=1}^{\infty} \P{D\!=\!k} \int_t^\infty \left( \int_0^t \Ev\left[ s^{ N(\!t-u\!)} \right] \mathrm d G(u)\! + \!s (G(x)\!-\!G(t)) \!+\! (1\!-\!G(x))\right)^k \mathrm d C(x). \ea\]
Now using the generating function $h$,
\[ \ba F(s,t)&=  \int_0^t h\left(\int_0^x \Ev\left[ s^{N(t-u)}\right]\mathrm d G(u) + 1-G(x)\right) \mathrm d C(x) \\
&\quad +\int_t^\infty h \left( \int_0^t \Ev\left[ s^{ N(t-u)}\right] \mathrm d G(u) + s (G(x)-G(t)) + (1-G(x))\right) \mathrm d C(x)
.\ea\]
Writing $\Ev\left[ s^{ N(t-u)}\right]=F(s,t-u),$ we then have
  \[ \ba F(s,t) &=\int_0^t h \left(\int_0^x F(s,t-u)\mathrm d G(u) + 1-G(x)\right) \mathrm d C(x) \\
&\quad +\int_t^\infty  h \left( \int_0^t F(s, t-u) \mathrm d G(u) + s (G(x)-G(t)) + (1-G(x))\right) \mathrm d C(x).
\ea\]
Define $\phi^f: [0,\infty] \mapsto [0,1]$ for $t \geq 0$ by $\phi^f(t) = \displaystyle \lim_{s \uparrow 1} F(s,t)$, then
\be \ba \phi^f(t) &=\int_0^t h \left(\int_0^x \phi^f(t-u)\mathrm d G(u) + 1-G(x)\right) \mathrm d C(x)\\
&\quad + h \left( \int_0^t \phi^f(t-u) \mathrm d G(u) + 1-G(t)\right) (1-C(t)), \quad t \geq 0.\ea\ee

\subsection{Fixed point equation}
For essentially the same reason as set out in the previous section, we shall study the fixed point equation
\be \Phi = \TCf \Phi, \quad \Phi \in \mathcal{V}, \label{eq::CP_FPE_MAIN} \ee
where $\mathcal{V}$ is defined in \eqref{eq::function_space} and
$\TCf$ is for $\Phi \in \mathcal{V}$ defined as 
\be \ba (\TCf \Phi)(t) &= \int_0^t h \left( \int_0^x \Phi(t-u)\mathrm d G(u) + 1 - G(x) \right) \mathrm d C(x) \\
 & \quad +  h \left( \int_0^t \Phi(t-u)\mathrm d G(u) + 1 - G(t)\right)(1-C(t)), \quad t \geq 0. \ea \ee
We will construct an iterative solution $\phi$ to (\ref{eq::CP_FPE_MAIN}) and show that this solution satisfies
\begin{enumeratei}
\item  $\phi$ is the smallest solution to (\ref{eq::CP_FPE_MAIN}),
\item  $\phi(t) = 1 -  \P{ N^f(t) = \infty }$ for all $ t \in [0,\infty)$,
\end{enumeratei}
so that $\phi$ is in fact equal to $\phi^f$ (which embraces that the existence of any solution to (\ref{eq::CP_FPE_MAIN}), not identically equal to 1, guarantees that the process is explosive). Hereafter, we interchangeably use $\phi$ and $\phi^f$. Note that $\phi$ characterizes the distribution of the explosion time, see \eqref{eq::CP_FW_EXPL_TIME}, since for $t \geq 0$, $\phi(t) = \P{V^f \geq t}$. We proceed by arguing that a much weaker condition is in fact sufficient to prove that a process is explosive. Then we turn back to the iterative solution and note that property (ii) implies that $\phi$ should be non-increasing, which we also prove by other means. We further investigate the behaviour of $\phi$ by showing that for an explosive process $\phi < 1$ on all of $(0,\infty)$ and that for such a process, conservative survival does not happen. Some of the proofs in the coming section lean on similar results in \cite{Grey74}, where an age-dependent process with \emph{independent} edges is studied.

\subsubsection{An iterative solution}
 
In this section we show that the distribution function of the explosion time can be obtained as a sequential limit. To construct an iterative solution to (\ref{eq::CP_FPE_MAIN}), we define $\phi_0 \equiv 0$ and for each $k \geq 0$, we recursively define $\phi_k$, for $t \geq 0$ by
\be
\ba
\label{eq::it}
\phi_{k+1}(t) &= \lr{\TCf \phi_k}(t) \\
&= \int_0^t h\left( \int_0^x \phi_k(t-u)\mathrm d G(u) + 1 - G(x) \right)\mathrm d C(x) \\
 & \quad + h\left( \int_0^t \phi_k(t-u)\mathrm d G(u) + 1 - G(t)\right)(1-C(t)).
\ea
\ee
\begin{theorem}\label{thm::CP_FPE_IT_SMALLEST} %Cont. Per. Fix. Pnt. Eq.Iterative solution is Smallest solution
$\phi$ defined for $t \geq 0$ by $\phi(t) = \displaystyle \lim_{k \to \infty} \phi_k(t) $ is the smallest solution to (\ref{eq::CP_FPE_MAIN}), in the sense that if $\psi$ is any other non-negative solution, then $\phi(t) \leq \psi(t) $ for all $t \geq 0$.
\end{theorem}

\begin{proof}[Proof of Theorem \ref{thm::CP_FPE_IT_SMALLEST}]
The aim here is to show that both sides of (\ref{eq::it}) converge (for $k \to \infty$) point wise to their respective counterparts in (\ref{eq::CP_FPE_MAIN}), which proves that $\phi$ solves (\ref{eq::CP_FPE_MAIN}). First, notice that for fixed $t \geq 0$, $(\phi_k(t))_{k \in \mathbb{N}}$ is an non-decreasing sequence, bounded from above by 1. Indeed, $\phi_1(t) \geq 0$ and if we assume that, for some $k$, $\phi_{k}(t) \geq \phi_{k-1}(t)$, then 

\begin{equation}
\begin{split}
\phi_{k+1}(t) & =  \int_0^t h \lr{ \int_0^x \phi_k(t-u)\mathrm d G(u) + 1 - G(x)}\mathrm d C(x) \\ 
& \quad +  h \lr{ \int_0^t \phi_k(t-u)\mathrm d G(u) + 1 - G(t)}(1-C(t)) \\
& \geq   \int_0^t h \lr{ \int_0^x \phi_{k-1}(t-u)\mathrm d G(u) + 1 - G(x)}\mathrm d C(x) \\
& \quad + h \lr{ \int_0^t \phi_{k-1}(t-u)\mathrm d G(u) + 1 - G(t)}(1-C(t)) \\
& =  \phi_k(t).
\end{split}
\end{equation}
This induction argument assures us that $(\phi_k(t))_{k \in \mathbb{N}}$ is indeed non-decreasing for fixed $t \geq 0$. To see that $(\phi_k(t))_{k \in \mathbb{N}}$ is bounded, we prove that it is point wise dominated by any positive solution $\psi$ to (\ref{eq::CP_FPE_MAIN}). For $\phi_0 \equiv 0$ this is evident and if it holds for some $k\geq 0$, then

\be \ba
\phi_{k+1}(t) &= \int_0^t h \lr{ \int_0^x \phi_k(t-u)\mathrm d G(u) + 1 - G(x)}\mathrm d C(x) \\
& \quad + h \lr{ \int_0^t \phi_k(t-u)\mathrm d G(u) + 1 - G(t)}(1-C(t)) \\
&\leq  \int_0^t h \lr{ \int_0^x \psi(t-u)\mathrm d G(u) + 1 - G(x)}\mathrm d C(x) \\
& \quad + h \lr{ \int_0^t \psi(t-u)\mathrm d G(u) + 1 - G(t)}(1-C(t)) \\
&= \psi(t). \ea \ee
It now merely remains to verify that $\psi \equiv 1$ solves (\ref{eq::CP_FPE_MAIN}). Since $(\phi_k(t))_{k \in \mathbb{N}}$ is bounded and non-increasing, we have established that, for all fixed $t \geq 0$, $\displaystyle \lim_{k \to \infty} \phi_k(t) = \phi(t)$ exists.  
\\
It remains to show that the right hand side (RHS) of (\ref{eq::it}) converges to the RHS of (\ref{eq::CP_FPE_MAIN}). For fixed $x \in [0,t]$, the sequence 
\[\left( \int_0^x \phi_k(t-u)\mathrm d G(u) +  1 - G(x) \right)_{k \in \mathbb{N}}\] is non-decreasing and bounded from above by one. Also, \[\phi_k(t-u) \leq \phi(t-u),\] for all $u \in [0,t]$ and as $\phi \in L^1(\mathbb{R}^+,G)$, it follows upon an appeal to Lebesque's Dominated Convergence theorem (LDC) that
\begin{equation*}
\int_0^x \phi_k(t-u)\mathrm d G(u) \to \int_0^x \phi(t-u)\mathrm d G(u), \vspace*{1 cm} \mbox{ as } k \to \infty.
\end{equation*}
Invoking the continuity of $h$ entails that, for all $x \in [0,t]$,
\begin{equation*}
h \left( \int_0^x \phi_k(t-u)\mathrm d G(u) +  1 - G(x) \right) \to h \left(\int_0^x \phi(t-u)\mathrm d G(u) +  1 - G(x)\right), \vspace*{1 cm} \mbox{ as } k \to \infty.
\end{equation*}
As $h \leq 1$ on $[0,1]$ and $t \mapsto 1  \in L^1(\mathbb{R}^+,C)$, another application of LDC establishes the proof of convergence. 
\end{proof}
%In fact, we have proven an analogue of Theorem ?? in [paperGrey].

\begin{theorem}\label{thm::CP_FPE_IT_RIGHT} %Cont. Per. Fix. Pnt. Eq. Iterative solution is right one
$\phi$ defined for $t \geq 0$ by $\phi(t) = \displaystyle \lim_{k \to \infty} \phi_k(t) $ is the right solution to (\ref{eq::CP_FPE_MAIN}), that is, $\phi(t) = 1 -  \P{ N^f(t) = \infty }$ for all $ t \in [0,\infty)$. 
\end{theorem}
\begin{proof}[Proof of Theorem \ref{thm::CP_FPE_IT_RIGHT}]
We show by induction that for all $k \geq 1$,
\be
\phi_k(t) = 1 - \mathbb{P} \{ \exists v \in \mbox{BP}(t) \mbox{ s.t. } |v| = k  \},
\ee
where the notation BP($t$) is used for the collection of \emph{all dead or alive} individuals (vertices) in the branching process born before time $t$ and $|v|$ denotes the generation of an individual $v \in$ BP($t$).
We first start analysing $\phi_1(t) = (\TCf \phi_0)(t)$. 
\begin{eqnarray*}
h(1-G(x)) &=& \E{ \P{ X > x }^D }\\
&=& \P{ X_1 > x, \cdots,X_D > x  },
\end{eqnarray*}
where $(X_i)_i$ are i.i.d. random variables with distribution $G$. As follows from \eqref{eq::CP_FPE_MAIN}, for $t \geq 0$,
\begin{eqnarray*}
\phi_1(t) &=& (\TCf \phi_0)(t) = \int_0^t h \left(  1 - G(x) \right) \mathrm d C(x) +  h \left( 1 - G(t)\right)(1-C(t)) \\
&=& \int_0^t \P{  X_1 > x, \cdots,X_D > x  } \mathrm d C(x) + \P{  X_1 > t, \cdots,X_D > t  } \P{ C>t  } \\
&=& \P{  X_1 > C, \cdots,X_D > C, C \leq t  } + \P{ X_1 > t, \cdots,X_D > t, C > t },
\end{eqnarray*}
which we may write as
\[ \phi_1(t) = 1 - \mathbb{P} \{ \exists v \in \mbox{BP}(t) \mbox{ s.t. } |v| = 1  \}.\]
We proceed by assuming that for some $k > 0$,
\begin{equation*}
\phi_k(t) = 1 - \P{  \exists v \in \mbox{BP}(t) \mbox{ s.t. } |v| = k  }.
\end{equation*}
Then,
\be \ba \label{eq:CP:it:prob} \phi_{k+1}(t) &= (\TCf \phi_k)(t) \\
&= \int_0^t h \left( 1 - \int_0^x \P{  \exists v \in \mbox{BP}(t-u) \mbox{ s.t. } |v| = k  } \mathrm d G(u) \right) \mathrm d C(x) \\
 & \quad + h \left(1 - \int_0^t \P{  \exists v \in \mbox{BP}(t-u) \mbox{ s.t. } |v| = k  } \mathrm d G(u)  \right) \P{  C>t  }.\ea \ee
We start by analysing the terms in brackets, for $x \leq t$,
\be \ba
1 - \int_0^x \mathbb{P} \{ \exists v \! \in \! \mbox{BP}(t-u) \mbox{ s.t. } |v|\! = \!k  \} \mathrm d G(u) &= 1\! - \!\mathbb{P} \{ \exists v \!\in\! \mbox{BP}(t-U) \mbox{ s.t. } |v|\! = k\! , U \!\leq \!x \} \\
&= 1 - \P{E(t,x)} \\
&= \P{\overline{E(t,x)}},
\ea \ee
where $U$ is a random variable with law $G$,
\[ E(t,x) = \{ \exists v \in \mbox{BP}(t-U) \mbox{ s.t. } |v| = k , U \leq x \}, \] 
and $\bar{A}$ is the complement of an event $A$.
Now, let $\left(U^{(i)} \right)_i$ be an i.i.d. sequence of random variables with distribution $G$, and define the independent identically distributed events $(E_j(t,x))_j$, for $j \geq 1$, as
\[ E_j(t,x) = \{ \exists v \in \mbox{BP}^{(j)}(t-U^{(j)}) \mbox{ s.t. } |v| = k , U^{(j)} \leq x \}. \]
Then
\[ h \left( \P{ \overline{ E(t,x)} } \right) = \P{  \overline{E_1(t,x)} , \ldots, \overline{E_D(t,x)} },\]
and we deduce that the first term in the RHS of (\ref{eq:CP:it:prob}) equals
\[ \int_0^t h\left( \mathbb{P} \left(  \overline{E(t,x)} \right) \right) \mathrm d C(x) =
\P{  \not\exists v \in \mbox{BP}(t) \mbox{ s.t. } |v| = k+1 , C \leq t  }.\]
The second line in \eqref{eq:CP:it:prob} becomes
\[
\ba
h \left( 1 \!-\! \int_0^t \mathbb{P} \left( \exists v\! \in\! \mbox{BP}(t-u) \mbox{ s.t. } |v|\! = \!k  \right) \mathrm d G(u)  \right) \mathbb{P} \left( C \!>\! t \right) &=
\P{   \overline{E_1}(t,t) , \ldots ,  \overline{E_D}(t,t)  , C > t }, \ea \]
or,
\[ \P{   \not\exists v \in \mbox{BP}(t) \mbox{ s.t. } |v| = k+1 , C > t  }.\]
All together,
\begin{eqnarray*}
\phi_{k+1}(t) &=& \P{  \not\exists v \in \mbox{BP}(t) \mbox{ s.t. } |v| = k+1  } \\
&=& 1 - \P{ \exists v \in \mbox{BP}(t) \mbox{ s.t. } |v| = k+1 }.
\end{eqnarray*}
To finish, note that
\begin{equation*}
\{\exists v \in \mbox{BP}(t) \mbox{ s.t. } |v| = k+1 \} \subseteq \{\exists v \in \mbox{BP}(t) \mbox{ s.t. } |v| = k \},
\end{equation*}
which constitute a nested sequence of events and therefore by the continuity of the measure $\P{\cdot}$,
\begin{equation*}
\displaystyle \lim_{k \to \infty} \phi_k(t) = \lim_{k \to \infty} \mathbb{P} \{\exists v \in \mbox{BP}(t) \mbox{ s.t. } |v| = k \} = \mathbb{P} \left( \bigcap \limits_{k=1}^{\infty} \{\exists v \in \mbox{BP}(t) \mbox{ s.t. } |v| = k \} \right).
\end{equation*}
Now, if there exists $k \in \mathbb{N}$ such that for all $v \in $ BP($t$), $|v| \leq k$, then as $D < \infty$ a.s., we must have $|$BP($t$)$|<\infty$, and hence,
\begin{equation*}
\bigcap_{k=1}^{\infty} \{\exists v \in \mbox{BP}(t) \mbox{ s.t. } |v| = k \} = \{ |\mbox{BP}(t)| = \infty \}.
\end{equation*}
It remains to verify that $\{ |\mbox{BP}(t)| = \infty \} = \{ N_f(t) = \infty \}$. To do so, observe that, if with  non-zero probability, for all $k$, $\exists v \in $ BP($t$) such that $|v| = k$, then there exists an infinite ray in the branching process that has a convergent sum of birth times along it. Thus $N_f(t) = \infty$, since with non-zero probability a branching process starting at a vertex in this ray does not die out. On the other hand if $N^f(t) = \infty$, then $|$BP($t$)$| = \infty$. We conclude that 
\begin{equation*}
1 - \displaystyle \lim_{k \to \infty}  \phi_k(t) = \mathbb{P} \left( \cap_{k=1}^{\infty} \{\exists v \in \mbox{BP}(t) \mbox{ s.t. } |v| = k \} \right) = \mathbb{P} \left( N^f(t) = \infty \right).
\end{equation*}
\end{proof}
\noindent
Thus the iterative solution is both the right solution and the smallest solution to (\ref{eq::CP_FPE_MAIN}), leading to a much weaker criterion for a process to be explosive, which is captured in the next corollary:

\begin{corollary}
\label{cor::CP_WEAKER}
Let $h$ be a probability generating function. Let both $G$ and $C$ be distribution functions of non-negative random variables. The process $\HGCf$ is explosive if and only if there exists $T > 0$ and a function $\Psi: [0,T] \to [0,1]$ such that $\Psi \neq 1$ and 
\be \ba
\label{eq::CP_WEAKER}
\Psi(t) &\geq \lr{\TCf \Psi}(t) \\
&= \int_0^t h\left( \int_0^x \Psi(t-u)\mathrm d G(u) + 1 - G(x)\right)\mathrm d C(x) \\
 & \quad + h\left( \int_0^t \Psi(t-u)\mathrm d G(u) + 1 - G(t)\right)\left(1-C(t)\right), \hspace*{1 cm} \forall t \in [0,T]. 
\ea \ee
\end{corollary}

\begin{proof}
Necessity is obvious as $\phi$ itself satisfies (\ref{eq::CP_WEAKER}) with equality. To prove sufficiency we define $\psi_0: [0,\infty) \to [0,1]$  by
\begin{equation}
 \psi_0(t) = \left\{ 
  \begin{array}{l l}
    \Psi(t) & \quad \text{if $t \in [0,T]$},\\
    1 & \quad \text{if $t > T$},
  \end{array} \right.
\end{equation}
and, for each $k \geq 1$, $\psi_k$ is recursively defined by $\psi_k = \TCf \psi_{k-1}$. We will show that, for each fixed $t \in [0,\infty)$, $(\psi_k(t))_{k \in \mathbb{N}}$ is a decreasing sequence, bounded from below by 0 and from above by $\psi_0(t)$. To see this, fix first $t \in [0,T]$ and note that $\psi_0(t-u) \leq \Psi(t-u)$ for any $u \in [0,x]$, if $x \in [0,t]$, so that in fact,  $\psi_1(t) \leq \Psi(t)$. The upper bound, in the region where $t > T$, follows easily from %To see that we can bound $0 < \phi_1(t)$ by $\phi_0(t) = 1$ for $t>T$ note that for all $x \geq 0$,
\begin{equation*}
0 \leq \int_0^x \psi_0(t-u)\mathrm d G(u) + 1 - G(x) \leq \int_0^x \mathrm d G(u) + 1 - G(x)= 1,
\end{equation*}
since this implies, for $t > T$,
\begin{equation*}
\psi_1(t) \leq \int_0^t h(1)\mathrm d C(x) + h(1)\left(1 - C(x)\right) = 1 = \psi_0(t).
\end{equation*}
We are now in a position to prove by induction that, for any $k \in \mathbb{N}$ and $t \in [0,\infty)$, the inequality $\psi_{k+1}(t) \leq \psi_k(t)$ holds. To establish this, assume  that for some $k \in \mathbb{N}$ we have $\psi_{k}(t) \leq \psi_{k-1}(t)$, then, for all $t \in [0,\infty)$,
\[ \ba
\psi_{k+1}(t) &= \int_0^t h \left( \int_0^x \psi_k(t-u)\mathrm d G(u) + 1 - G(x) \right) \mathrm d C(x) \\
&+ h \left( \int_0^t \psi_k(t-u)\mathrm d G(u) + 1 - G(t))(1-C(t)\right), \ea \]
which is smaller than or equal to
\[ \ba & \int_0^t h \left( \int_0^x \psi_{k-1}(t-u)\mathrm d G(u) + 1 - G(x)\right)\mathrm d C(x) \\ &+ h \left( \int_0^t \psi_{k-1}(t-u)\mathrm d G(u) + 1 - G(t)\right)(1-C(t)) = \psi_k(t). \ea\]
We conclude that, for all $t \in [0,\infty)$, $\displaystyle \psi(t) = \lim_{k \to \infty} \psi_k(t)$  exists and that by invoking LDC it follows that $\psi$ is in fact a solution to (\ref{eq::CP_FPE_MAIN}). The process is indeed explosive, since $\psi$ is dominated by $\Psi$ on $[0,T]$, and the latter is strictly smaller than $1$ somewhere on its domain of definition. 
\end{proof}
\noindent
In the remainder of this section, it turns out useful to have the following lemma at hand: 
\begin{lemma}
For any function $\psi: [0,\infty) \to [0,1]$ it holds true, for all $0 \leq x \leq t \leq t'$,
\begin{enumeratei}
\item  \[\left( \int_0^x \psi ( t'-u)\mathrm d G(u) + 1 - G(x) \right) -  \left( \int_0^t \psi(t'-u)\mathrm d G(u) + 1 - G(t) \right) \geq 0,\]
\item  if, moreover, $\psi$ is non-increasing, \[h \left( \int_0^x \psi ( t'-u)\mathrm d G(u) + 1 - G(x) \right) -  h \left( \int_0^t \psi(t'-u)\mathrm d G(u) + 1 - G(t) \right) \geq 0,\] 
\item furthermore, for those non-increasing $\psi$, \[h\left( \int_0^t \psi(t-u)\mathrm d G(u) + 1 - G(t) \right) \geq  h\left( \int_0^{t'} \psi(t'-u)\mathrm d G(u) + 1 - G(t') \right).\]
\end{enumeratei}
\label{lm::estimates}
\end{lemma}
\begin{proof}
Elementary calculations and the fact that $\psi \leq 1$ together with monotonicity of $h$, $G$ and $C$.
\end{proof}

\begin{theorem}
For each $k \in \mathbb{N}$, $\phi_k$ is a non-increasing function and hence $\phi$ is. 
\end{theorem}
\begin{proof}
We prove the result by induction on $k$. Fix $0 \leq t \leq t'$ and note that the hypothesis obviously holds for $\phi_0 \equiv 0$. Hence, we may assume that, for some $k \in \mathbb{N}$, $\phi_k$ is non-increasing, but then, $\phi_{k+1}(t) - \phi_{k+1}(t')$ equals 
\[ \ba  & \int_0^t \left( h \lr{ \int_0^x \phi_k(t-u)\mathrm d G(u) + 1 - G(x)} - h\lr{ \int_0^x \phi_k(t'-u)\mathrm d G(u) + 1 - G(x)} \right) \mathrm d C(x) \\
 & \quad - \int_t^{t'}  h \lr{ \int_0^x \phi_k(t'-u)\mathrm d G(u) + 1 - G(x) } \mathrm d C(x) \\
 & \quad + h \lr{ \int_0^t \phi_k(t-u)\mathrm d G(u) + 1 - G(t)}(1-C(t)) \\ 
 & \quad - h \lr{ \int_0^{t'} \phi_k(t'-u)\mathrm d G(u) + 1 - G(t')}(1-C(t')).\ea \]
Since $\phi_k$ is assumed to be non-increasing, we may bound this from below, 
\[ \ba \phi_{k+1}(t) - \phi_{k+1}(t') \geq 0 &- \int_t^{t'}  h \lr{ \int_0^t \phi_k(t'-u)\mathrm d G(u) + 1 - G(t)} \mathrm d C(x) \\
 &  + h \lr{  \int_0^t \phi_k(t-u)\mathrm d G(u) + 1 - G(t)}(1-C(t)) \\
 &  - h \lr{  \int_0^{t'} \phi_k(t'-u)\mathrm d G(u) + 1 - G(t')}(1-C(t')).  \ea \]
Explicit calculation of the first outer-integral yields 
\[ \ba \phi_{k+1}(t) - \phi_{k+1}(t') \geq &-  h \lr{  \int_0^t \phi_k(t'-u)\mathrm d G(u) + 1 - G(t)}  \left( C(t') - C(t) \right) \\
 &  + h \lr{  \int_0^t \phi_k(t-u)\mathrm d G(u) + 1 - G(t)}(1-C(t)) \\
 &  - h \lr{  \int_0^{t'} \phi_k(t'-u)\mathrm d G(u) + 1 - G(t')}(1-C(t')). \ea \]
We employ again that $\phi_k$ is non-increasing (to bound the first integral from below), so that we eventually have the following lower bound:
\[ \ba \phi_{k+1}(t) - \phi_{k+1}(t') & \geq h \lr{  \int_0^t \phi_k(t-u)\mathrm d G(u) + 1 - G(t)}(1-C(t')) \\
 & \quad - h \lr{  \int_0^{t'} \phi_k(t'-u)\mathrm d G(u) + 1 - G(t')}(1-C(t')),  \ea \]
which is positive by Lemma \ref{lm::estimates}(iii). Thus $\phi_{k+1}$ is non-increasing and the proof is complete. 
\end{proof}

\begin{theorem}
If there exists $\tau > 0$ such that $\phi(\tau) < 1$, then $\phi(t) < 1$ for all $t \in (0,\infty)$.
\end{theorem}
\begin{proof}
Suppose on the contrary that there exists $t > 0$ such that $\phi(t) = 1$, but then, as $\phi$ is non-increasing, there must exist $t_0 > 0$ such that
\begin{equation}
  \begin{array}{l l}
    \phi(t) = 1 & \quad \text{for $t \leq t_0$,}\\
    \phi(t) < 1 & \quad \text{for $t > t_0$.}
  \end{array}
\end{equation}
Then, for any $t > t_0$, upon splitting the integrals at $t - t_0$ and noting that $\phi(t-u) = 1$ for $u \in [t-t_0,t]$,
\[ \ba 
\phi(t) &= \int_0^{t-t_0} h \left( \int_0^x \phi(t-u) \mathrm d G(u) + 1 - G(x) \right)\mathrm d C(x) \\
		& \quad + \int_{t-t_0}^t h \left( \int_0^{t-t_0} \phi(t-u) \mathrm d G(u) + \int_{t-t_0}^x \mathrm d G(u) + 1 - G(x)\right)\mathrm d C(x) \\
 		& \quad + h \left(\int_0^{t-t_0} \phi(t-u)\mathrm d G(u) + \int_{t-t_0}^t \mathrm d G(u) + 1 - G(t)\right)(1-C(t)). \ea \]
There is no dependence on $x$ any-more inside the $h(\cdot)$ term in the second integral, hence the above equals
\[ \ba  \phi(t) &=  \int_0^{t-t_0} h \left( \int_0^x \phi(t-u) \mathrm d G(u) + 1 - G(x) \right)\mathrm d C(x) \\
		& \quad + h \left(\int_0^{t-t_0} \phi(t-u)\mathrm d G(u) + 1 - G(t-t_0)\right)(C(t)-C(t-t_0)) \\
 		& \quad + h \left(\int_0^{t-t_0} \phi(t-u)\mathrm d G(u) + 1 - G(t - t_0)\right)(1-C(t)). \ea \]
We note that terms in the second and third line partly cancel each other, therefore we may write
\[ \ba
\phi(t) &=  \int_0^{t-t_0} h \left( \int_0^x \phi(t-u) \mathrm d G(u) + 1 - G(x) \right)\mathrm d C(x) \\
 		& \quad + h \left(\int_0^{t-t_0} \phi(t-u)\mathrm d G(u) +  1 - G(t-t_0)\right)(1-C(t-t_0)). \ea \]
\noindent
Define $\phi^*: [0,\infty) \to [0,1]$ by
\begin{equation}
\phi^*(\tau) = \phi(\tau + t_0), \hspace*{1 cm} \tau \in [0,\infty),
\end{equation}
to see that, for $\tau \in [0,\infty)$, 
\[ \ba
\phi^*(\tau) &= \int_0^{\tau} h \left( \int_0^x \phi(\tau + t_0 -u) \mathrm d G(u) + 1 - G(x) \right)\mathrm d C(x) \\
  & \quad + h \left(\int_0^{\tau} \phi(\tau + t_0-u)\mathrm d G(u) +  1 - G(\tau)\right)(1-C(\tau)), \ea \]
or,
\[ \ba \phi^*(\tau) &= \int_0^{\tau} h \left( \int_0^x \phi^*(\tau -u) \mathrm d G(u) + 1 - G(x) \right)\mathrm d C(x) \\
 		& \quad + h \left(\int_0^{\tau} \phi^*(\tau -u)\mathrm d G(u) +  1 - G(\tau)\right)(1-C(\tau)). \ea \]
Hence, in fact, $\phi^*$ is a solution to (\ref{eq::CP_FPE_MAIN}), leading to a contradiction since 
$\phi^*(\frac{t_0}{2})  <  \phi(\frac{t_0}{2})$
violates the minimality of $\phi$. We conclude that $\phi(t) < 1$ for all $t>0$.
\end{proof}

\subsection{Comparison theorems}
In this section we show that explosiveness is essentially determined by the behaviour of both $G$ and $C$ around $0$ and $h$ around $1$: 
\begin{theorem}
Let $h$ and $h^*$ be probability generating functions such that  for some $\theta < 1$, $h^*(s) \leq h(s)$ for all $s \in [\theta,1]$. Let $G$ be the distribution function of a non-negative random variable such that $G(0) = 0$. Let $C$ be the distribution function of a non-negative random variable. If the process $\HGCf$ explosive, then so is $(h^*,G,C)^f$.
\end{theorem}
\begin{proof}
Consider the explosive process $\HGCf$ and denote by $\phi^f$ the corresponding smallest solution to (\ref{eq::CP_FPE_MAIN}). We observe that for all $t \geq 0$,
\begin{equation*}
1 \geq \int_0^t \phi^f(t-u) \mathrm d G(u) + 1 - G(t) \geq 1 - G(t), 
\end{equation*}
where the right hand side tends to $1$ when $t$ goes to $0$. Because $G(0) = 0$, there exists  $T>0$ such that, for $t \in [0,T]$, 
\begin{equation*}
\int_0^t \phi^f(t-u) \mathrm d G(u) + 1 - G(t) \geq \theta. 
\end{equation*}
Thus, for such $t$, 
\begin{equation*}
h \left( \int_0^t \phi^f(t-u) \mathrm d G(u) + 1 - G(t) \right) \geq h^* \left( \int_0^t \phi^f(t-u) \mathrm d G(u) + 1 - G(t) \right), 
\end{equation*}
since by assumption $h^*(s) \leq h(s)$ for all $s \in [\theta,1]$.
We establish that, for any $t \in [0,T]$,
\be \ba
\phi^f(t) &\geq \int_0^t h^* \left( \int_0^x \phi^f(t-u) \mathrm d G(u) + 1 - G(x) \right)\mathrm d C(x) \\ 
& \quad + h^* \left( \int_0^t \phi^f(t-u) \mathrm d G(u) + 1 - G(t)\right)(1-C(t)) \\
& = \lr{T^f_{(h^*,G,C)} \phi^f}(t).
\ea \ee
An appeal to Corollary \ref{cor::CP_WEAKER}, with $\Psi = \phi^f$, completes the proof.
\end{proof}
\begin{theorem}
Let $h$ be a probability generating function. Let both $G$ and $G^*$ be distribution functions of non-negative random variables such that $G(0)=G^*(0)=0$. Let $C$ be the distribution function of a non-negative random variable. Assume that there is a $T$ such that $G^* \geq G$ on $[0,T]$. If the process $\HGCf$ is explosive, then so is $(h,G^*,C)^f$. 
\end{theorem}
\begin{proof}
Fix a $t \leq T$ and write $\phi_t(u)=\phi^f(t-u)$ for $u \leq t$, where $\phi^f$ is the smallest solution to (\ref{eq::CP_FPE_MAIN}). Partial integration gives,  since $G(0)=0$,
\begin{equation*}
\int_0^x \phi_t(u)\mathrm d G(u) = G(x)\phi_t(x) - \int_0^x G(u) \mathrm d \phi_t(u),
\end{equation*}   
for all $x \leq t$. We thus have
\begin{eqnarray}
\int_0^x \phi_t(u)\mathrm d G(u) + 1-G(x) &=& 1- G(x)(1-\phi_t(x)) - \int_0^x G(u) \mathrm d \phi_t(u) \\
&\geq& 1- G^*(x)(1-\phi_t(x)) - \int_0^x G^*(u) \mathrm d \phi_t(u),
\end{eqnarray}
since $G^*\geq G$ on $[0,T]$ and $\phi_t$ is a non-decreasing function (hence $\mathrm d \phi_t$ is non-negative).
Writing now the RHS back in the same manner, we obtain
\begin{equation*}
 \int_0^x \phi_t(u)\mathrm d G(u) + 1-G(x) \ge  \int_0^x \phi_t(u) \mathrm d G^*(u) + 1-G^*(x), \hspace*{2 cm} x \leq t.
\end{equation*}
Then, for all $t \leq T$, 
\[ \ba 
\phi^f(t) &= \lr{ \TCf \phi^f } (t) \\
 & \geq \lr{T^f_{(h^*,G,C)} \phi^f}(t),
\ea \]
and Corollary \ref{cor::CP_WEAKER}, where we use $\Psi = \phi^f$ as a test function, captures that $(h,G^*,C)^f$ is explosive. 
\end{proof}

\begin{theorem}
Let $h$ be a probability generating function. Let $G$ be the distribution function of a non-negative random variable. Let both $C$ and $C^*$ be distribution functions of non-negative random variables such that $C(0)=C^*(0)=0$. Assume that there is a $T$ such that $C^* \leq C$ on $[0,T]$. If the process $\HGCf$ is explosive, then so is $(h,G^*,C)^f$. 
\end{theorem}
\begin{proof}
Fix $t < T$. Let \[h_t(x) = h\lr{\int_0^x \phi_t(u)\mathrm d G(u) + 1-G(x)},\] for $x \leq t$, where $\phi_t$ is defined as in the previous proof.  Note that $h_t$ now is a \emph{decreasing} function. 
We write, since $C(0)=0$,
\[\int_0^t h_t(x) \mathrm d C(x)= C(t) h_t(t) - \int_0^t C(x) \mathrm d h_t(x). \] 
Hence,
\[ \begin{aligned} \phi^f(t) &= \int_0^t h_t(x) \mathrm d C(x) + h_t(t)(1-C(t)) \\
&= h_t(t) C(t) - \int_0^t C(x) \mathrm d h_t(x) + h_t(t) (1-C(t)) \\
&= h_t(t) - \int_0^t C(x) \mathrm d h_t(x)\\
&\ge  h_t(t) - \int_0^t C^*(x) \mathrm d h_t(x), \end{aligned}\]
where in the last line we used that $\mathrm d h_t$ is non-positive. From here the statement follows by writing the RHS back and using Corollary \ref{cor::CP_WEAKER} with $\Psi = \phi^f$, since for $t \leq T$,
\[  
\phi^f(t) \geq \lr{T^f_{(h,G,C^*)} \phi^f}(t).
 \]
\end{proof}

\subsection{No conservative survival}
Now that we have the machinery at hand to decide on explosiveness questions it is time to ask whether an explosive process might survive conservatively with positive probability. The next theorem implies that this cannot be the case:

\begin{theorem}
\label{thm::CP_EXP_PROB}
Let $h$ be a probability generating function. Let $G$ be the distribution function of a non-negative random variable. Let $C$ be the distribution function of a non-negative random variable.
Consider the Galton-Watson process describing the generation sizes of the time-continuous branching process $\HGCf$ and denote the probability generating function of the corresponding offspring distribution by $h_{\text{eff}}$. Denote also $\eta_{\infty} = \displaystyle \lim_{t \to \infty} \mathbb{P}(N_t = \infty) $. We have
\begin{equation}
h_{\text{eff}}(s) =  \int_0^{\infty} h \lr{1 - (1-s) G(x)} \mathrm d C(x),\quad |s| \leq 1,
\end{equation}
and
\begin{equation}
1 - \eta_{\infty} = \int_0^{\infty} h \left( 1 - \eta_{\infty} G(x) \right) \mathrm d C(x).
\label{eq::CP_CONSS2}
\end{equation}
Further, the probability of conservative survival must equal zero in an explosive process.
\end{theorem}
\begin{proof}
Write $\phi = 1 - \eta$. Since $\phi \geq 0$ is non-increasing, the limit $\eta_{\infty} = 1 -\displaystyle \lim_{t \to \infty} \phi(t)$ exists. 
Let $\epsilon > 0$ be given. As $h$ is uniformly continuous on $[0,1]$, there exists a $\delta >0$ such that $|x-y| < \delta$ implies $|h(x) - h(y)| < \frac{\epsilon}{3}$. To exploit these observations, fix $t > T$ both so large that $|1 - C(T)| \leq \frac{\epsilon}{3}$ and $|\eta(t-u) - \eta_{\infty}| < \delta$ for all $u \in [0,T]$. Write for those $u$, $\eta(t-u) = \eta_{\infty} + \delta(u)$ with $|\delta(u)| \leq \delta$, to establish by splitting the integrals at $t$ and $T$, 
\be \ba
& \left| \int_0^{\infty} h \left( 1 - \eta_{\infty} G(x) \right) \mathrm d C(x) - \int_0^{t} h \left( 1 - \int_0^{x} \eta(t-u) \mathrm d G(u) \right) \mathrm d C(x) \right| \\
& \leq \left|\int_0^{T}\left( h \left( 1 - \eta_{\infty} G(x) \right) -  h \left( 1 - \eta_{\infty} G(x) -\int_0^{x} \delta(u) \mathrm d G(u) \right) \right)\mathrm d C(x) \right| \\
& \quad + |1 - C(T)| + |C(t) - C(T)|,
 \ea \ee
as $\int_0^{x} \delta(u) g(u) \mathrm{d}u \leq \delta$, we may bound this from above by
\be \ba
\left|\int_0^{T} \frac{\epsilon}{3} \mathrm d C(x)\right| + \frac{\epsilon}{3} + \frac{\epsilon}{3} =  \epsilon.
 \ea \ee 
As $\epsilon$ was arbitrary, the second assertion \eqref{eq::CP_CONSS2} follows.
In the corresponding Galton Watson process, the offspring that will be born is given by
\begin{equation*}
D^f = \sum_{i=1}^D \indicator{X_i < C}, 
\end{equation*}
where $D$ has distribution $G$ and $(X_i)_i$ constitute an i.i.d. sequence with distribution $G$. 
The random variable $D^f$ has probability generating function $h_{\text{eff}}$, given for $|s| \leq 1$ by $h_{\text{eff}}(s) = \mathbb{E}s^{D^f}$, that is
\begin{eqnarray*}
h_{\text{eff}}(s) &=& \mathbb{E} s^{\sum_{i=0}^D \indicator{X_i < C}} \\
&=&  \sum_{k=0}^{\infty} \mathbb{P}(D = k) \mathbb{E} s^{\sum_{i=0}^k \indicator{X_i < C}}  \\
&=& \sum_{k=0}^{\infty}  \mathbb{P}(D = k)\int_0^{\infty} \mathbb{E} s^{\sum_{i=0}^k \indicator{X_i < x}} \mathrm d C(x) \\
&=& \sum_{k=0}^{\infty} \mathbb{P}(D = k) \int_0^{\infty}\left( \mathbb{E}  s^{ \indicator{X < x}} \right)^k \mathrm d C(x) \\
  &=&\int_0^{\infty} \sum_{k=0}^{\infty} \mathbb{P}(D = k) \left( \mathbb{E}  s^{ \indicator{X < x}} \right)^k \mathrm d C(x)  \\ 
 &=& \int_0^{\infty} h(1 - (1-s) G(x)) \mathrm d C(x).
\end{eqnarray*}
It is well known that the survival probability $q$  in an ordinary branching process equals the largest $s \in [0,1]$ such that $1-s = h_{\text{eff}}(1-s)$ and as there are at most two candidates ($0$ and a possibly larger number), $q$ must equal $\eta_{\infty}$ since the latter satisfies \eqref{eq::CP_CONSS2}. This completes the proof.
\end{proof}
\noindent As a corollary we have that an explosive process \emph{without} contagious period explodes almost surely:
\begin{corollary}
If we consider the quantity $\eta_{\infty} = \displaystyle \lim_{t \to \infty} \mathbb{P}(N(t) = \infty) $ in the process $(h,G)$ that is postulated to be explosive, then,
\[ 1 - \eta_{\infty} = h(1-\eta_{\infty}). \]
\label{cor::CP_CONSS}
\end{corollary}
\begin{proof}
There are two ways to see this, either by setting $\mathbb{P} \left( \tau^C < \infty \right) = 0$ or recalling from the theory of ordinary branching processes that the extinction probablity $\eta$ is such that $\eta = h(\eta)$ and noting that $\eta_{\infty} = 1$.
\end{proof}
\begin{remark}
Conditional on surival, explosion is a tail event and by Kolmogorov's 0-1 law  it happens with probability either $0$ or $1$, as we just deduced by more elaborate means.
\end{remark}

\subsection{The explosiveness question}
If we are given an explosive process $(h,G)$, then a natural question to ask is under what conditions will adding a contagious period $C$ stop this process from being explosive? As pointed out in Section 1.3.1, one way to answer the explosiveness question is by observing that $C$ must be such that it kills with probability $1$ an edge on every finite ray. We employ that observation to prove Theorem \ref{thm::CP_ray}.
\\
Another way of addressing the explosiveness question is by employing the fixed point equation describing $\HGCf$ and using that the corresponding equation to $\HG$ has a solution not identically 1. By this approach we will establish that for any $\alpha \in (0,1)$ , $C$ such that $C \neq 0$, and any function $L$ slowly varying at infinity, the process $\HaLGCf$ is explosive if and only if $\HaLG$ is explosive, i.e., Theorem \ref{thm::HaGCf}. We extend this result to Theorem \ref{thm::HaLGCf}.
\\
To apply the first method, we need to quantify those rays a bit further. To this end, let $(T_i)_i$ be the associated life-times on the random ray that has shortest length (if all rays in a particular realization are infinite, then take the left-most ray). We see that a sufficient condition for $\HGCf$ to be explosive reads
\[ \mathbb{P}\left( \forall i: C_i \geq T_i , \sum_{i = 1}^{\infty} T_i < \infty  \right) > 0 ,\]
where $(C_i)_i$ is an i.i.d. sequence of random variables with distribution $C$, which we employ to prove Theorem \ref{thm::CP_ray}:
\begin{proof}[Proof of Theorem \ref{thm::CP_ray}]
First, note that 
\[ \mathbb{P} \left( \sum_{i = 1}^{\infty} T_i < \infty \right) = \P{N_t = \infty}. \]
Hence, since $\HG$ is explosive, $\mathbb{P} \left( \sum_{i = 1}^{\infty} T_i < \infty \right) = 1$, as follows from Corollary \ref{cor::CP_CONSS}. Thus,
\[ \mathbb{P}\left( \forall i: C_i \geq T_i , \sum_{i = 1}^{\infty} T_i < \infty \right) = \P{\forall i: C_i \geq T_i}.  \]
To handle this probability, we condition on all the possible values that $(T_i)_i$ may take,
\be \ba
\mathbb{P}\left( \forall i: C_i \geq T_i  \right) &= \E{\mathbb{P}\left( \left. \forall i: C_i \geq T_i  \right|(T_i)_i  \right)} \\
&= \mathbb{E}\left[  \prod_{i=1}^{\infty} \left( 1 - C(T_i) \right) \right].
\ea \ee
It remains to verify that this is strictly positive (note that here we need $C(t) < 1$ for all $t > 0$). But, $\prod_{i=1}^{\infty} \left( 1 - C(T_i) \right) > 0$ if and only if $\sum_{i = 1}^{\infty} C(T_i) < \infty$. The latter is true because $\sum_{i = 1}^{\infty} T_i < \infty $ a.s., so that there is a random variable $N$ ($N < \infty$ a.s.) such that $T_i < \theta$ for all $i \geq N$, implying that, since $C(t) \leq \beta t$ for $t < \theta$,
\[\sum_{i = 1}^{\infty} C(T_i) \leq \sum_{i = 1}^{N} C(T_i) + \beta \sum_{i = N+1}^{\infty} T_i <\infty, \] with probability 1. Hence,  $\prod_{i=1}^{\infty} \left( 1 - C(T_i) \right) > 0$ almost surely, and the result follows.
\end{proof}
\noindent
Recall $\Ha$ from \eqref{eq::Ha}. Before proving that $\HaLGCf$ is explosive if and only if $\HaLG$ is explosive for any $C \neq 0$ and $L$ slowly varying at infinity, we first show that $(h_{\alpha},M,C)$ is always explosive, where $M$ is the distribution function of an exponential random variable with parameter $\lambda > 0$. The calculation establishing this will serve as an illustrative application of Corollary \ref{cor::CP_WEAKER}, as it involves explicitly finding a suitable test function. Let us start by arguing in a heuristic way that $(h_{\alpha},M,C)$ should be explosive. For this process we have that the probability generating function $h_{\text{eff}}$, of the offspring that is eventually born, is given by (according to Theorem \ref{thm::CP_EXP_PROB})
\[ h_{\text{eff}}(s) = 1 - \left(1-s\right)^{\alpha} \int_0^{\infty} M(x)^{\alpha} \mathrm d C(x), \hspace*{1 cm} |s| \leq 1. \] 
Thus $\delta = \int_0^{\infty} M(x)^{\alpha} \mathrm d C(x) > 0$,  $h'_{\text{eff}}(1) = \infty$ and, 
\begin{equation}
\ba
s - h_{\text{eff}}(s) &= (s-1) + (1-s)^{\alpha} \delta\\ 
&= (1-s)^{\alpha}(\delta - (1-s)^{1 - \alpha}) \\
&\geq  (1-s)^{\alpha}\frac{\delta}{2}, 
\ea
\label{eq::CP_EXP_ESTIMATE}
\end{equation}
if $s$ close to $1$. Hence,
\begin{equation*}
\int_{1-\epsilon}^{1} \frac{\mathrm d s}{s- h_{\text{eff}}(s)} \leq \frac{2}{\delta} \int_{1-\epsilon}^{1} \frac{\mathrm d s}{(1-s)^{\alpha}} < \infty,
\end{equation*}
for $\epsilon$ so small that the inequality in (\ref{eq::CP_EXP_ESTIMATE}) holds. It is well-known, see for instance \cite{Har63}, that this implies that the process $(h_{\text{eff}},G)$ is explosive, which gives a hint on the explosiveness of $(h_{\alpha},M,C)$. Note that we only gave an heuristic proof as there is no independence in the process to employ. That this heuristic argument  actually rests on firm grounds is the content of the next theorem:

\begin{theorem}
Let $\alpha \in (0,1)$. Let $M$ be the law of an exponential random variable with parameter $\lambda > 0$. Let $C$ be the distribution function of a random variable  such that $C(\theta) < 1$ for some $\theta >0$. Then, the process $\HaGCf$ is explosive. 
\label{thm::CP_Markov}
\end{theorem}
\begin{proof}
Let $g$ be the derivative of $M$, i.e., $g(u) = \lambda e^{-\lambda u}$ for $u \geq 0$. Pick $\beta$ so large that $\frac{\beta}{\beta + 1} > \alpha$ and set $\psi$ for $t \geq 0$ equal to $\psi(t) = 1 - t^{\beta}$. By Corollary \ref{cor::CP_WEAKER} it suffices to show that there exists $T>0$ such that
\be
\ba
\psi(t) &\geq \int_0^t h \left( \int_0^x \psi(t-u) \mathrm d G(u) + 1 - G(x)\right) \mathrm d C(x) \\ 
& \quad + h \left(  \int_0^t \psi(t-u) \mathrm d G(u) + 1 - G(t)\right)(1-C(t)), \hspace*{2 cm} t \in [0,T].
\ea
\ee
However, this is equivalent to,
\be
\eta(t) \leq  \int_0^t  \left( \int_0^x \eta(t-u) \mathrm d G(u) \right)^{\alpha} \mathrm d C(x) + \left( \int_0^t \eta(t-u) \mathrm d G(u) \right)^{\alpha} \left( 1 - C(t) \right), \quad t \in [0,T], \label{eq::CP_Markov}
\ee
if we set $\eta = 1 - \psi$. In fact, the last term on the RHS already dominates the LHS. Indeed, for $t \leq \theta$, $1 - C(t) > 0$, and thus
\begin{equation*}
\int_0^t \eta(t-u) \mathrm d G(u) \geq \lambda \  \exp(-\lambda t) \frac{t^{\beta + 1}}{\beta + 1}.
\end{equation*}
Hence,
\begin{equation}
\left( \int_0^t \eta(t-u) \mathrm d G(u) \right)^{\alpha} ( 1 - C(t) ) \geq A t^{\alpha (\beta + 1)},
\label{eq::CP_EXP_INT_ESTIMATE}
\end{equation}
where \[A = \lr{\frac{\lambda e^{-\lambda \theta} }{\beta + 1}}^{\alpha} ( 1 - C(\theta) ) > 0\] is a constant. Now, there is a $T \leq \theta$ such that the term in the RHS of (\ref{eq::CP_EXP_INT_ESTIMATE}) is larger than or equal to $\eta(t) = t^{\beta}$ for all $t \leq T$, since $\alpha (\beta+1) < \beta$.
\end{proof}

\noindent By picking suitable test functions it is possible to generalize this result further into Theorem \ref{thm::HaGCf}:
\begin{proof}[Proof of Theorem \ref{thm::HaGCf}]
Since $(h,G)$ is explosive, there is a $\phi: [0,\infty) \to [0,1]$, with $\phi(t) < 1$ for $t > 0$, such that 
\[ \phi(t) = \Ha \left( \int_0^t \phi(t-u) \mathrm d G(u) + 1 - G(t)\right), \hspace*{1 cm} t \geq 0. \]
Write $\phi = 1 - \eta$, then, using that $\Ha(s) = 1 - (1-s)^{\alpha}$ for $s \in [0,1]$, 
\[ \eta(t) =  \left( \int_0^t \eta(t-u)\mathrm d G(u) \right)^{\alpha} .\]
Let $1 > A > 0$ be a constant and put $\psi = A \eta$, then, for $t \leq \theta$,
\begin{eqnarray*}
 \left( \int_0^t \psi(t-u)\mathrm d G(u) \right)^{\alpha} (1-C(t)) &\geq& (1-C(\theta)) A^{\alpha}  \left( \int_0^t \eta(t-u)\mathrm d G(u) \right)^{\alpha} \\
 &=& (1-C(\theta)) A^{\alpha}  \eta(t) \\
 &=& (1-C(\theta)) A^{\alpha - 1} \psi(t) \\
 &\geq& \psi(t), 
\end{eqnarray*}
if $A$ is sufficiently small. Since $1 > \psi > 0$ on $(0,\theta]$,  Corollary \ref{cor::CP_WEAKER}  proves the result (see the proof of Theorem \ref{thm::CP_Markov} and in particular equation \eqref{eq::CP_Markov}).
\end{proof}
\noindent An appeal to Potter's theorem extends this result to the most general case:

\begin{proof}[Proof of Theorem \ref{thm::HaLGCf}]
Assume that $\HaLG$ is explosive. In the proof of Theorem \ref{th::AGE_allAlphaL}, equation \eqref{eq::AGE_POTTER} points out that  $(h_{\alpha - \epsilon},G)$ is explosive if we set $0 <\epsilon < \alpha$. Hence, for suitably small $\epsilon > 0$, $(h_{\alpha + \epsilon},G)$ is explosive by Theorem \ref{th::AGE_allAlpha}. Theorem \ref{thm::HaGCf} entails that $(h_{\alpha + \epsilon},G,C)^f$ is explosive and it thus follows from another appeal to equation \eqref{eq::AGE_POTTER} together with the comparison theorems that $\HaLGCf$ is explosive.  
\end{proof}

\newpage\null\thispagestyle{empty}\newpage

\section{Backward Age-Dependent Branching Processes with Contagious Periods}
In the backward process we study how an infection reaches to a person backward in time. It is denoted by  $\HGCb$, with $h$ a probability generating function and both $G$ and $C$ distribution functions of non-negative random variables. This process is described in more detail in Section 1.4. There it is pointed out that we might identify the process $\HGCb$ with either $(h_{\text{eff}},G_{\text{eff}})$ or
 $(h,G_C)$ and here we choose to make the identification $\HGCb = (h,G_C)$.
For $x < \infty$, $G_C(x)$ is defined as 
\be
\ba
G_C(x) &=  \P{X \leq x , X \leq C} \\
 & = \int_0^x (1-C(u))\mathrm d G(u).
\ea
\ee

\subsection{Fixed point equation}
The function $\phi^b:[0,\infty) \mapsto [0,1]$, given for $t \geq 0$ by $\phi^b(t) = 1 -\P{N^b(t) = \infty}$, is the smallest solution to 
\be
\Phi = \TCb \Phi ,\quad \Phi \in \mathcal{V},
\ee
where $\mathcal{V}$ is defined in \eqref{eq::function_space}. The operator $\TCb = T_{(h,G_C)}$ is given for $t \geq 0$ by 
\be
\lr{\TCb \Phi}(t) =  h\left(\int_0^t \Phi(t-u) \mathrm d G_C + 1-G_C(t)  \right),
\ee
see \eqref{eq::AGE_FP_usual}. Due to the identification $(h,G_C) = \HGCb$, all machinery of Section 1 carries through. 

\subsection{The explosiveness question}
Here we show that $\HGCb$ is explosive if and only if $\HG$ is explosive, under the condition that $C(0) \neq 1$ and $h$ is the probability generating function of a plump distribution $D$, i.e., Theorem \ref{thm::BCP_explosiveness}:

\begin{proof}[Proof of Theorem \ref{thm::BCP_explosiveness}]
To employ Theorem \ref{th::AGE_PLUMP} we make some preliminary observations. To start with, we note that there exists $\theta \in (0,1)$ such that $1-C(t) \geq \theta$ for all $t \leq T$. Moreover, as $D$ is plump, there exists $\epsilon >0$ such that 
\be F^{\leftarrow}_D(1 - \frac{1}{m}) \geq m^{1 + \epsilon} \label{eq::CP_BACK_f_n_plump}, \ee
for all sufficiently large $m$.
We define $f: \mathbb{N} \to [0,\infty)$ as
\[ f(0) = m_0 \mbox{ and }f(n+1) = F^{-1}_D(1 - \frac{1}{f(n)}), n \geq 0, \]
with  $m_0^{\epsilon} > \frac{1}{\theta} > 1$ large enough for the plumpness inequality to hold for every $m \geq m_0$. Then, by \eqref{eq::CP_BACK_f_n_plump},
\be f(n+1) \geq f(n)^{\epsilon}f(n) \geq f(0)^{\epsilon}f(n) > \frac{1}{\theta} f(n), \label{eq::CP_BACK_f_n_increasing} \ee
because $f(\cdot)$ is strictly increasing as follows from the first inequality and an induction argument. Now, for $t \leq T$,
\[ G_C(t) = \int_0^t (1-C(x))\mathrm d G(x) \geq \int_0^t  \theta \mathrm d G(x) = \theta G(t), \]
and hence,
\be
G^{\leftarrow}_C(t) \leq G^{\leftarrow}\lr{\frac{t}{\theta}},
\label{eq::CP_BACK_plump1}
\ee
if $G^{\leftarrow}(\frac{t}{\theta}) \leq T$ or $t \leq \theta G(T)$. Let $N < \infty$ be so large that $\frac{1}{f(n)} \leq \theta G(T)$ for all $n \geq N$, which is possible because $f$ is strictly increasing, see \eqref{eq::CP_BACK_f_n_increasing}. Then,
\be \ba 
\s{n=0}{\infty} G^{\leftarrow}_C \lr{\frac{1}{f(n)}} & = \s{n=0}{N} G^{\leftarrow}_C \lr{\frac{1}{f(n)}} + \s{n=N+1}{\infty} G^{\leftarrow}_C \lr{\frac{1}{f(n)}} \\
& \leq \s{n=0}{N} G^{\leftarrow}_C \lr{\frac{1}{f(n)}} + \s{n=N+1}{\infty} G^{\leftarrow} \lr{\frac{1}{\theta f(n)}} \\
& < \infty.
\ea \ee
The first inequality is due to \eqref{eq::CP_BACK_plump1}.
The last inequality follows from the fact that $\theta f(n+1)> f(n)$, for all $n$ (recall  \eqref{eq::CP_BACK_f_n_increasing}),
\[ \s{n=N+1}{\infty} G^{\leftarrow} \lr{\frac{1}{\theta f(n)}} = \s{n=N}{\infty} G^{\leftarrow} \lr{\frac{1}{\theta f(n+1)}} \leq \s{n=N}{\infty} G^{\leftarrow} \lr{\frac{1}{f(n)}} < \infty,\]
because $\HG$ is explosive.
\end{proof}

\newpage

\section{Comparison of Backward and Forward Processes with Contagious Periods}
In this section we prove Theorem \ref{th::CP_COMPARING_BW_FW}, which states that the explosion time in the forward process stochastically dominates the explosion time in the backward process:
\begin{proof}[Proof of Theorem \ref{th::CP_COMPARING_BW_FW}]
We show that, for all $t \geq 0$, $\phi^f(t) \geq \phi^b(t)$, where $\phi^b$ and $\phi^f$ are the smallest functions such that
\be \phi^b = \TCb \phi^b, \label{eq::CP_COMPARING_phib} \ee
alternatively, for $t \geq 0$, 
\eqn{\phi^b(t) =h\left(\int_0^t \phi^b(t-u) \mathrm d G_C(u) + 1-G_C(t)  \right).}
And, 
\be \phi^f = \TCf \phi^f, \label{eq::CP_COMPARING_phif}\ee
which translates, for $t \geq 0$, into
\be \ba  \phi^f(t) &= \int_0^t h \left( \int_0^x \phi^f(t-u)\mathrm d G(u) + 1 - G(x) \right) \mathrm d C(x) \\
 & \quad +  h \left( \int_0^t \phi^f(t-u)\mathrm d G(u) + 1 - G(t)\right)(1-C(t)). \label{eq::CP_COMPARING_phi_f} \ea \ee
We rewrite \eqref{eq::CP_COMPARING_phi_f} as

\be \ba  \phi^f(t) &= \int_0^{\infty} h \left( \indicator{[0,t]}(x) \lr{\int_0^x \phi^f(t-u)\mathrm d G(u) + 1 - G(x)} \right. \\ & \quad + \left.  \indicator{[t,\infty]}(x) \lr{\int_0^t \phi^f(t-u)\mathrm d G(u) + 1 - G(t)} \right) \mathrm d C(x),  \ea \ee
and, employing the convexity of $h$ using Jensen's inequality,
\be \ba \phi^f(t) & \geq h \lr{\int_0^t \int_0^x \phi^f(t-u) \mathrm dG(u) \mathrm d C(x) + \int_0^t (1-G(x)) \mathrm d C(x) \right. \\ & \quad + \left.(1-C(t)) \left( \int_0^t \phi^f(t-u)\mathrm d G(u) + 1 - G(t)\right) } . \ea \ee
Now, the double integral may be written as
\[ \int_0^t \int_u^t \phi^f(t-u)\mathrm d C(x) \mathrm d G(u) =  \int_0^t \phi^f(t-u) (C(t) - C(u))\mathrm d G(u),\]
and,
\[ \int_0^t (1-G(x)) \mathrm d C(x) + (1-G(t))(1-C(t)) = 1 - \int_0^t  (1-C(x))\mathrm d G(x).  \]
Hence,
\be \phi^f(t) \geq h \lr{\int_0^t \phi^f(t-u) (1 - C(u))\mathrm d G(u) + 1 - \int_0^t  (1-C(u))\mathrm d G(u)}, \ee
or,
\be \phi^f(t) \geq h \lr{\int_0^t \phi^f(t-u) \mathrm d G_C(u) +1 - G_C(t)}. \ee
We will now derive an iterative solution for $(h,G_C)$ and show that it is dominated by $\phi^f$. Set $\phi_0 = \phi^f$ and $\phi_{k+1} = \TCb \phi_k$, i.e., for $t \geq 0$,
\[\phi_{k+1}(t) = h \lr{\int_0^t \phi_{k}(t-u)\mathrm d G_C(u) +1 - G_C(t)}.\]
Then, for all $t$,
$\phi_1(t) \leq \phi_0(t)$, and, by induction similarly as in the proof of Theorem \ref{thm::CP_FPE_IT_SMALLEST}, $\phi_{k+1}(t) \leq \phi_{k}(t)$, for all $k$.
Hence, for all $t \geq 0$, $\phi(t):= \displaystyle \lim_{k \to \infty} \phi_{k}(t)$ exists and $\phi(t)\leq \phi^f(t)$. Thus, by an appeal to LDC,
\[ \phi(t) = h \lr{\int_0^t \phi(t-u)\mathrm d G_C(u) +1 - G_C(t)}, \]
or,
\[ \phi = \TCb \phi. \]
Since the \emph{right} solution for the backward process $\phi^b$ is always the smallest solution by Theorem \ref{thm::AGE_right}, we have
\[ \phi^b(t) \leq \phi(t) \leq \phi^f(t),\]
for all $t \geq 0$, and thus
\[ \P{V^b < t} \geq \P{V^f < t},\]
for $t \geq 0$. We obtained the last inequality after recalling the probabilistic interpretation of $\phi^f$ (and $\phi^b$), that is, for $t \geq 0$,
\[ \ba \phi^f(t) &= 1 - \P{N^f(t) = \infty} \\
&= \P{V^f > t}.
\ea \]
\end{proof}

\newpage
\section{Forward Age-Dependent Branching Processes with Incubation Periods}
In this section we study the influence of an incubation period on age-dependent branching processes. 
That is, we investigate the explosiveness of the process $\HGIf$, where $h$ is a probability generating function and both $G$ and $I$ are distribution functions of non-negative random variables. 
Similarly to the contagious period case, we have for $N_f(t)$ the number of individuals in the coming generation at time $t$:
\be N_f(t) = \sum_{i=1}^D \indicator{ t \leq X_i,\tau^I \leq X_i} + \indicator{\tau^I \leq X_i < t} N^{(i)}_f(t-X_i), \ee
where $\left( N^{(i)}_f(t) \right)_i$ are i.i.d. copies of $N_f(t)$ for each $t \geq 0$.
Define $F: [0,1] \times [0,\infty) \to [0,1]$ for $(s,t) \in [0,1] \times  [0,\infty)$ by $F(s,t) = \E{s^{N_f(t)}} $, then for such $(s,t)$,
\[ F(s,t) =\sum_{k=1}^{\infty} \P{D=k} \int_{0}^\infty \lr{ \Ev\left[ s^{\indicator{ t \leq X_1, x \leq X_1} + \indicator{x \leq X_1 < t } N^{(1)}(t-X_1)}\right]}^k  \mathrm d I(x).\]
Considering the indicator functions, it is best to split the integral at $t$,
\[ \ba F(s,t) &=  \int_0^t h\lr{\Ev\left[ s^{\indicator{ t \leq X_1  }  + \indicator{x \leq X_1 < t }N_f^{(1)}(t-X_1)}\right]} \mathrm d I(x)\\
&\quad +\int_t^{\infty} h\lr{\Ev\left[ s^{\indicator{ x \leq X_1  }}\right]} \mathrm d I(x)
.\ea\] 
We proceed by integrating over $X_i$ (note that a factor $1$ occurs when an indicator is not satisfied), to obtain
\[ \ba F(s,t) &=  \int_0^t h\lr{G(x) + \int_x^t \Ev\left[ s^{N^{(1)}_f(t-u)}\right]\mathrm d G(u) + s(1-G(t))} \mathrm d I(x)\\
&\quad +\int_t^{\infty} h\lr{G(x) + s(1-G(x))} \mathrm d I(x)
.\ea\] 
Note that $\Ev\left[ s^{N^{(1)}_f(t-u)}\right] = F(s,t-u)$ and thus,
\[ \ba F(s,t) &=  \int_0^t h\lr{G(x) + \int_x^t F(s,t-u)\mathrm d G(u) + s(1-G(t))} \mathrm d I(x)\\
&\quad +\int_t^{\infty} h\lr{G(x) + s(1-G(x))} \mathrm d I(x)
.\ea\]
We see that, for $t \geq 0$, $\phi_f(t) = \displaystyle \lim_{s \uparrow 1} F(s,t) = 1 - \P{N_f(t) = \infty}$ satisfies
\[ \ba \phi_f(t) &=  \int_0^t h\lr{G(x) + \int_x^t \phi_f(t-u)\mathrm d G(u) + 1-G(t)} \mathrm d I(x) + 1-I(t) 
.\ea\]

\subsection{Fixed point equation}
To answer the explosiveness question for the $\HGIf$ process we study the fixed point equation \be \Phi = \TIf \Phi, \hspace*{2 cm} \Phi \in \mathcal{V}, \label{eq::IP_FOR_FPE} \ee
where $\mathcal{V}$ is defined in \eqref{eq::function_space}. The behaviour of this equation resembles exactly the behaviour of the fixed point equation describing the process with contagious periods. That is, the process is conservative if and only if \eqref{eq::IP_FOR_FPE} allows only \emph{one} solution in $\mathcal{V}$.
 The operator  $\TIf: \mathcal{V} \mapsto \mathcal{V}$ is defined for $\Phi \in \mathcal{V}$ by
\be (\TIf \Phi)(t) =  \int_0^t h\lr{G(x) + \int_x^t \Phi(t-u)\mathrm d G(u) + 1-G(t)} \mathrm d I(x) + 1-I(t), \quad t \geq 0.  \ee
We proceed now by a similar approach as in the contagious period counterpart: we find an iterative solution and show that it is both the \emph{right} and smallest solution in $\mathcal{V}$ to \eqref{eq::IP_FOR_FPE}.

\subsubsection{An iterative solution}
Our aim here is to show that the law of the explosion time may be obtained as a sequential limit. To this end, set $\phi_0 =0$, and, for $k \geq 0$, $\phi_{k+1} = \TIf \phi_k$. Indeed, by a similar induction argument as carried out in the last section, we see that
\begin{enumeratei}
\item  $(\phi_k(t))_k$ is a non-decreasing sequence for any fixed $t \geq 0$,
\item  $\phi(t) = \displaystyle \lim_{k \to \infty } \phi_k(t) \leq \psi(t)$ if $\psi$ is any positive solution to (\ref{eq::IP_FOR_FPE}) and in particular $\phi(t) \leq 1$ for all $t \geq 0$. 
\end{enumeratei}
LDC gives us that $\phi$ is indeed a solution. Hence, we can conclude Theorem \ref{thm:IP_FOR_FPE_IT_SMALLEST}:
\begin{theorem}\label{thm:IP_FOR_FPE_IT_SMALLEST} %Cont. Per. Fix. Pnt. Eq.Iterative solution is Smallest solution
$\phi$ defined for $t \geq 0$ by $\phi(t) = \displaystyle \lim_{k \to \infty} \phi_k(t) $ is the smallest solution to (\ref{eq::IP_FOR_FPE}), in the sense that if $\psi$ is any other non-negative solution, then $\phi(t) \leq \psi(t) $, for all $t \geq 0$.
\end{theorem}

\noindent The probabilistic interpretation of this iterative solution is the context of the next theorem:
\begin{theorem}\label{thm::IP_FPE_IT_RIGHT} %Cont. Per. Fix. Pnt. Eq. Iterative solution is right one
$\phi$ defined for $t \geq 0$ by $\phi(t) = \displaystyle \lim_{k \to \infty } \phi_k(t) $ is the right solution to (\ref{eq::IP_FOR_FPE}), that is, $\phi(t) = 1 -  \P{ N_f(t) = \infty }$, for all $ t \in [0,\infty)$. 
\end{theorem}
\begin{proof}
For the notation we refer to the proof of Theorem \ref{thm::CP_FPE_IT_RIGHT}. 
If we fix $t \geq 0$, then, for any $x \leq t$,
\[ h\lr{1 - G(t) + G(x)} = \mathbb{P}\left( \forall i \in \{ 1 , \ldots , D \}: X_i > t \mbox{ or } X_i \leq x \right), \]
and hence,
\be
\ba
\phi_1(t) &= \mathbb{P} \left( \tau^I \leq t, \forall i \in \{ 1 , \ldots , D \}: X_i > t \mbox{ or } X_i \leq \tau^I  \right) + \mathbb{P} \left(\tau^I > t \right) \\
&=1 - \mathbb{P} \lr{ \exists v \in \mbox{BP}(t) \mbox{ s.t. } |v| = 1  },
\ea
\ee
where $\tau^I \sim I$. 
By induction it then follows
\be
\ba
\phi_k(t) &= \mathbb{P}  \left( \tau^I \leq t, \forall i \in \{ 1 , \ldots , D \}: \neg \left( \exists v \in \mbox{BP}^{(i)}(t-X_i) \mbox{ s.t. } |v| = k-1, \tau^I \leq X_i \leq t \right) \right) \\
& \quad + \mathbb{P} \left( \tau^I > t \right) \\
&= 1 - \mathbb{P} \{ \exists v \in \mbox{BP}(t) \mbox{ s.t. } |v| = k  \}.
\ea
\ee
The steps, remaining to establish that $\phi$ is indeed the right solution, are identical to the ones carried out in the proof of Theorem \ref{thm::CP_FPE_IT_RIGHT}. 
\end{proof}

\noindent Just as in the contagious period setting, i.e., Corollary \ref{cor::CP_WEAKER}, the explosiveness question is completely decided by the behaviour of the process $\HGIf$ immediately after the root starts giving birth, which is a corollary to the theorems that we just proved: 
\begin{corollary}
\label{cor::IP_WEAKER}
Let $h$ be a probability generating function. Let both $G$ and $I$ be distribution functions of non-negative random variables. The process $\HGIf$ is explosive  if and only if there exists $T > 0$ and a function $\Psi: [0,T] \to [0,1]$ such that $\Psi \neq 1$ and 
\be 
\label{eq::IP_WEAKER}
\Psi(t) \geq \lr{\TIf \Psi}(t), \hspace*{1 cm} \forall t \in [0,T],
\ee
or, equivalently,
\begin{equation}
\Psi(t) \geq \int_0^t h \left( G(x) + \int_x^t \Psi(t-u)\mathrm d G(u) + 1 - G(t) \right) \mathrm d I(x) + 1 - I(t), \hspace*{1 cm} \forall t \in [0,T].
\end{equation}
\end{corollary}
\begin{proof}
Identical to the proof of Corollary \ref{cor::CP_WEAKER}.
\end{proof}

\noindent Moreover, the explosion time of an explosive process is never bounded away from $0$:

\begin{theorem}
If there exists $\tau > 0$ such that $\phi(\tau) < 1$, then $\phi < 1$ on $(0,\infty)$.
\end{theorem}
\begin{proof}
Suppose to the contrary that there exists $t > 0$ such that $\phi(t) = 1$, but then, as $\phi$ is non-increasing, there must exist $t_0 > 0$ such that
\begin{equation}
  \begin{array}{l l}
    \phi(t) = 1 & \quad \text{for $t \leq t_0$,}\\
    \phi(t) < 1 & \quad \text{for $t > t_0$.}
  \end{array}
\end{equation}
Then, for any $t > t_0$,
\begin{equation*}
\begin{split}
\phi(t) &= \int_0^{t-t_0} h \left( G(x) + \int_x^{t-t_0} \phi(t-u)\mathrm d G(u)   + \int_{t-t_0}^t \mathrm d G(u) +  1 - G(t) \right) \mathrm d I(x) \\
& \quad + \int_{t-t_0}^t h \left( G(x) + \int_x^{t} \mathrm d G(u)  +  1 - G(t) \right) \mathrm d I(x) + 1 - I(t) \\
&= \int_0^{t-t_0} h \left( G(x) + \int_x^{t-t_0} \phi(t-u)\mathrm d G(u)  +  1 - G(t-t_0) \right) \mathrm d I(x) + 1  - I(t-t_0).
\end{split}
\end{equation*}
Define $\phi^*: [0,\infty) \to [0,1]$ by
\begin{equation}
\phi^*(\tau) = \phi(\tau + t_0), \hspace*{1 cm} \tau \in [0,\infty),
\end{equation}
so that,
\begin{eqnarray*}
\phi^*(\tau) &=& \int_0^{\tau} h \left( G(x) + \int_x^{\tau} \phi(\tau + t_0-u) \mathrm d G(u) +  1 - G(\tau) \right) \mathrm d I(x)  + 1   - I(\tau) \\
&=& \int_0^{\tau} h \left( G(x) + \int_x^{\tau} \phi^*(\tau -u)  \mathrm d G(u)  +  1 - G(\tau) \right) \mathrm d I(x)  + 1   - I(\tau),
\end{eqnarray*}
in other words $\phi^*$ is a solution to (\ref{eq::IP_FOR_FPE}). But, $\phi^*(\frac{t_0}{2})<\phi (\frac{t_0}{2})$, which contradicts that $\phi$ is the smallest solution to (\ref{eq::IP_FOR_FPE}).\end{proof}

\subsection{Comparison theorems}
Corollary \ref{cor::IP_WEAKER} suggests that the behaviour of a process is completely determined around $0$, which we further investigate by proving some comparison theorems: 

\begin{theorem}
Let both $h$ and $h^*$ be probability generating functions such that for some $\theta < 1$, $h^*(s) \leq h(s)$ for all $s \in [\theta,1]$. Let $G$ be the distribution function of a non-negative random variable such that $G(0) = 0$. Let $I$ be the distribution function of a non-negative random variable. If the process $\HGIf$ is explosive, then so is $(h^*,G,I)_f$.
\end{theorem}
\begin{proof}
Consider the explosive process $\HGIf$ and denote by $\phi_f$ the corresponding smallest solution to (\ref{eq::IP_FOR_FPE}). We observe that, for all $t \geq 0$,
\begin{equation*}
1 \geq G(x) + \int_x^t \phi_f(t-u)\mathrm d G(u) + 1 - G(t) \geq 1 - G(t). 
\end{equation*}
Hence, since $G(0)=0$, there exists  $T>0$ such that, for $t \in [0,T]$, 
\[G(x) + \int_x^t \phi_f(t-u)\mathrm d G(u) + 1 - G(t) \geq \theta, \]
and thus, for those $t$, 
\begin{equation*}
h \lr{G(x) + \int_x^t \phi_f(t-u)\mathrm d G(u) + 1 - G(t)} \geq h^* \lr{G(x) + \int_x^t \phi_f(t-u)\mathrm d G(u) + 1 - G(t)}. 
\end{equation*}
We establish that, for any $t \in [0,T]$,
\be
\phi_f(t) \geq \int_0^t h^* \lr{G(x) + \int_x^t \phi_f(t-u)\mathrm d G(u) + 1 - G(t)} \mathrm d I(x) + 1 - I(t),
\ee
or, 
\be 
\phi_f(t) \geq \lr{ T_f^{(h^*,G,I)} \phi_f }(t).
\ee
An appeal to Corollary \ref{cor::IP_WEAKER}, with $\Psi = \phi_f$, completes the proof.
\end{proof}

\begin{theorem}
Let $h$ be a probability generating function. 
Let both $G$ and $ G^*$ be distribution functions of non-negative random variables such that $G(0)=G^*(0)=0$, with densities $g$ and $g^*$ respectively. Assume that there exists a $T>0$ such that $g^* \geq g$ on $[0,T]$. Let $I$ be the distribution function of a non-negative random variable. If the process $\HGIf$ is explosive, then so is $(h,G^*,I)_f$. 
\end{theorem}
\begin{proof}
Let $\phi_f$ be the smallest solution to \eqref{eq::IP_FOR_FPE}, corresponding to the explosive process $\HGIf$. Set for $0 \leq u \leq t \leq T$, $\phi_t(u) = \phi_f(t-u)$, then,  for $x \leq t$, since $\phi_t(t) = 1$,
\be \ba
G(x) + \int_x^t \phi_t(u) \mathrm d G(u) + 1 - G(t) &= G(x)(1-\phi_t(x)) + 1 - \int_x^t G(u)\mathrm d \phi_t(u) \\
&= \int_x^t G(x)\mathrm d \phi_t(u) + 1 - \int_x^t G(u)\mathrm d \phi_t(u),
\label{eq::CP_COMP_G_1}
 \ea \ee
after writing the first term in the right hand side of the first line as an integral. Now,
\[ G(x) - G(u) = \int_u^x g(v) \mathrm dv = - \int_x^u g(v) \mathrm dv,  \]
and we obtain
\be \int_x^t G(x)\mathrm d \phi_t(u) - \int_x^t G(u)\mathrm d \phi_t(u) = - \int_x^t \int_x^u g(v) \mathrm dv \mathrm d \phi_t(u). \label{eq::CP_COMP_G_2} \ee
Combining \eqref{eq::CP_COMP_G_1} and \eqref{eq::CP_COMP_G_2}, we have
\be \ba
G(x) + \int_x^t \phi_t(u)\mathrm d G(u) + 1 - G(t) &= 1 - \int_x^t \int_x^u g(v) \mathrm dv \mathrm d \phi_t(u) \\
& \geq 1 - \int_x^t \int_x^u g^*(v) \mathrm dv \mathrm d \phi_t(u) \\
& = G^*(x) + \int_x^t \phi_t(u)g^*(u) \mathrm du + 1 - G^*(t),
 \ea \ee
by a similar exercise. We find, for $t \leq T$, 
\[ \phi_f(t) = \lr{\TIf \phi_f}(t) \geq \lr{T_f^{(h,G^*,I)}\phi_f}(t), \]
and by Corollary \ref{cor::IP_WEAKER} the process $(h,G^*,I)_f$ must be explosive.
\end{proof}
\begin{remark}
We cannot hope that, for instance, $G^* \geq G$ on some interval would also lead to a comparison theorem. Since, if the life-lengths become shorter, they might become even shorter than the incubation periods. Hence, we must have, for small $t \geq 0$,  \[\P{X \leq t| X \geq \tau^I} \leq \P{X^* \leq t| X^* \geq \tau^I},\] where $X \sim G$, $X^* \sim G^*$ and $\tau^I \sim I$. This is precisely reflected in $g^* \geq g$ on $[0,T]$.
\end{remark}

\begin{theorem}
Let $h$ be a probability generating function. 
Let  $G$ be the distribution function of a non-negative random variable. 
Let both $I$ and $I^*$ be distribution functions of non-negative random variables such that $I(0)=I^*(0)=0$. Assume there exists a $T>0$ such that $I \leq I^*$ on $[0,T]$. If the process $\HGIf$ is explosive, then so is $(h,G,I^*)_f$. 
\label{thm::IP_FOR_COMP_I}
\end{theorem}
\begin{proof}
Let $\phi_f$ be the smallest solution in $\mathcal{V}$ such that $\phi_f = \TIf \phi_f$. Fix $T \geq t \geq 0$ and define $h_t$ for $x \leq t$ as
\be h_t(x) = h \lr{ G(x) + \int_x^t \phi_f(t-u)\mathrm d G(u) + 1 - G(t) }. \ee
We will  see shortly that $h_t$ is increasing, which establishes the result, since then, for $t \leq T$,
\be \ba
\phi_f(t) &= \lr{\TIf \phi_f}(t) \\
&= \int_0^t h_t(x)  \mathrm d I(x) + 1 - I(t) \\
&= 1 -  \int_0^t I(x) \mathrm d h_t(x),
 \ea \ee
by partial integration, because $h_t(t) = 1$ and $I(0)=0$. Next, note that $\mathrm d h_t$ is non-negative by the claim, and thus
\be \phi_f(t) \geq 1 - \int_0^t I^*(x) \mathrm d h_t(x) = \lr{T_f^{(h,G,I^*)}\phi_f}(t),\ee
and Corollary \ref{cor::IP_WEAKER} ensures that $(h,G,I^*)_f$ is explosive. 
It thus remains to verify the monotonicity of $h_t$. Fix $t \geq x' \geq x$ and define, for $u \leq t$, 
\[ f_t(u) = G(u) + \int_u^t \phi_f(t-v) \mathrm d G(v) + 1 - G(t), \]
to see that 
\be \ba
f_t(x') - f_t(x) &= G(x') - G(x) + \int_{x'}^t \phi_f(t-v) \mathrm d G(v) - \int_{x}^t \phi_f(t-v) \mathrm d G(v) \\
&= G(x') - G(x) - \int_{x}^{x'} \phi_f(t-v) \mathrm d G(v).
\ea \ee
Now, for all $u \leq t$, we know that $0 \leq \phi_f(t-u) \leq 1$, and hence 
\[ f_t(x') - f_t(x) \geq 0, \]
and as $h$ is increasing,
\[ h_t(x') - h_t(x) = h \lr{f_t(x')} - h \lr{f_t(x)} \geq 0,\]
which was the only assertion left in need for a proof. 
\end{proof}

\subsection{The explosiveness question}
Here we investigate what conditions on $G$ and $I$ make a process $\HGIf$ explosive. We conjecture that it is explosive if and only if $\HG$ and $(h,I)$  both explode. In this section we prove one side of the conjecture:  the explosiveness of $\HG$ and $(h,I)$ is  necessary. Then, in the next section, we shall shift our attention  to the backward process and show that the conjecture holds for it under a very mild condition on $G$ and $I$. In the last section, we shall deduce explosiveness in the forward process from explosion in the backward process by presenting an algorithm that finds an exploding path.
\\
We can construct a realization of the $\HGIf$ process in the following way. First pick a realization of $\HG$ and take a tree point of view. That is, consider it as a realization of a Galton-Watson tree with offspring probability generating function $h$ and (life-time) weights on the edges that are independently chosen from some distribution $G$. Then, give each vertex a(n) (incubation) weight according to some distribution $I$. Remove every edge for which its weight is smaller than the incubation on its upper vertex. The connection is now made by identifying the desired realization of $\HGIf$ with the component of the root. 
\\
Hence, it is clear that a realization of $\HGIf$ may only explode if the underlying realization(s) of $\HG$ is/are explosive. Thus, explosiveness of $\HG$ is necessary for $\HGIf$  to be explosive. That it also necessitates the explosiveness of $(h,I)$  is the object of study in  Theorem \ref{thm::IP_for_necessary}: 

\begin{proof}[Proof of Theorem \ref{thm::IP_for_necessary}]
It is equivalent to prove that the explosiveness of $\HGIf$  implies that $(h,I)$ is explosive. 
Consider an ordinary Galton-Watson process with offspring probability generating function $h$ where to each vertex $v$ two random variables $\tau^I_v \sim I$ and $X_v \sim G$ are attached (see Fig. \ref{fig::IP_FOR_NECESSARY_R}). Denote the root by $r$. We say, for a given edge $e = \{u,v \}$, that $u$ is its upper vertex if $v$ is a child of $u$ and denote this ordered edge by $uv$.
\\
The process $\HGIf$ can be thought of as the time-continuous $BP^{(1)}$ equal to the component of the root in the graph that we obtain by letting each edge $uv$ have length
\[ L^{(1)}_{uv} = X_v, \]
and remove it if $\tau^I_u > X_v$, see Fig. \ref{fig::IP_FOR_NECESSARY_BP1}.
\\
We also consider the time-continuous branching process $BP^{(2)}$ where each edge $uv$ has length
\[ L^{(2)}_{uv} = \tau^I_u, \]
see Fig. \ref{fig::IP_FOR_NECESSARY_BP2}.
Hence, to each realization $R$ of the Galton-Watson process enriched with two random variables attached to its vertices correspond realizations $R^{(1)}$ and $R^{(2)}$ of $BP^{(1)}$, respectively $BP^{(2)}$. If $R^{(1)}$ contains a finite ray, then $R^{(2)}$ must contain such a ray as well. Indeed, every surviving edge $uv$ in $R^{(1)}$ has length
\[ L^{(1)}_{uv} = X_v \geq \tau^I_u = L^{(2)}_{uv},\]
and we see that the exact same ray is finite in  $BP^{(2)}$. 
\\
Consider a third time-continuous branching process $BP^{(3)}$, where the root is born at time $\tau^I_r$. We let every other edge $uv$ have length $\tau^I_v$ (see Fig. \ref{fig::IP_FOR_NECESSARY_BP3}). Every realization $R^{(3)}$ of $BP^{(3)}$ can mapped bijectively to a realization $R^{(2)}$ of $BP^{(2)}$, upon noting that both are obtained by shifting the incubation periods on the vertices in a realization $R$ either up or down. Moreover, due to this shift, a ray in one realization is finite if and only if it is finite in the other realization. Thus $BP^{(3)}$ is explosive if and only if $BP^{(2)}$ is explosive. Also, $BP^{(3)}$ is just the process $(h,I)$ with the root now born at time $\tau^I_r$. We conclude that $(h,I)$ is explosive if the process $\HGIf$ explodes.
\end{proof}

 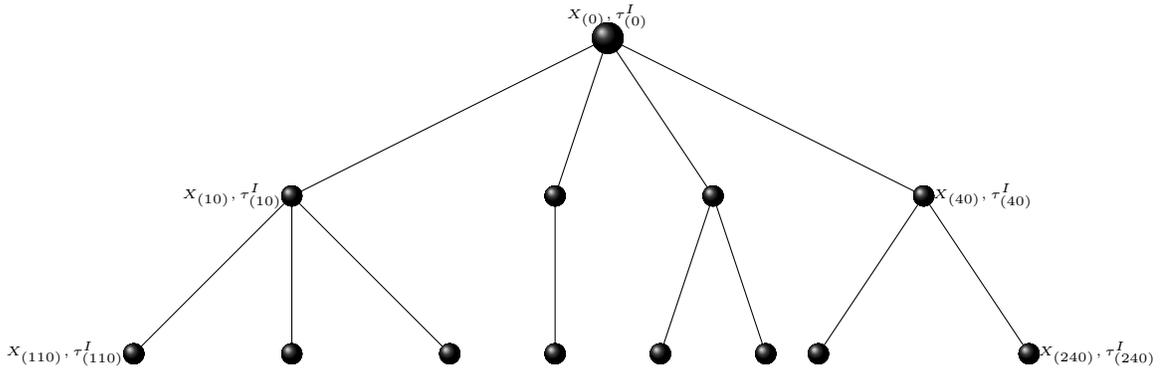
\begin{figure}

\centering
\begin{tikzpicture}[scale=1.4,
dotr/.style={circle,fill=black,minimum size=4pt,inner sep=0pt,
            outer sep=-1pt},
dot/.style={circle,fill=black,minimum size=3pt,inner sep=0pt,
            outer sep=-1pt}           ],
\tikzstyle help lines=[color=gray,very thin]
\tikzstyle en grid=[style=help lines]

\coordinate[label=above:\tiny {$X_{(0)},  \tau_{(0)}^I $}] (r) at (0,0);

\coordinate[label=left:\tiny {$X_{(10)}, \tau_{(10)}^I$}] (r1) at (-3, -1.5);
\coordinate[label=left:\tiny {$X_{(110)}, \tau_{(110)}^I$}] (r11) at (-4.5, -3);
\coordinate[label=left:$ $] (r12) at (-3, -3);
\coordinate[label=below:\tiny {$ $}] (r13) at (-1.5, -3);

\coordinate[label=above:\tiny {$ $}] (r2) at (-0.5, -1.5);
\coordinate[label=above:\tiny {$ $}] (r21) at (-0.5, -3);

\coordinate[label=above:\tiny {$ $}] (r3) at (1, -1.5);
\coordinate[label=above:\tiny {$ $}] (r31) at (0.5, -3);
\coordinate[label=above:\tiny {$ $}] (r32) at (1.5, -3);

\coordinate[label=right:\tiny {$X_{(40)}, \tau_{(40)}^I$}] (r4) at (3, -1.5);
\coordinate[label=below:\tiny {$ $}] (r41) at (2, -3);
\coordinate[label=right:\tiny {$X_{(240)}, \tau_{(240)}^I$}] (r42) at (4, -3);

\draw (r) -- (r1);
\draw (r1) -- (r11);
\draw (r1) -- (r12);
\draw (r1) -- (r13);

\draw (r) -- (r2);
\draw (r2) -- (r21);

\draw (r) -- (r3);
\draw (r3) -- (r31);
\draw (r3) -- (r32);

\draw (r) -- (r4);
\draw (r4) -- (r41);
\draw (r4) -- (r42);
 \shade[ball color=black] (r) circle (0.15);
 \shade[ball color=black] (r1) circle (0.1);
 \shade[ball color=black] (r11) circle (0.1);
 \shade[ball color=black] (r12) circle (0.1);
 \shade[ball color=black] (r13) circle (0.1);
 \shade[ball color=black] (r2) circle (0.1);
 \shade[ball color=black] (r21) circle (0.1);
 \shade[ball color=black] (r3) circle (0.1);
 \shade[ball color=black] (r31) circle (0.1);
 \shade[ball color=black] (r32) circle (0.1);
 \shade[ball color=black] (r4) circle (0.1);
 \shade[ball color=black] (r41) circle (0.1);
 \shade[ball color=black] (r42) circle (0.1);
 \end{tikzpicture}
\caption{A possible realization $R$ of the Galton-Watson process enriched with random variables $X_v \sim G$ and $\tau_v^I \sim I$ attached to each of its vertices $v$. The notation $(ij0)$ means that $i$ is a child of $j$ who has the root $0$ as a mother. Note that not all random variables are drawn and that all edges are assumed to have unit length and are \emph{yet} unweighted.}
\label{fig::IP_FOR_NECESSARY_R}
\end{figure}

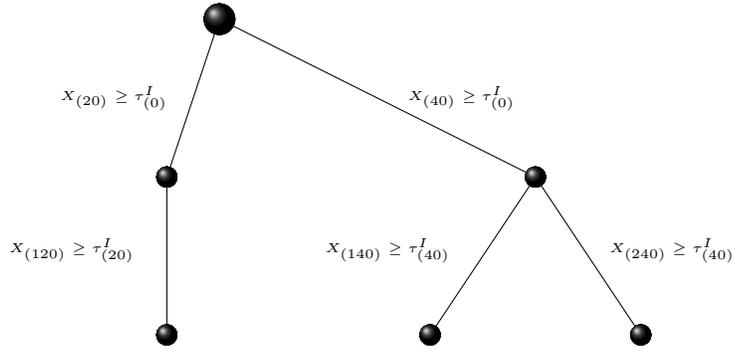
\begin{figure}
\centering
\begin{tikzpicture}[scale=1.4,
dotr/.style={circle,fill=black,minimum size=4pt,inner sep=0pt,
            outer sep=-1pt},
dot/.style={circle,fill=black,minimum size=3pt,inner sep=0pt,
            outer sep=-1pt}           ],
\tikzstyle help lines=[color=gray,very thin]
\tikzstyle en grid=[style=help lines]

\coordinate[label=above:\tiny {$ $}] (k) at (0,0);
\coordinate[label={[black]center:\tiny  $X_{(20)} \geq \tau_{(0)}^I$}] (l) at (-1,-0.75);
\coordinate[label={[black]center:\tiny  $X_{(40)} \geq \tau_{(0)}^I$}] (l) at (2.3,-0.75);

\coordinate[label=left:\tiny {$ $}] (k1) at (-3, -1.5);
\coordinate[label=left:\tiny {$ $}] (k11) at (-4.5, -3);
\coordinate[label=left:$ $] (k12) at (-3, -3);
\coordinate[label=below:\tiny {$ $}] (k13) at (-1.5, -3);

\coordinate[label=above:\tiny {$ $}] (k2) at (-0.5, -1.5);
\coordinate[label=above:\tiny {$ $}] (k21) at (-0.5, -3);
\coordinate[label={center:\tiny  $X_{(120)} \geq \tau_{(20)}^I$}] (l) at (-1.4,-2.2);

\coordinate[label=above:\tiny {$ $}] (k3) at (1, -1.5);
\coordinate[label=above:\tiny {$ $}] (k31) at (.5, -3);
\coordinate[label=above:\tiny {$ $}] (k32) at (1.5, -3);

\coordinate[label=right:\tiny {$ $}] (k4) at (3, -1.5);
\coordinate[label=below:\tiny {$ $}] (k41) at (2, -3);
\coordinate[label=right:\tiny {$ $}] (k42) at (4, -3);
\coordinate[label={center:\tiny  $X_{(140)} \geq \tau_{(40)}^I$}] (l) at (1.6,-2.2);
\coordinate[label={center:\tiny  $X_{(240)} \geq \tau_{(40)}^I$}] (l) at (4.3,-2.2);

\draw (k) -- (k2);
\draw (k2) -- (k21);

\draw (k) -- (k4);
\draw (k4) -- (k41);
\draw (k4) -- (k42);
 \shade[ball color=black] (k) circle (0.15);
 
 \shade[ball color=black] (k2) circle (0.1);
 \shade[ball color=black] (k21) circle (0.1);
 
 \shade[ball color=black] (k4) circle (0.1);
 \shade[ball color=black] (k41) circle (0.1);
 \shade[ball color=black] (k42) circle (0.1);
 \end{tikzpicture}
\caption{A possible realization of $BP^{(1)}$ corresponding to $R$ in Fig. \ref{fig::IP_FOR_NECESSARY_R}. Note that weights are now put on the edges and only the surviving edges are drawn, i.e., those for which the infection time is larger than the incubation time on its upper vertex. }
\label{fig::IP_FOR_NECESSARY_BP1}
\end{figure}

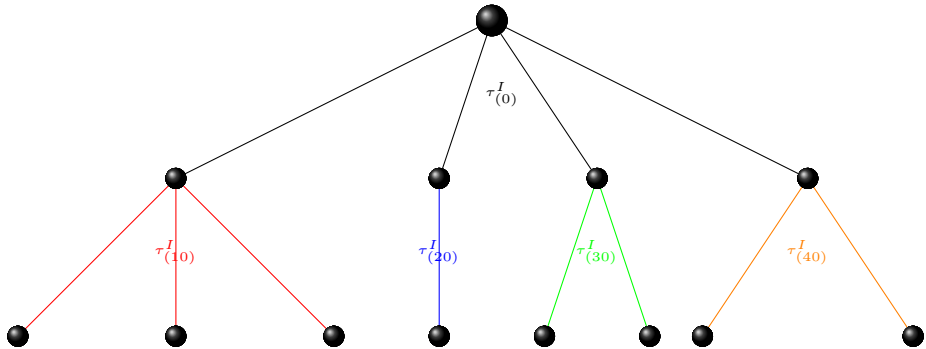
\begin{figure}
\centering
\begin{tikzpicture}[scale=1.4,
dotr/.style={circle,fill=black,minimum size=4pt,inner sep=0pt,
            outer sep=-1pt},
dot/.style={circle,fill=black,minimum size=3pt,inner sep=0pt,
            outer sep=-1pt}           ],
\tikzstyle help lines=[color=gray,very thin]
\tikzstyle en grid=[style=help lines]

\coordinate[label=above:\tiny {$ $}] (r) at (0,0);
\coordinate[label={[black]center:\tiny  $\tau_{(0)}^I$}] (l) at (0.1,-0.7);

\coordinate[label=left:\tiny {$ $}] (r1) at (-3, -1.5);
\coordinate[label=left:\tiny {$ $}] (r11) at (-4.5, -3);
\coordinate[label=left:$ $] (r12) at (-3, -3);
\coordinate[label=below:\tiny {$ $}] (r13) at (-1.5, -3);
\coordinate[label={[red]center:\tiny  $\tau_{(10)}^I$}] (l) at (-3,-2.2);

\coordinate[label=above:\tiny {$ $}] (r2) at (-0.5, -1.5);
\coordinate[label=above:\tiny {$ $}] (r21) at (-0.5, -3);
\coordinate[label={[blue]center:\tiny  $\tau_{(20)}^I$}] (l) at (-0.5,-2.2);

\coordinate[label=above:\tiny {$ $}] (r3) at (1, -1.5);
\coordinate[label=above:\tiny {$ $}] (r31) at (0.5, -3);
\coordinate[label=above:\tiny {$ $}] (r32) at (1.5, -3);
\coordinate[label={[green]center:\tiny  $\tau_{(30)}^I$}] (l) at (1,-2.2);

\coordinate[label=right:\tiny {$ $}] (r4) at (3, -1.5);
\coordinate[label=below:\tiny {$ $}] (r41) at (2, -3);
\coordinate[label=right:\tiny {$ $}] (r42) at (4, -3);
\coordinate[label={[orange]center:\tiny  $\tau_{(40)}^I$}] (l) at (3,-2.2);

\draw (r) -- (r1);
\draw [red](r1) -- (r11);
\draw [red](r1) -- (r12);
\draw [red](r1) -- (r13);

\draw (r) -- (r2);
\draw [blue](r2) -- (r21);

\draw (r) -- (r3);
\draw [green](r3) -- (r31);
\draw [green](r3) -- (r32);

\draw (r) -- (r4);
\draw [orange](r4) -- (r41);
\draw [orange](r4) -- (r42);
 \shade[ball color=black] (r) circle (0.15);
 \shade[ball color=black] (r1) circle (0.1);
 \shade[ball color=black] (r11) circle (0.1);
 \shade[ball color=black] (r12) circle (0.1);
 \shade[ball color=black] (r13) circle (0.1);
 \shade[ball color=black] (r2) circle (0.1);
 \shade[ball color=black] (r21) circle (0.1);
 \shade[ball color=black] (r3) circle (0.1);
 \shade[ball color=black] (r31) circle (0.1);
 \shade[ball color=black] (r32) circle (0.1);
 \shade[ball color=black] (r4) circle (0.1);
 \shade[ball color=black] (r41) circle (0.1);
 \shade[ball color=black] (r42) circle (0.1);
 \end{tikzpicture}
\caption{A possible realization of $BP^{(2)}$ corresponding to $R$ in Fig. \ref{fig::IP_FOR_NECESSARY_R}. All edges have a weight equal to the incubation period of its upper vertex in $R$. Note that same-colour edges have equal weight. }
\label{fig::IP_FOR_NECESSARY_BP2}
\end{figure}

\begin{figure}
\centering
\begin{tikzpicture}[scale=1.4,
dotr/.style={circle,fill=black,minimum size=4pt,inner sep=0pt,
            outer sep=-1pt},
dot/.style={circle,fill=black,minimum size=3pt,inner sep=0pt,
            outer sep=-1pt}           ],
\tikzstyle help lines=[color=gray,very thin]
\tikzstyle en grid=[style=help lines]

\coordinate[label=above:\tiny {$ $}] (rr) at (10,1.5);

\coordinate[label=above:\tiny {$ $}] (k) at (10,0);
\coordinate[label={[black]left:\tiny  $\tau_{(0)}^I$}] (l) at (10,0.7);

\coordinate[label=left:\tiny {$ $}] (k1) at (7, -1.5);
\coordinate[label=left:\tiny {$ $}] (k11) at (5.5, -3);
\coordinate[label=left:$ $] (k12) at (7, -3);
\coordinate[label=below:\tiny {$ $}] (k13) at (8.5, -3);
\coordinate[label={[red]center:\tiny  $\tau_{(10)}^I$}] (l) at (8,-0.7);

\coordinate[label=above:\tiny {$ $}] (k2) at (9.5, -1.5);
\coordinate[label=above:\tiny {$ $}] (k21) at (9.5, -3);
\coordinate[label={[blue]center:\tiny  $\tau_{(20)}^I$}] (l) at (9.4,-0.7);

\coordinate[label=above:\tiny {$ $}] (k3) at (11, -1.5);
\coordinate[label=above:\tiny {$ $}] (k31) at (10.5, -3);
\coordinate[label=above:\tiny {$ $}] (k32) at (11.5, -3);
\coordinate[label={[green]center:\tiny  $\tau_{(30)}^I$}] (l) at (10.8,-0.7);

\coordinate[label=right:\tiny {$ $}] (k4) at (13, -1.5);
\coordinate[label=below:\tiny {$ $}] (k41) at (12, -3);
\coordinate[label=right:\tiny {$ $}] (k42) at (14, -3);
\coordinate[label={[orange]center:\tiny  $\tau_{(40)}^I$}] (l) at (12,-0.7);

\draw (k) -- (rr);

\draw [red](k) -- (k1);
\draw [gray](k1) -- (k11);
\draw [gray](k1) -- (k12);
\draw [gray](k1) -- (k13);

\draw [blue](k) -- (k2);
\draw [gray](k2) -- (k21);

\draw [green](k) -- (k3);
\draw [gray](k3) -- (k31);
\draw [gray](k3) -- (k32);

\draw [orange](k) -- (k4);
\draw [gray](k4) -- (k41);
\draw [gray](k4) -- (k42);
 \shade[ball color=black] (rr) circle (0.15);
 \shade[ball color=black] (k) circle (0.15);
 \shade[ball color=black] (k1) circle (0.1);
 \shade[ball color=black] (k11) circle (0.1);
 \shade[ball color=black] (k12) circle (0.1);
 \shade[ball color=black] (k13) circle (0.1);
 \shade[ball color=black] (k2) circle (0.1);
 \shade[ball color=black] (k21) circle (0.1);
 \shade[ball color=black] (k3) circle (0.1);
 \shade[ball color=black] (k31) circle (0.1);
 \shade[ball color=black] (k32) circle (0.1);
 \shade[ball color=black] (k4) circle (0.1);
 \shade[ball color=black] (k41) circle (0.1);
 \shade[ball color=black] (k42) circle (0.1);
 \end{tikzpicture}
 \caption{A possible realization of $BP^{(3)}$ corresponding to $R$ in Fig. \ref{fig::IP_FOR_NECESSARY_R}. The root is born at time equal to the incubation time of the root in $R$. All other edges have as a weight the incubation time on their lower vertex in $R$. Compare to Fig. \ref{fig::IP_FOR_NECESSARY_BP2}.}
\label{fig::IP_FOR_NECESSARY_BP3}
\end{figure}
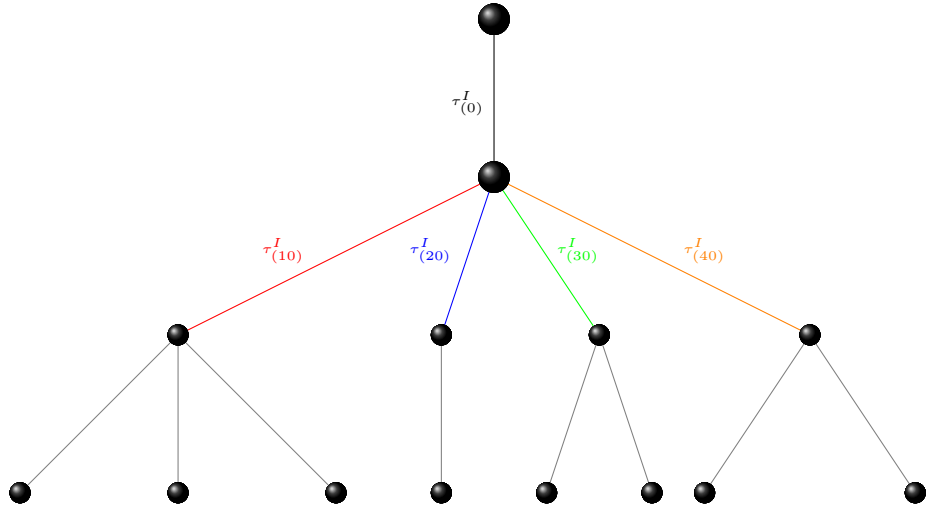

\newpage\null\thispagestyle{empty}\newpage
\newpage\null\thispagestyle{empty}\newpage

\section{Backward Age-Dependent Branching Processes with Incubation Periods}
Object of study is the process $\HGIb = (h,G_I)$, where $h$ is a probability generating function and both $G$ and $I$ are the distribution functions of non-negative random variables. We calculate, 
for $x < \infty$, that 
\be
\ba
G_I(x) &= \P{\tau^I \leq X \leq x} \\
 & = \int_0^x I(u)\mathrm d G(u).
\label{eq::IP_BACK_GI}
\ea
\ee
As always, we start analysing  this process by studying its characteristic fixed point equation.

\subsection{Fixed point equation}
Since $\HGIb = (h,G_I)$, the fixed point equation to study is
\be \Phi = \TIb \Phi , \hspace*{2 cm} \Phi \in \mathcal{V}, \ee
conform \eqref{eq::AGE_FP_usual}. The operator $\TIb = T_{(h,G_I)}$ is defined for $\Phi \in \mathcal{V}$ as
\be \ba 
\lr{T_{(h,G_I)} \Phi }(t)&=  h\left( \int_0^t \Phi(t-u) \mathrm d G_I(u) + 1-G_I(t) \right) \\
&=  h\left(\int_0^t \Phi(t-u) I(u)\mathrm d G(u) +  1- \int_0^t I(u)\mathrm d G(u) \right), \quad t \geq 0.
\ea \ee

\subsection{The explosiveness question}
Here we prove part of the conjecture that
\[ \HGIb \mbox{ is explosive } \Leftrightarrow \HG \mbox{ and } (h,I) \mbox{ are both explosive.} \]
That is, we prove Theorems \ref{thm::IP_back_necessary} and \ref{thm::IP_back_sufficient}:
\begin{proof}[Proof of Theorem \ref{thm::IP_back_necessary}]
We have $\HGIb = (h,G_I)$, where, for $x < \infty$, 
\[ G_I(x) = \int_0^x I(u)\mathrm d G(u) \leq I(x) \int_0^x \mathrm d G(u)=I(x)G(x). \]
Hence, $G_I \leq I$ and $G_I \leq G$ on $[0,\infty)$ and comparison Theorem \ref{thm::AGE_COMP_G} establishes the result.
\end{proof}

\begin{proof}[Proof of Theorem \ref{thm::IP_back_sufficient}]
We shall make use of the fact that $(\Ha,\frac{1}{2}G^2)$ and $(\Ha,\frac{1}{2}I^2)$ are both explosive. On the one hand, if $I\geq G$, then
\[G_I(x) = \int_0^x I(u) \mathrm d G(u) \geq \int_0^x G(u) \mathrm d G(u) = \frac{1}{2}G^2(x). \]
On the other hand, if the second assumption holds, we may restrict to the case where the density of $I$ extends to $[0,\delta)$. Indeed, if $I(0) > 0$, then fix $x \in [0,\delta)$ to have
\[ G_I(x) = \int_0^x I(u)\mathrm d G(u) \geq I(0)G(x). \] Theorem \ref{th::AGE_CONSTANT} proves that $(\Ha,I(0)G)$ is explosive and comparison Theorem \ref{thm::AGE_COMP_G} establishes explosiveness of $(\Ha,G_I)$. 
\\
Hence, we may assume that, for $0 \leq x \leq \delta$,
\[G(x) = \int_0^x g(u) \mathrm d u,\]
and,
\[I(x) = \int_0^x i(u) \mathrm du.\]
Because $g \geq i$ on $(0,\delta)$, we have, for any $0 \leq x \leq \delta$,
\[ \ba
G_I(x) &= \int_0^x I(u) g(u) \mathrm d u \\
& \geq \int_0^x I(u) i(u) \mathrm du \\
& = \frac{1}{2}I^2(x).
\ea \]
It remains to verify that $(\Ha,\frac{1}{2}G^2)$ and $(\Ha,\frac{1}{2}I^2)$ are both explosive, as the proof is then finished upon an appeal to comparison Theorem \ref{thm::AGE_COMP_G}. 
\\
But, Theorem \ref{th::AGE_POWER} tells us that $(\Ha,G^2)$ and $(\Ha,I^2)$ are explosive, and the explosiveness of $(\Ha,\frac{1}{2}G^2)$ and $(\Ha,\frac{1}{2}I^2)$ is precisely the content of Theorem \ref{th::AGE_CONSTANT}.
\end{proof}
 
\newpage

\section{Comparison of Backward and Forward Processes with Incubation Periods}
This section is entirely devoted to the proofs of Theorems \ref{thm::IP_COMPARING_BW_FW}, \ref{thm::IP_COMPARISON_AGING} and \ref{thm::IP_COMPARISON_MAIN}:

\begin{proof}[Proof of Theorem \ref{thm::IP_COMPARING_BW_FW}]
We show that that, for all $t \geq 0$, $\phi_f(t) \geq \phi_b(t)$, where $\phi_f$ is the smallest solution to 
\be \Phi = \TIf \Phi, \hspace*{2 cm} \Phi \in \mathcal{V}, \ee
and $\phi_b$ is the smallest solution to
\be \Phi = \TIb \Phi, \hspace*{2 cm} \Phi \in \mathcal{V}. \ee
By \eqref{eq::IP_FOR_FPE} we have, for $t \geq 0$,
\[ \phi_f(t) =  \int_0^{\infty} h\lr{ \indicator{[0,t]}(x) \lr{G(x) + \int_x^t \phi_f(t-u)\mathrm d G(u) + 1-G(t)} + \indicator{(t,\infty)}(x)}  \mathrm d I(x), \]
which we may bound from below due to the convexity of $h$,
\be \ba \phi_f(t) &\geq   h\lr{ \int_0^{t} \lr{G(x) + \int_x^t \phi_f(t-u)\mathrm d G(u) + 1-G(t)} \mathrm d I(x) + 1 - I(t)} \\
&= h\left( 1- \int_0^t I(u)\mathrm d G(u) + \int_0^t \phi_f(t-u) I(u)\mathrm d G(u) \right) , \ea \ee
upon changing the order of integration in the double integral and carrying out partial integration on the other terms. Hence, recall \eqref{eq::IP_BACK_GI},
\[ \phi_f(t) \geq  h\left( 1-G_I(t) + \int_0^t \phi_f(t-u) \mathrm d G_I(u) \right), \]
and $\phi_f(t) \geq \phi_b(t)$ follows from constructing an iterative solution to 
\[ \Phi = \TIb \Phi, \hspace*{2 cm} \Phi \in \mathcal{V}, \]
in the same fashion as in the proof of Theorem \ref{th::CP_COMPARING_BW_FW}. The proof is complete. Indeed, $\phi_f(t) \geq \phi_b(t)$ implies $V_b \overset{d}\leq V_f$, hence, $V_f < t$ with positive probability implies $V_b < t$ with positive probability. Thus, if the forward process is explosive, then so is the backward process.
\end{proof}

\noindent The following proof is inspired by ideas of a proof presented in \citep{AmDeGrNe13}, see Section 1.8 for details. However, as we cannot make use of  independence, it is very different from the one given there:

\begin{proof}[Proof of Theorem  \ref{thm::IP_COMPARISON_AGING}]
We start with a few definitions. Let $c = \frac{1}{2} \mathbb{P}(X > \delta)$ and $m$ a constant so large that 
\be \frac{1}{8} cm^{1 - \sqrt{\alpha}} \geq e . \label{eq::IP_COMPARISON_m}\ee
Define $f$ by $f(0) = m$, and, 
\be f(n+1) = F^{\leftarrow}_D(1 - \frac{1}{f(n)^{\sqrt{\alpha}}}), \quad n > 0 ,\ee
so that 
\be f(n) = m^{1 / (\sqrt{\alpha})^n}, n \geq 0. \ee
Let $N$ be a constant so large that
\be 1 - 2 \exp( -\exp (1 / (\sqrt{\alpha})^n)) \geq 1 - \frac{1}{n^2}, \quad \forall n \geq N. \label{eq::IP_COMPARISON_N} \ee
Further, let $\lr{X^{n}_i}_{i,n}$ and $\lr{\widehat{X}^{n}_i}_{i,n}$ be i.i.d. sequences with distribution $G$. Let $\lr{I^n_{i}}_{i,n}$ be i.i.d. according to some distribution $I$ (note that we introduced a new notation for the incubation periods). The core of the proof is an algorithm that finds with positive probability a finite ray or exploding path (i.e., a path with \emph{infinitely} many vertices on it, each being a descendant of her predecessor, such that the sum of life-times over the entire path is \emph{finite}). The path starts at a vertex that we denote by $x_N$, with degree at least $f(N)$. 
For $n = N,N+1, \cdots$:

\begin{itemize}
\item Consider node $x_n$, the lowest node in the candidate exploding path. We denote by $D_n$ its number of possible children and by $I^n_{k_n}$ its incubation period.
\item If at least $c f(n)$ children survive, i.e., $D^e_n \geq c f(n)$, order those alive children of $x_n$ by how many children they in turn have, from largest to smallest. Let 
\be W_n = \frac{c f(n)^{1-\sqrt{\alpha}}}{2}  \label{eq::IP_COMPARISON_Wn}. \ee Define the options from $x_n$ to be the first $W_n$ children in the ordering (note that $W_n <  (c f(n))^{1-\sqrt{\alpha}}/2 < D^e_n$).
Label those options in this ordering and denote their life-lengths, respectively incubation periods, by $X^{n+1}_i$ and $I^{n+1}_i$ for $i \in \{1 , \cdots, W_n \}$.
\\
If, however, $D_n^e < c f(n)$, the algorithm terminates in failure. 
\item Put 
\[k_{n+1} = \argmin{i \in \{1 , \cdots, W_n \}} \left( (X^{n+1}_i  |X^{n+1}_i \geq I^n_{k_n}) + I^{n+1}_i  \right), \]
and set $x_{n+1}$ to be child $k_{n+1}$ in the ordering. 
\end{itemize}
We shall now prove that the algorithm produces almost surely an exploding path whenever it does \emph{not} terminate in failure. We begin by making some preparations. First, note that
\[ (X^{n+1}_i  |X^{n+1}_i \geq I^n_{k_n}) = (X^{n+1}_i -I^{n}_{k_n}|X^{n+1}_i \geq I^n_{k_n}) + I^n_{k_n}, \]
which we compare to \eqref{eq::aging} by observing that without loss of generality (w.l.o.g.) we may assume $I(\delta)=1$. Indeed, define $I^*$ as 
\begin{equation}
 I^*(x) = \left\{ 
  \begin{array}{l l}
    I(x) & \quad \text{ if } x \in [0,\delta / 2],\\
    I(\delta / 2) + (x-\delta / 2)\frac{1-I(\delta / 2)}{\delta / 2} & \quad  \text{ if } x \in [\delta / 2,\delta], \\
    1 & \quad  \text{ if } x \in [\delta,\infty). \\
  \end{array} \right.
\end{equation}
Thus $I = I^*$ on $[0,\delta / 2]$, and $I^*(\delta) = 1$, hence, by comparison Theorem \ref{thm::IP_FOR_COMP_I}, $\HGIf$ is explosive if and only if $(h,G,I^*)_f$ is. 
Also, w.l.o.g. (to keep the calculations neat), we may assume that
\be \mathbb{P}(D \geq z) = \frac{1}{z^{\alpha}}, \label{eq::IP_COMPARISON_ASS_growth} \ee
for all sufficiently large $z$. Indeed, by Potter's bound, for all $\epsilon > 0$,
\be \mathbb{P}(D \geq z) = \frac{\widehat{L}(z)}{z^{\alpha}} \geq \frac{1}{z^{\alpha+\epsilon}}, \ee
for large $z$, where $\widehat{L}$ is a function varying slowly at infinity. We anticipate on the fact that the remainder of this proof is independent of the value of $\alpha \in (0,1)$, see below for further details. 
\\
The path produced by this algorithm has length $\mathcal{L}$, 
\[ \ba \mathcal{L} = \s{n=N}{\infty} X^{n+1}_{k_{n+1}} &= \s{n=N}{\infty} ( X^{n+1}_{k_{n+1}}|X^{n+1}_{k_{n+1}} \geq I^n_{k_n}) \\
&= I^N_{k_N} + \s{n=N}{\infty} \lr{ (X^{n+1}_{k_{n+1}} - I^n_{k_n} |X^{n+1}_{k_{n+1}} \geq I^n_{k_n}) + I^{n+1}_{k_{n+1}}} \\
&= I^{N}_{k_N} + \s{n=N}{\infty}  \minun{i \in \{ 1, \cdots, W_n \}}  \left( (X^{n+1}_i -I^{n}_{k_n}|X^{n+1}_i \geq I^n_{k_n}) + I^{n+1}_i  \right), \ea \]
we used that $X^{n+1}_{k_{n+1}} = ( X^{n+1}_{k_{n+1}}|X^{n+1}_{k_{n+1}} \geq I^n_{k_n})$ to establish the second equality, which we rewrote into the telescopic sum in the second line. Lemma \ref{lm::AGING_ST_DOM_SUM} shows that we thus have the following domination: 
\[ \mathcal{L}  \leqd I^{N}_{k_N} + \s{n=N}{\infty} \minun{i \in \{ 1, \cdots, W_n \}} \left( \widehat{X}^{n+1}_{i} +I^{n+1}_{i} \right). \]
Because the backward process is explosive, we know that 
\be \s{n=1}{\infty} \minun{i \in \{1 , \cdots, e^{1 / (\sqrt{\alpha})^n} \}} (X^{n}_i  |X^{n}_i \geq I^n_{i})  < \infty \mbox{ a.s.} \label{eq::IP_COMPARISON_BACK} ,\ee
and Lemma \ref{lm::FOR_EQUI_SUMMABILITY} gives that 
\[ \mathcal{L} \leqd I^{N}_{k_N} + \s{n=N}{\infty} \minun{i \in \{ 1, \cdots, e^{1 / (\sqrt{\alpha})^n} \}} \left( \widehat{X}^{n+1}_{i} +I^{n+1}_{i} \right) < \infty \mbox { a.s. ,} \]
 since \eqref{eq::IP_COMPARISON_m} - \eqref{eq::IP_COMPARISON_Wn} entail that 
\[ W_n \geq   \e^{1 / (\sqrt{\alpha})^n}, \quad \forall n \geq N, \]
hereby finishing the first part of the proof. Note that \eqref{eq::IP_COMPARISON_BACK} justifies the assumption \eqref{eq::IP_COMPARISON_ASS_growth}, as \eqref{eq::IP_COMPARISON_BACK} holds for all $\alpha \in (0,1)$ (recall that explosiveness of the backward process is independent of $\alpha \in (0,1)$).
\\
It remains to verify that with positive probability the algorithm does not terminate in failure. With non-zero probability the root gives birth to $f(N)$ children, hence we may set $x_N$ to be the root. That the remainder of the algorithm does not fail is the context of Lemma \ref{lm::IP_COMPARISON_An}.  
\end{proof}

\begin{lemma}
Let $\lr{X^{n}_i}_{i,n}$ and $\lr{\widehat{X}^{n}_i}_{i,n}$ be i.i.d. sequences with distribution $G$. Let $\lr{I^n_{i}}_{i,n}$ be i.i.d. according to some distribution $I$. 
Define, for $n\geq 0$ and $x \geq 0$,
\[Z^{n+1}_{i} \left( x \right) = \left( X^{n+1}_{i} - x  |X^{n+1}_{i} \geq  x \right)+ I^{n+1}_{i},  \] 
and,
\[k_{n+1} = \argmin{i \in \{ 1, \cdots, W_n \}} Z^{n+1}_{i} \left(I^{n}_{k_n} \right)  = \argmin{i \in \{ 1, \cdots, W_n \}}  \left( (X^{n+1}_i -I^{n}_{k_n}|X^{n+1}_i \geq I^n_{k_n}) + I^{n+1}_i  \right).\] 
Then, for $M \geq 0$,
\be S_M \left( I^0_{k_0}, \cdots, I^{M+1}_{k_{M+1}}  \right) := \s{n=0}{M} \minun{i \in \{ 1, \cdots, W_n \}} Z^{n+1}_{i}\left( I^n_{k_n} \right) \leqd \s{n=0}{M} \minun{i \in \{ 1, \cdots, W_n \}} \left( \widehat{X}^{n+1}_{i} +I^{n+1}_{i} \right). \label{eq::AGING_ST_DOM_SUM_eq1} \ee
Hence, for $N$ in the proof of Theorem \ref{thm::IP_COMPARISON_AGING},
\be \s{n=N}{\infty} \minun{i \in \{ 1, \cdots, W_n \}} Z^{n+1}_{i}\left( I^n_{k_n} \right) \leqd \s{n=N}{\infty} \minun{i \in \{ 1, \cdots, W_n \}} \left( \widehat{X}^{n+1}_{i} +I^{n+1}_{i} \right). \label{eq::AGING_ST_DOM_SUM_eq2} \ee
\label{lm::AGING_ST_DOM_SUM} 
\end{lemma}
\begin{remark}
Note that $k_n$ as it is defined here agrees with the definition of $k_n$ in the algorithm presented in the proof of Theorem  \ref{thm::IP_COMPARISON_AGING}.
\end{remark}

\begin{proof}[Proof of Lemma \ref{lm::AGING_ST_DOM_SUM}]
 We shall carry out an induction argument. We start with the case $M=0$. The variables in $\minun{i \in \{ 1, \cdots, W_0 \}} Z^{1}_{i}\left( I^0_{k_0} \right)$ are dependent, but they become independent as soon as we condition on the value of $I^0_{k_0} = x$. Due to assumption \eqref{eq::aging},
\begin{eqnarray*} 
\P{\minun{i \in \{ 1, \cdots, W_0 \}} Z^{1}_{i}\left( I^0_{k_0} \right) \leq t} &=& \int_0^{\delta} \P{\minun{i \in \{ 1, \cdots, W_0 \}} Z^{1}_{i}\left( x \right) \leq t} \mathrm d\P{I^0_{k_0} \leq x} \\
&\geq& \int_0^{\delta} \P{\minun{i \in \{ 1, \cdots, W_0 \}} \left( \widehat{X}^{1}_{i} +I^{1}_{i} \right) \leq t} \mathrm  d\P{I^0_{k_0} \leq x} \\
&=& \P{\minun{i \in \{ 1, \cdots, W_0 \}} \left( \widehat{X}^{1}_{i} +I^{1}_{i} \right) \leq t},
\end{eqnarray*}
recall that we justified the assumption $I(\delta) = 1$. We proceed the induction proof by assuming that the assertion holds for some $M \geq 0$. Consider
\begin{equation} \P{\s{n=0}{M+1} \minun{i \in \{ 1, \cdots, W_n \}} Z^{n+1}_{i}\left( I^n_{k_n} \right) \leq t}. 
\label{eq::AGING_ST_B1}
\end{equation}
Unfortunately, the individual terms are not independent. However, we may bypass this complication by conditioning on the value of $I^{M+1}_{k_{M+1}} = x$, that is, we investigate
\begin{equation}
 \P{ \left. S_M \left( I^0_{k_0}, \cdots, I^{M}_{k_{M}},I^{M+1}_{k_{M+1}} = x  \right) + \minun{i \in \{ 1, \cdots, W_{N+1} \}} Z^{M+2}_{i}\left( x \right) \leq t \  \right| I^{M+1}_{k_{M+1}} = x }. 
\label{eq::AGING_ST_B2}
\end{equation}
Now, again by assumption \eqref{eq::aging}, for all $x \leq \delta$,
\[\minun{i \in \{ 1, \cdots, W_{M+1} \}} Z^{M+2}_{i}\left( x \right) \leqd \minun{i \in \{ 1, \cdots, W_{M+1} \}} \left( \widehat{X}^{M+2}_{i} +I^{M+2}_{i} \right),\] and, moreover, $\minun{i \in \{ 1, \cdots, W_{M+1} \}} Z^{M+2}_{i}\left( x \right)$ is independent of $S_M \left( I^0_{k_0}, \cdots, I^{M}_{k_{M}},x  \right)$. Hence, the probability in (\ref{eq::AGING_ST_B2}) is larger than or equal to
\[ \P{ \left. S_M \left( I^0_{k_0}, \cdots, I^{M}_{k_{M}},x  \right) + \minun{i \in \{ 1, \cdots, W_{N+1} \}} \left( \widehat{X}^{M+2}_{i} +I^{M+2}_{i} \right) \leq t \ \right| I^{M+1}_{k_{M+1}} = x }.  \]
This result holds for all $x \leq \delta$, thus (\ref{eq::AGING_ST_B1}) is larger than or equal to
\[ \P{S_M \left( I^0_{k_0}, \cdots, I^{M+1}_{k_{M+1}}  \right) + \minun{i \in \{ 1, \cdots, W_{N+1} \}} \left( \widehat{X}^{M+2}_{i} +I^{M+2}_{i} \right) \leq t }.  \]
The important observation here is that the sum $S_M$ and the minimum are independent, so that by the induction hypothesis this probability is larger than or equal to
\begin{multline*} \P{\s{n=0}{M} \minun{i \in \{ 1, \cdots, W_n \}} \left( \widehat{X}^{n+1}_{i} +I^{n+1}_{i} \right) + \minun{i \in \{ 1, \cdots, W_{N+1} \}} \left( \widehat{X}^{M+2}_{i} +I^{M+2}_{i} \right) \leq t }, 
\end{multline*} 
which equals
\[ \P{\s{n=0}{M+1} \minun{i \in \{ 1, \cdots, W_n \}} \left(\widehat{X}^{n+1}_{i} +I^{n+1}_{i} \right) \leq t }. \]
Hence, we established \eqref{eq::AGING_ST_DOM_SUM_eq1}. 
To verify \eqref{eq::AGING_ST_DOM_SUM_eq2}, we observe that, for all $M \geq 0$ and $t \geq 0$,
\be \P{\s{n=N}{N+M} \minun{i \in \{ 1, \cdots, W_n \}} Z^{n+1}_{i}\left( I^n_{k_n} \right) \leq t} \geq \P{\s{n=N}{N+M} \minun{i \in \{ 1, \cdots, W_n \}} \left(\widehat{X}^{n+1}_{i} +I^{n+1}_{i} \right) \leq t },  
\label{eq:AGING_ST_DOM_SUM_eq3}
\ee
which follows after realizing that the proof carried out above remains valid if we start with $n=N$ instead of $n=0$. The events inside the probabilities on both sides of \eqref{eq:AGING_ST_DOM_SUM_eq3} constitute a nested sequence and by the continuity of the measure $\P{\cdot}$ we conclude that the inequality remains to hold when $M \to \infty$, hereby establishing the result. 
\end{proof} 

\begin{lemma}
Let $\lr{X^{n}_i}_{i,n}$ and $\lr{\widehat{X}^{n}_i}_{i,n}$ be i.i.d. sequences with distribution $G$. Let $\lr{I^n_{i}}_{i,n}$ be i.i.d. according to some distribution $I$. 
Assume that
\[ \s{n=1}{\infty} \minun{i \in \{1 , \cdots, e^{1 / (\sqrt{\alpha})^n} \}} (X^{n}_i  |X^{n}_i \geq I^n_{i})  < \infty \mbox{ a.s.} .\]
Then, 
\[ \s{n=1}{\infty} \minun{i \in \{1 , \cdots, e^{1 / (\sqrt{\alpha})^n} \}} \left(  \widehat{X}^{n}_i + I^{n}_i  \right) < \infty  \mbox{ a.s.}.\]
\label{lm::FOR_EQUI_SUMMABILITY}
\end{lemma}
\begin{proof}
It is easily verified that $(X_i^n|X_i^n \geq I_i^n) \geqd \frac{\widehat{X}_i^n +I_i^n }{2}$. Indeed, 
\[ \P{2 (X_i^n|X_i^n \geq I_i^n) \leq t} \leq \P{I_i^n + (X_i^n|X_i^n \geq I_i^n) \leq t} = \int_0^{\infty} \P{x + (X_i^n|X_i^n \geq x) \leq t} \mathrm d I(x). \]
Since $(X_i^n|X_i^n \geq x)$ stochastically dominates $\widehat{X}_i^n$ for all $x$, the last integral is less than or equal to
\[ \int_0^{\infty} \P{x + \widehat{X}_i^n \leq t} \mathrm d I(x) = \P{I_i^n+\widehat{X}_i^n \leq t}. \]
 Thus, for all $M$,
 \[  \P{ 2 \s{n=1}{M} \minun{i \in \{1 , \cdots, e^{1 / (\sqrt{\alpha})^n} \}}  (X^{n}_i  |X^{n}_i \geq I^n_{i})   \leq t} \leq \P{  \s{n=1}{M} \minun{i \in \{1 , \cdots, e^{1 / (\sqrt{\alpha})^n} \}} \left(  \widehat{X}^{n}_i + I^{n}_i  \right) \leq t}. \]
Now, call the events inside the left probability $E_M$ and notice that $E_{k+1} \subset E_k$ for all $k \geq 1$ (a similar observation holds on the right hand side). We may thus conclude, by continuity of $\P{\cdot}$, that the inequality remains to hold if we let $M \to \infty$. In other words,
\[  \s{n=1}{\infty} \minun{i \in \{1 , \cdots, e^{1 / (\sqrt{\alpha})^n} \}} \left(  \widehat{X}^{n}_i + I^{n}_i  \right) \leqd 2 \s{n=1}{\infty} \minun{i \in \{1 , \cdots, e^{1 / (\sqrt{\alpha})^n} \}} \left( X^{n}_i  |X^{n}_i \geq I^n_{i}  \right). \]
\end{proof}

\begin{lemma}
Consider the events and definitions in the proof of Theorem  \ref{thm::IP_COMPARISON_AGING}. Define 
\[ A_N = \{ D_N \geq f(N) \}, \]
and, for $n \geq N+1$, 
\[A_n = \{ D_n \geq f(n) , D_{n-1}^{e} \geq c f(n-1) \}.\]
Then, for $n \geq N$,
 \[ \ba \mathbb{P}(A_{n+1}|A_n) &\geq 
\left( 1 - \exp \lr{ -\frac{f(n) \mathbb{P}(X \geq \delta)  }{8}} \right)
\left( 1 - \exp \lr{ - \frac{c f(n)^{1 - \sqrt{\alpha}}}{8}} \right)\\
&= \left( 1 - \exp \lr{-\frac{c f(n)   }{4}} \right)
\left( 1 - \exp \lr{- \frac{c f(n)^{1 - \sqrt{\alpha}}}{8}} \right). \ea \]
\label{lm::IP_COMPARISON_events}
\end{lemma}
\begin{proof}
We start by writing the probability
\[ \ba
\mathbb{P}(A_{n+1}|A_n) &= \P{ D_{n+1} \geq f(n+1), D_{n}^{e} \geq c f(n) | D_n \geq f(n) , D_{n-1}^{e} \geq c f(n-1) } \\
&= \P{D_{n+1} \geq f(n+1), D_{n}^{e} \geq c f(n) | D_n \geq f(n)}
\ea \]
in a more revealing way:
\be \mathbb{P}(A_{n+1}|A_n) = \P { D_{n}^{e}\! \geq  \!c f(n)  |  D_n \! \geq \! f(n) } 
\cdot \P { D_{n+1} \! \geq \! f(n+1)  | D_{n}^{e} \! \geq c \! f(n),D_n \! \geq \! f(n) } . \label{eq::IP_COMPARISON_events_1} \ee
Let us focus on the first probability in the product on the right of \eqref{eq::IP_COMPARISON_events_1}. As $I^n_{k_n} \leq \delta$ (recall $I(\delta) = 1$),
\[ \Dne = \s{i=1}{\Dn}  \indicator{X_i^{n+1}  \geq I^n_{k_n}} \geq \s{i=1}{\Dn} \indicator{X_i^{n+1}\geq \delta}. \]
Since $D_n \geq f(n)$, we have
\[ \Dne \geq \s{i=1}{f(n)} \indicator{X_i^{n+1}  \geq \delta} \equalsd  \mbox{ Bin}(f(n),\P{X \geq \delta}).\] Applying large deviation theory on the last random variable gives
\[ \ba \P{ \Dne < cf(n)| D_n \geq f(n)} &\leq \P{\mbox{Bin}(f(n),\P{X \geq \delta}) \leq \frac{\P{X \geq \delta}f(n)}{2}} \\ &\leq \exp \lr{-\frac{\P{X \geq \delta}f(n)}{8}}, \ea \]
since $c =  \P{X \geq \delta} /2$.
We bound the second probability in the right hand side of \eqref{eq::IP_COMPARISON_events_1} from below by
\[1 - \P{D_{n+1} < f(n+1)| D_{n}^{e} \geq c f(n),D_n \geq f(n) }. \] 
Consider the amount of good children $G_n$, where a child of $x_n$ is called good if it has at least $f(n+1)$ children. $D_{n+1}< f(n+1)$ only if $G_n < W_n$. We will show that this is very unlikely to happen. First note that $G_n \overset{d}{=}$ Bin($D^e_n,p$), where $p = 1 - F_D(f(n+1)) = f(n)^{-\sqrt{\alpha}}$. Next, \[ W_n = \frac{c f(n)^{1-\sqrt{\alpha}}}{2}  \leq \frac{\mathbb{E}G_n}{2}.\]
Hence, 
\[ \ba \P{D_{n+1} < f(n+1) | D_{n}^{e} \geq c f(n),D_n \geq f(n)  } &\leq \mathbb{P}\left( \left. G_n < \frac{\mathbb{E}G_n}{2} \right| D_{n}^{e} \geq c f(n) \right)\\
& \leq \exp \lr{- \frac{c f(n)^{1 - \sqrt{\alpha}}}{8}}. \ea \]
\end{proof}

\begin{lemma}
Consider the events in Lemma \ref{lm::IP_COMPARISON_events}, together with the definitions of $N$ and $m$ in the proof of Theorem  \ref{thm::IP_COMPARISON_AGING}.
For $N$ as in \eqref{lm::IP_COMPARISON_events} we have
\[\mathbb{P}\left( A_{N+1},A_{N+2}, \cdots|A_N \right) > 0.\]
\label{lm::IP_COMPARISON_An}
\end{lemma}
\begin{proof}
Due to the choice of $m$, 
\[ \frac{1}{8} cf(n)^{1 - \sqrt{\alpha}} \geq e^{1 / (\sqrt{\alpha})^n} .\]
Hence, we observe that
\[ \mathbb{P}\left( A_{n+1}|A_n \right)  \geq 1 - 2 \exp ( - \exp (1 / (\sqrt{\alpha})^n) ) \geq 1 - \frac{1}{n^2}, \]
for $n \geq N$. Now, write for $k \geq 1$,
\[ \mathbb{P}\left( A_{N+k},A_{N+k-1}, \cdots, A_{N+1}|A_N \right) \]
as 
\[\mathbb{P}\left( A_{N+k}|A_{N+k-1}, \cdots, A_N \right) \cdot \mathbb{P}\left( A_{N+k-1}|A_{N+k-2}, \cdots, A_N \right) \cdots \mathbb{P}\left( A_{N+1}|A_N \right).\]
We employ this observation by noting that, for $i \leq j$,
\[ \mathbb{P}\left( A_{j}|A_{j-1}, A_{j-2}, \cdots, A_{j-i} \right) = \mathbb{P}\left( A_{j}|A_{j-1}\right), \] 
 so that,
\[ \mathbb{P}\left( A_{N+k},A_{N+k-1}, \cdots, A_{N+1}|A_N \right) = \prod_{j = N}^{N +k -1} \mathbb{P}\left( A_{j+1}|A_j \right). \]
The assertion is now established by letting $k$ go to infinity, since $\prod_{j = N}^{\infty} \left( 1 - \frac{1}{n^2} \right) > 0$.
\end{proof}

\noindent We now remark that the condition $(X-x|X \geq x) \overset{d}{\leq} X$ is met in most cases of interest. I.e., the explosiveness follows if either $g(0) > 0$ or $g(0) =0$ and $g$ is non-decreasing on some non-empty interval $[0,\delta]$:
\\

\begin{proof}[Proof of Theorem \ref{thm::IP_COMPARISON_MAIN}]
On the one hand, if $g(0) =0$ and $g$ is locally increasing, then we may compare the process $\HaLGIf$ to  $(h _{\alpha},G^*,I)_f$, where $G^*$ defined below is seen to satisfy the assumption \eqref{eq::aging}.
Let $T>0$ be so that $g$ increases on $[0,T]$, and define $g^*$,  for $x \geq 0$, by
\begin{equation}
 g^*(x) = \left\{ 
  \begin{array}{l l}
    g(x) & \quad \text{ if } x \in [0,T],\\
    g(T) & \quad  \text{ if } x \in [T,T^*], \\
    0 & \quad  \text{ if } x \in [T^*,\infty), \\
  \end{array} \right.
\end{equation}
where $T^*$ is chosen so that $\int_0^{\infty}  g^*(x) \mathrm dx = 1.$
If $X^*$ is a random variable that has $ g^*$ as its density, then, for $x \leq \frac{T}{2}$,
\[ \mathbb{P}\left( X^*-x \leq t|X^* \geq x \right) \geq \P{X^* \leq t}, \]
for all $t \geq 0$. Indeed, if $t + x \leq T^*$, then
\[ \ba 
\mathbb{P}\left( X^*-x \leq t|X^* \geq x \right) &= \frac{1}{1 - G^*(x)} \int_0^{t}  g^*(u+x) \mathrm du  \\
& \geq \frac{1}{1 - G^*(x)} \int_0^{t}  g^*(u) \mathrm du  \\
& \geq G^*(t),
\ea \]
since $g^*$ is non-decreasing on $[0,T^*]$.
\\
If $t + x \geq T^*$, then
\[ \ba 
\mathbb{P}\left( X^*-x \leq t|X^* \geq x \right) &= \frac{1}{1 - G^*(x)} \int_0^{t}  g^*(u+x) \mathrm du  \\
& = \frac{1}{1 - G^*(x)} \lr{1 - G^*(x)}  \\
& =1.
\ea \]
\\
On the other hand, if $g \geq \epsilon > 0$ on $[0, \delta]$, then we may compare the process $\HaLGIf$  to a Markov process which meets assumption \eqref{eq::aging} with equality.
\end{proof}

\newpage

\bibliographystyle{plain}
 \bibliography{literature}

\end{document}